\documentclass{jams-l}

\usepackage{amssymb}
\usepackage[makeroom]{cancel}
\usepackage{graphicx}
\usepackage[cmtip,all]{xy}
\usepackage[usenames,dvipsnames]{color}
\usepackage{upgreek}
\usepackage{multirow}
\usepackage{hyperref}
\definecolor{linkcolour}{rgb}{0,0.2,0.6}
\hypersetup{colorlinks,breaklinks,urlcolor=linkcolour, linkcolor=linkcolour}
\usepackage{multicol}
\usepackage{colortbl}
\definecolor{bobby}{gray}{0}
\definecolor{LightCyan}{rgb}{0.88,1,1}
\usepackage{arydshln}
\usepackage{soul}

\newtheorem{theorem}{Theorem}[section]
\newtheorem{lemma}[theorem]{Lemma}
\newtheorem{proposition}[theorem]{Proposition}

\theoremstyle{definition}
\newtheorem{definition}[theorem]{Definition}
\newtheorem{convention}[theorem]{Convention}

\newtheorem{example}[theorem]{Example}

\theoremstyle{remark}
\newtheorem{remark}[theorem]{Remark}

\numberwithin{equation}{section}

\newcommand{\separ}{\!\!\;\!\!\;}
\newcommand\dcap{\mathrel{\ooalign{\rotatebox[origin=c]{-90}{$\longrightarrow$}\cr\kern0.4ex\hbox{$\not$}}}}

\usepackage{geometry}
\begin{document}

\large

\title{Category theory for genetics I:\\mutations and sequence alignments}

\author{R\'{e}my Tuy\'{e}ras}
\address{M.I.T., Department of Mathematics, 77 Massachusetts avenue, Cambridge, MA 02139}
\curraddr{}
\email{rtuyeras@mit.edu}
\thanks{This research was supported by the AFOSR grant, \emph{Categorical approach to agent interaction}, FA9550-14-1-0031 and the AFOSR grant, \emph{Pixel matrices and other compositional analyses of interconnected systems}, FA9550-17-1-0058.}

\date{}

\dedicatory{}

\begin{abstract}
The present article is the first of a series whose goal is to define a logical formalism in which it is possible to reason about genetics. In this paper, we introduce the main concepts of our language whose domain of discourse consists of a class of limit-sketches and their associated models. While our program will aim to show that different phenomena of genetics can be modeled by changing the category in which the models take their values, in this paper, we study models in the category of sets to capture mutation mechanisms such as insertions, deletions, substitutions, duplications and inversions. We show how the proposed formalism can be used for constructing multiple sequence alignments with an emphasis on mutation mechanisms.
\end{abstract}

\maketitle

\section{Introduction}

\subsection{Short presentation}
The goal of the present article is to define a type of algebraic structures in which it is possible to \emph{do genetics}. The main operation provided by these structures is a formal way of `gluing' different pieces of information together. To motivate the various notions introduced in this paper, we focus on a particular example, namely the construction of multiple sequence alignments, which are often used in phylogenetics. The proposed formalism can be applied to many other situations that would more generally look at sequential polymer comparisons and/or interactions.

\subsection{Motivations}
Our objective is to construct a bridge between two completely disconnected domains of science, specifically genetics and category theory, through a series of papers. While genetics is well-known for its complexity, category theory is recognized for its clarity and expressive power \cite{SpivakBook,BrownPorter,MacLane}. The goal of the present program would be to reach a level of abstraction that would allow one to tackle questions whose formulation are too complicated to be addressed with the current tools.

The language of the present paper is rather mathematical, but the results and definitions that it contains always try to capture the biological reality. Throughout the paper, some terms might be used in a biological sense while others might be used in a mathematical one -- this will usually be specified. For instance, the sentence ``a structure in which it is possible to do genetics'' means that we want to define a formal language rather than a model of some particular living body. The need for such an abstraction, in biology, has, for example, been recognized in \cite{Lazebnik,Servedia14}. 

Attempts at linking genetics (or in fact molecular biology) to categorical thinking are not new. A first example is \cite{Japanese_work}, in which a category-like formalism is used to discuss the algebraic properties of ``DNA wallpapers''. Another work is \cite{slice_bio}, in which Carbone \& Gromov model DNA, RNA and proteins by using topological and geometrical objects such as surfaces and moduli spaces. The program proposed herein tries to understand the mechanisms of genetics in themselves by forgetting the spacial aspect and focusing on the biological operations occurring in the body. Such an algebraic approach to biology has already been discussed, from the point of view of neuroscience, in several unpublished works by Ehresmann (for example, see \cite{A_Ehresmann}) using the concepts of \emph{limit} and \emph{cone}. The present paper takes a step further, in the context of genetics, by providing a precise `limit theory' (in fact, a \emph{limit sketch}) that can be used to formalize precise concepts of genetics.
In this respect, our structures will define formal environments in which one wants to express a problem and say things about its solution.

In addition to offering a formalism, the proposed program aims to tackle technical and/or conceptual problems of various sub-fields of genetics. While the next article \cite{Recomb} will focus on questions related to genotypes, phenotypes, haplotypes, homologous recombination, and genetic linkage, the present article focuses on questions related to the construction of multiple sequence alignments \cite{Rosenberg,Mount,Daugelaite} from a mechanistic point of view (see section \ref{ssec:Main_example} for an introduction).

Specifically, the present work is an attempt to give a categorical answer to the program proposed in \cite{MorrisonFramework,MorrisonWhy} regarding the construction of multiple sequence alignments by trying to ``[recognize] mechanisms rather than assuming that all the variation occurs at random [at every position in the DNA strand]'' \cite[page 156, right col., l. 5]{MorrisonFramework}. More precisely, our goal is to show that the language of category theory can be used to put more emphasis on evolutionary mechanisms so that ``[mutation] events [can] be identified as the alignment proceeds rather than being identified after the alignment is completed''\cite[page 156, right col., l. 9]{MorrisonFramework}.

Finally, the reader can find a Python library that aims to implement the content of the whole program at the web address: \url{https://github.com/remytuyeras/pedigrad-library}. The library will be updated as the program evolves towards more tools.

\subsection{Road map and results}
The goal of the present paper is to define a class of theories, called \emph{chromologies}, whose models, called \emph{pedigrads}, will be shown (through the program) to recover various aspects of genetics by changing the associated categories of values. 

We begin by defining chromologies in sections \ref{ssec:pre-ordered_sets} through section \ref{ssec:chromologies}, while the pedigrads for these theories will be defined in sections \ref{ssec:Logical_systems} \& \ref{ssec:Pedigrads}. Intuitively, chromologies allow us to do all sorts of basic DNA manipulations such as sequence alignments, CRISPR \cite{Pennisi} and homologous recombination, whereas the pedigrads allow us to give a context to these operations (which can be handled differently depending on the environment in which they are processed).

In section \ref{sec:Examples_of_pedigrads_in_sets}, we define a class of functors (Definition \ref{def:set_E_b_varepsilon}) that model the environments of DNA sequences (Example \ref{exa:elements_as_words}). We show that these functors can be endowed with two types of pedigrad structures in the category of sets (Theorems \ref{theo:E_b_varepsilon_W_iso_pedigrad_exactly_distributive} \& \ref{theo:E_b_varepsilon_W_surj_pedigrad_injective}): the first type detects the exact consistency of the data while the second type detects the consistency of the data up to uncertainty (see Example \ref{exa:right_kan_extension_data_consistency}). These functors are then used to formalize the concept of sequence alignment in terms of a functor (Definition \ref{def:sequence_alignment}). We will see that taking the right Kan extension of a sequence alignment functor can be viewed as constructing multiple sequence alignments (Remark \ref{rem:RKE_gluing_tables}). In Example \ref{exa:Global-local_seq_alignments}, we will see that the right Kan extension of a sequence alignment functor contains both local and global pieces of information that inform us of the presence of uncertainties in the integrated data. At the end of the section, we will see in this uncertainty a justification for the concepts of chromology and pedigrad, which will give us ways to locate, specify or isolate the existing uncertainties (see Example \ref{exa:uncertainty_vs_chromologies}).

In section \ref{sec:Solving_main_example}, we will introduce the concept of a slice of a sequence alignment functor (Definition \ref{def:Slices}), which will allow us to resolve the mentioned uncertainties and select multiple sequence alignments that are consistent with the overall data. First, in section \ref{sec:slices_seq_alignment}, we will discuss the selection of consistent sequence alignments through the use of chromologies. Then, in Remark \ref{rem:finding_right_seq_alignment}, we will suggest an algorithm for constructing multiple sequence alignments via the use of slices. Finally, in Section \ref{ssec:slices_and_mechanisms}, we will show that the presence of uncertainties in the right Kan extension of a sequence alignment functor can be due to mutation mechanisms. We will show that the resolution of these uncertainties, permitted by slices, gives us a way to recognize mutation mechanisms (Examples \ref{exa:slices_and_mechanisms_duplications} \& \ref{exa:slices_and_mechanisms_inversions}).

\subsection{Acknowledgments}
I would like to thank the referee for their very useful comments and remarks, which led to a significant improvement of an earlier version of this paper.
I would also like to thank Brendan Fong, David Spivak and Eric Neumann for useful discussions. Finally, I would like to thank Andres Saez and Anjanet Loon for their careful reading of this manuscript.

\section{Chromologies and Pedigrads}\label{sec:Chromologies_and_Pedigrads}

The goal of this section is to introduce a set of theories whose logical models try to capture the logic of genetics. To justify why our theories look the way they do, we need to recall a few facts regarding the construction of theories in general. First, recall that, classically, models for theories are defined as sets equipped with some operations. For instance, a \emph{ring} is a set $R$ equipped with two operations $\cdot:R \times R \to R$ and $+:R \times R \to R$ making certain diagrams commute.

More categorically, rings are also product-preserving functors from a certain  product sketch\footnote{A small category equipped with a subset of its wide spans (see Definition \ref{def:wide_spans}).} $\mathtt{Ring}$ (the theory) to the category $\mathbf{Set}$ of sets and functions \cite{Ehresmann}. This functorial point of view was introduced by Lawvere \cite{Lawvere1963} in 1963 via the concept of what is now called a \emph{Lawvere theory} -- the theory $\mathtt{Ring}$ being an example. The advantage of functors over sets equipped with functions is that functors allow us to clearly distinguish between what is intrinsically true in a model (via the theory) and what can occasionally be true in the model (via the images of the functor). Then, the formalism accompanying the language of functors allows us to more carefully think about the mechanisms governing the models.

Since Lawvere theories were meant to capture the logic of algebraic structures equipped with multivariate functions, their objects were taken to be the set of natural numbers in order to specify the arities of the functions. Along those lines, since the goal of the present section is to define a theory that captures the logic of genetics and whose operations take DNA segments as inputs, the objects of our theory will look like DNA segments. 
Note that, while, in rings, one \emph{adds} and \emph{multiplies} terms together, in genetics, one \emph{cuts}, \emph{aligns} and \emph{recombines} DNA strands together. Therefore, our theory will be based on these operations.

More specifically, recall that an integer object in a Lawvere theory can be represented as a finite sequence of atoms; e.g the object 6 would be represented by six atoms as follows.
\begin{equation}\label{eq:Lawvere_integers}
6 = \xymatrix@C-30pt{(\bullet&\bullet&\bullet&\bullet&\bullet&\bullet)}
\end{equation}
These atoms can make it easier to see how the models defined on the Lawvere theory send the integer objects to the product objects in the category of values; \textit{e.g.} for a given functor $R$, the image $R(6)$ would be sent to a product of the following form where $R(\bullet) = R(1)$.
\[
R(\bullet) \times R(\bullet) \times R(\bullet) \times R(\bullet) \times R(\bullet) \times R(\bullet)
\]
In the case of DNA, the idea is to copy the previous picture, but by adding enough information to be able to model genetic mechanisms. If one looks at the type of pictures drawn by biologists to explain homologous recombination, alignment methods or even genetic linkage, one can often see pictures of chromosomal patches subdivided in terms of selected and masked regions, as shown below.
\[
\includegraphics[height=5cm]{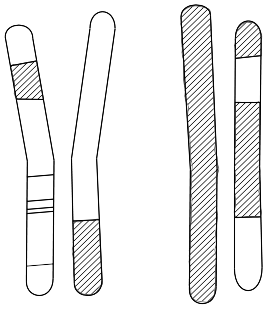}
\]
These colored regions are obviously reminiscent of the term \emph{chromo-some}\footnote{meaning \emph{color-body}} itself. The regional separations are also reminiscent of some sort of topology -- or metric. If one tries to merge these topological and colored components with the type of atomic representation given in (\ref{eq:Lawvere_integers}), we are likely to end up with the following type of pictures.
\begin{equation}\label{eq:intro_representation_patch}
\xymatrix@C-30pt{
(\bullet&\bullet&\bullet)&(\circ&\circ)&(\bullet&\bullet&\bullet&\bullet)&(\bullet&\bullet&\bullet&\bullet&\bullet)&(\circ&\circ&\circ)&(\circ)
}
\end{equation}
In picture (\ref{eq:intro_representation_patch}), the black nodes could indicate the regions of the chromosome that one wants to use while the white nodes could indicate the parts of the chromosome that one wants to ignore (mask). Note that a black-and-white paper does not give more than two colors to color the previous type of objects. Therefore, we will not hesitate to use labels to represent new colors. 
Formally, our sets of colors will be encoded by pre-ordered sets, whose semantics will allow us, among other operations, to  \emph{select} and \emph{cut}.

\subsection{Main example}\label{ssec:Main_example}
To help the exposition of the present paper, most of our examples will focus on a single problem, which will give a story to our demonstration. Each example will illustrate how the mathematical definitions given in this paper can be used to clarify, explain or solve precise aspects of our problem. 

Without any further introduction, our problem will look at the alleles of four different individuals for the same gene. Our goal will be to show how one can use chromologies and their pedigrads to help us relate these four individuals. In general, the first step to establishing the genealogy of a set of individuals is to align their genetic data according to an evolution model (\emph{e.g.} using substitution, insertion and deletion mutations) in order to determine how one individual evolved from another one (see \cite{Rosenberg,Mount,Daugelaite}). 

In our case, we will consider the set of individuals given in the following table with respect to the corresponding genetic data shown on the right.
\begin{center}
\begin{tabular}{l|l}
\hline
\multicolumn{1}{c|}{\cellcolor[gray]{0.8}Individuals} & \multicolumn{1}{c}{\cellcolor[gray]{0.8}Alleles} \\
\hline
\texttt{Anne} & $\mathtt{ACCGACTG}$ \\
\texttt{Bob} & $\mathtt{ACATCTG}$ \\
\texttt{Craig} & $\mathtt{ACCGTCA}$ \\
\texttt{Doug} & $\mathtt{ACTACTG}$ \\
\end{tabular}
\end{center}

Recall that, in bioinformatics, the only analytic method that can compare and align a set of DNA sequences is the \emph{dynamic programming algorithm} (see \cite{Simossis} or \cite[Chapter 3 and page 71]{Mount}). Other methods obviously exist, but these require heuristics that generally only produce an approximation of the best alignment \cite{Daugelaite,Feng,Chowdhury,Loytynoja}. A common aspect between these methods is that they all compare DNA sequences two by two. Many of them even rely on the pairwise dynamic programming algorithm, which can compare any given pair of sequences together. For instance, to compare the genetic data of two individuals, say \texttt{Anne} and \texttt{Bob}, via the dynamic programming algorithm \cite[page 71]{Mount}, we would first draw a table as given below, on the left-hand side, where second topmost row and second leftmost column are initialized with canonical scores (called \emph{gap penalities}).

\[
\begin{array}{|c|c|c|c|c|c|c|c|c|c|}
\hline
&\cellcolor[gray]{0.8}\varepsilon&\cellcolor[gray]{0.8}\mathtt{A}&\cellcolor[gray]{0.8}\mathtt{C}&\cellcolor[gray]{0.8}\mathtt{C}&\cellcolor[gray]{0.8}\mathtt{G}&\cellcolor[gray]{0.8}\mathtt{A}&\cellcolor[gray]{0.8}\mathtt{C}&\cellcolor[gray]{0.8}\mathtt{T}&\cellcolor[gray]{0.8}\mathtt{G}\\
\hline
\cellcolor[gray]{0.8}\varepsilon&0&1&2&3&4&5&6&7&8\\
\hline
\cellcolor[gray]{0.8}\mathtt{A}&1&&&&&&&&\\
\hline
\cellcolor[gray]{0.8}\mathtt{C}&2&&&&&&&&\\
\hline
\cellcolor[gray]{0.8}\mathtt{A}&3&&&&&&&&\\
\hline
\cellcolor[gray]{0.8}\mathtt{T}&4&&&&&&&&\\
\hline
\cellcolor[gray]{0.8}\mathtt{C}&5&&&&&&&&\\
\hline
\cellcolor[gray]{0.8}\mathtt{T}&6&&&&&&&&\\
\hline
\cellcolor[gray]{0.8}\mathtt{G}&7&&&&&&&&\\
\hline
\end{array}
\quad\quad
\Rightarrow
\quad\quad
\begin{array}{|c|c|c|c|c|c|c|c|c|c|}
\hline
&\cellcolor[gray]{0.8}\varepsilon&\cellcolor[gray]{0.8}\mathtt{A}&\cellcolor[gray]{0.8}\mathtt{C}&\cellcolor[gray]{0.8}\mathtt{C}&\cellcolor[gray]{0.8}\mathtt{G}&\cellcolor[gray]{0.8}\mathtt{A}&\cellcolor[gray]{0.8}\mathtt{C}&\cellcolor[gray]{0.8}\mathtt{T}&\cellcolor[gray]{0.8}\mathtt{G}\\
\hline
\cellcolor[gray]{0.8}\varepsilon&0&1&2&3&4&5&6&7&8\\
\hline
\cellcolor[gray]{0.8}\mathtt{A}&1&0&1&2&3&4&5&6&7\\
\hline
\cellcolor[gray]{0.8}\mathtt{C}&2&1&0&1&2&3&4&5&6\\
\hline
\cellcolor[gray]{0.8}\mathtt{A}&3&2&1&1&2&2&3&4&5\\
\hline
\cellcolor[gray]{0.8}\mathtt{T}&4&3&2&2&2&3&3&3&4\\
\hline
\cellcolor[gray]{0.8}\mathtt{C}&5&4&3&2&3&3&3&4&4\\
\hline
\cellcolor[gray]{0.8}\mathtt{T}&6&5&4&3&3&4&4&3&4\\
\hline
\cellcolor[gray]{0.8}\mathtt{G}&7&6&5&4&3&4&5&4&3\\
\hline
\end{array}
\]
Then, we would produce the table given on the right-hand side by following two types of scoring rules, one taking care of matches and the other one taking care of mismatches. Usually, these rules tell us how to fill out a box in the case where we have a two-by-two matrix whose bottom-right corner is empty and whose remaining boxes are already filled with scores, as shown below.
\[
\begin{array}{|c|c|}
\hline
p&q\\
\hline
r&\\
\hline
\end{array}
\]
In our case, the scoring table was filled with the following rules:
\begin{itemize}
\item[1)] if the nucleotides labeling the column and the row of the empty box are equal, then the empty box should be filled with the score $p$;
\item[2)] if the nucleotides labeling the column and the row of the empty box are different, then the empty box should be filled with the score $\mathsf{min}(p,q,r)+1$;
\end{itemize}

The best alignments for \texttt{Anne} and \texttt{Bob} are then obtained by tracing back the previous rules from the bottom-right corner of the table to its top-left corner. All the paths of moves (from right to left) that would make the earlier rules hold describe the nucleotide comparisons that give rise to the best alignments. In the present case, there are more than one paths. One of them is shown in the table given below, on the left (starting from the bottom-right corner). The associated alignment, given on the right, is read from the left-top corner to right-bottom one, where every symbol $\varepsilon$ represents a stationary move from the point of view of the sequences.
\[
\begin{array}{|c|c|c|c|c|c|c|c|c|c|}
\hline
&\cellcolor[gray]{0.8}\varepsilon&\cellcolor[gray]{0.8}\mathtt{A}&\cellcolor[gray]{0.8}\mathtt{C}&\cellcolor[gray]{0.8}\mathtt{C}&\cellcolor[gray]{0.8}\mathtt{G}&\cellcolor[gray]{0.8}\mathtt{A}&\cellcolor[gray]{0.8}\mathtt{C}&\cellcolor[gray]{0.8}\mathtt{T}&\cellcolor[gray]{0.8}\mathtt{G}\\
\hline
\cellcolor[gray]{0.8}\varepsilon&\cellcolor[gray]{0.85}\textbf{0}&1&2&3&4&5&6&7&8\\
\hline
\cellcolor[gray]{0.8}\mathtt{A}&1&\cellcolor[gray]{0.85}\textbf{0}&\cellcolor[gray]{0.85}\textbf{1}&2&3&4&5&6&7\\
\hline
\cellcolor[gray]{0.8}\mathtt{C}&2&1&0&\cellcolor[gray]{0.85}\textbf{1}&\cellcolor[gray]{0.85}\textbf{2}&3&4&5&6\\
\hline
\cellcolor[gray]{0.8}\mathtt{A}&3&2&1&1&\textbf{2}&\cellcolor[gray]{0.85}\textbf{2}&3&4&5\\
\hline
\cellcolor[gray]{0.8}\mathtt{T}&4&3&2&2&\textbf{2}&\cellcolor[gray]{0.85}\textbf{3}&3&3&4\\
\hline
\cellcolor[gray]{0.8}\mathtt{C}&5&4&3&2&3&3&\cellcolor[gray]{0.85}\textbf{3}&4&4\\
\hline
\cellcolor[gray]{0.8}\mathtt{T}&6&5&4&3&3&4&4&\cellcolor[gray]{0.85}\textbf{3}&4\\
\hline
\cellcolor[gray]{0.8}\mathtt{G}&7&6&5&4&3&4&5&4&\cellcolor[gray]{0.85}\textbf{3}\\
\hline
\end{array}
\quad\quad\quad\Rightarrow \quad\quad\quad
\begin{array}{cc}
\mathtt{A}\mathtt{C}\mathtt{C}\mathtt{G}\mathtt{A}\varepsilon\mathtt{C}\mathtt{T}\mathtt{G}&\texttt{Anne}\\
\mathtt{A}\varepsilon\mathtt{C}\varepsilon\mathtt{A}\mathtt{T}\mathtt{C}\mathtt{T}\mathtt{G}&\texttt{Bob}\\
\end{array}
\]
Of course, constructing a sequence alignment for a set of four individuals by only using pairwise comparisons is likely to miss certain optimal alignments. In fact, computing an optimal alignment for our set of individuals would require a generalization of the previous algorithm in a hypercube. The problem with an algorithm based on such a structure is that it can rapidly become computationally expensive. As a result, biologists prefer to use heuristics such as the so-called \emph{progressive method} \cite{Mount,Daugelaite,Feng,Chowdhury,Loytynoja}.

The need for these heuristics show that passing from a two-dimensional point of view to a higher dimensional one, by trying to ``glue'' the previous tables together (in order to reconstruct the hypercube) contains a lot of subtleties. 

In this article, we will show how category theory can help formalize, clarify and reason about this passage by defining formal ``gluing methods'' of two-dimensional tables as used above. These ``gluings algorithms'' will take the form of limit-preserving functors, our so-called \emph{pedigrads}, and the associated gluing instructions will be specified by collections of limit-cones, our so-called \emph{chromologies}.

\subsection{Pre-ordered sets}\label{ssec:pre-ordered_sets}
Throughout the paper, the most basic notions of ordered sets are expected to be known by the reader (\textit{e.g.} partially ordered sets; totally (or linearly) ordered sets; pre-ordered sets; see \cite[Page 11]{MacLane}). However, because pre-orders will play an important role, it was felt appropriate to recall their definition and give some examples of interest in a separate section. We also recall the definition of order-preserving functions and define a category of pre-ordered sets.

\begin{definition}[Pre-ordered sets]
A \emph{pre-ordered set} consists of a set $\Omega$ and a binary relation $\leq$ on $\Omega$ satisfying the following logical implications.
\begin{itemize}
\item[1)] (reflexivity) for every $x \in \Omega$, the relation $x \leq x$ holds;
\item[2)] (transitivity) for every $x,y,z \in \Omega$, if $x \leq y$ and $y \leq z$ hold, then so does $x \leq z$.
\end{itemize}
\end{definition}

\begin{example}\label{exa:pre-ordered_set_0_and_1}
The set $\{0,1\}$ is a pre-ordered set when equipped with the relations $0 \leq 1$; $0 \leq 0$ and $1 \leq 1$. The resulting pre-ordered set is usually known as the \emph{Boolean} pre-ordered set.
\end{example}

\begin{remark}[Representation]\label{rem:representation_pre_ordered_set}
Pre-ordered sets may happen to be sets of labels (or even sets of structures) instead of being sets of integers. In the case of the Boolean pre-ordered set given in Example \ref{exa:pre-ordered_set_0_and_1}, the labels $\mathtt{false}$ and $\mathtt{true}$ will sometimes be used instead of the integers $0$ and $1$, mainly for the sake of clarity when integers are used for another purpose.
\end{remark}

\begin{example}
The set $\{0,1\}$ could also be equipped with the discrete pre-order made of the reflexive relations $0 \leq 0$ and $1 \leq 1$ only.
\end{example}

\begin{example}\label{exa:product_pre-order_set_0_1}
For every positive integer $n$, the $n$-fold Cartesian product $\{0,1\}^{\times n}$ of the pre-ordered set given in Example \ref{exa:pre-ordered_set_0_and_1} is equipped with a pre-order relation $\leq$ that relates two tuples in $\{0,1\}^{\times n}$, say $(x_1,\dots,x_n) \leq (y_1,\dots,y_n)$, if, and only if, the relation $x_i \leq y_i$ holds for every index $i$ between $1$ and $n$.
\end{example}

\begin{example}
The interval $[0,1]$ is a pre-ordered set for the usual pre-order ``being less than or equal to'' defined on the set $\mathbb{R}$ of real numbers.
\end{example}

\begin{remark}[Pre-order categories]
Recall that a pre-ordered set is also a category in which there exists at most one arrow between every pair of objects. In the sequel, a pre-ordered set will sometimes be called a \emph{pre-order category} to emphasize its categorical nature.
\end{remark}

\begin{definition}[Order-preserving functions]
Let $(\Omega_1,\leq_1)$ and $(\Omega_2,\leq_2)$ be two pre-ordered sets. We shall speak of an \emph{order-preserving function} from $(\Omega_1,\leq_1)$ to $(\Omega_2,\leq_2)$ to refer to a function $f:\Omega_1 \to \Omega_2$ for which every relation $x \leq_1 y$ in $\Omega_1$ gives rise to a relation $f(x) \leq_2 f(y)$ in $\Omega_2$.
\end{definition}

\begin{convention}[Notation]
We shall denote by $\mathbf{pOrd}$ the category whose objects are pre-ordered sets and whose morphisms are order-preserving functions.
\end{convention}

\begin{example}[Projection]\label{exa:product_morphisms_set_0_1}
For every positive integer $n$, the $n$-fold Cartesian product $\{0,1\}^{\times n}$ of Example \ref{exa:product_pre-order_set_0_1} is equipped with a canonical collection of $n$ functions $\pi_i:\{0,1\}^{\times n} \to \{0,1\}$, for each $i \in \{1,\dots,n\}$, where a function $\pi_i$ sends a tuple $(x_1,\dots,x_n)$ in $\{0,1\}^{\times n}$ to its $i$-th component $x_i$ in $\{0,1\}$. These functions obviously preserve the order relations of $\{0,1\}^{\times n}$ in $\{0,1\}$ and thus define morphisms in $\mathbf{pOrd}$.
\end{example}

\subsection{Finite sets of integers}
For every positive integer $n$, we will denote by $[n]$ the finite set of integers $\{1,2,\dots,n\}$. We will also let $[0]$ denote the empty set. In the sequel, for every non-negative integer $n$, the set $[n]$ will implicitly be equipped with the order associated with the set of integers (note that the order associated with $[0]$ is the empty order).

\subsection{Segments}\label{ssec:Segments}
Let $(\Omega,\preceq)$ denote a pre-ordered set. 
A \emph{segment} over $\Omega$ consists of a pair of non-negative integers $(n_1,n_0)$, an order-preserving surjection\footnote{\emph{i.e.} an order-preserving function that is a surjection.} $t:[n_1] \to [n_0]$ and a function $c:[n_0] \to \Omega$.

\begin{remark}[Representation]
Segments have all the necessary data to encode the type of pictures given in (\ref{eq:intro_representation_patch}). For a segment $(t,c)$ as defined above, the finite set $[n_1]$ represents the range of elements composing the segment
\[
n_1 = \xymatrix@C-30pt{\bullet&\bullet &\, \cdots\,  & \bullet}
\]
while the fibers $t^{-1}(1), \dots, t^{-1}(n_0)$ of the surjection $t:[n_1] \to [n_0]$ gather these elements into patches (see the brackets below).
\[
t = \xymatrix@C-30pt{(\bullet&\bullet&\bullet)&(\bullet&\bullet&\bullet&\bullet)&(\bullet&\bullet&\,\cdots\, &\bullet)&(\bullet& \bullet)}
\]
Finally, the different colors associated with the patches of the segment are specified by the map $c:[n_0] \to \Omega$. For instance, if we take $\Omega$ to be the Boolean pre-ordered set $\{\mathtt{false}\leq \mathtt{true}\}$ of Example \ref{exa:pre-ordered_set_0_and_1} (see Remark \ref{rem:representation_pre_ordered_set}) and we choose to associate the white color with the $\mathtt{false}$ value and the black color with the $\mathtt{true}$ value, then a set of relations of the form $c(1) = \mathtt{false}$, $c(2) = \mathtt{true}$, $\dots$, $c(n_0-1) = \mathtt{true}$, and $c(n_0) = \mathtt{true}$ will be represented by coloring all the elements of $[n_1]$ living in the fibers $t^{-1}(1)$, $t^{-1}(2)$, $\dots$, $t^{-1}(n_0-1)$, and $t^{-1}(n_0)$ in white and then in black up to the last one, as shown below.
\[
(t,c) = \xymatrix@C-30pt{(\circ&\circ&\circ)&(\bullet&\bullet&\bullet&\bullet)&(\bullet&\bullet&\,\cdots\, &\bullet)&(\bullet& \bullet)}
\]
Note that if $\Omega$ contains more elements, then we need to use more colors (which can also be represented by numbers). These colors could also mean all sorts of things, including actions such as \texttt{ignore}, \texttt{read}, \texttt{start reading}, \texttt{stop reading}, \texttt{misread} (or \texttt{mutate}). The pre-order on the colors would then specify semantic priorities between the different tasks or functions associated with the colors (see the table of pictures below). All these features will be illustrated throughout the examples and remarks of section \ref{sec:Examples_of_pedigrads_in_sets}.
\smallskip

\begin{center}
\begin{tabular}{|c|c|c|}
\hline
\cellcolor[gray]{0.8}2 colors & \cellcolor[gray]{0.8}4 colors & \cellcolor[gray]{0.8}5 colors\\
\hline
$\{0,1\}$ & $\{0,1,2,3\}$ & $\{0,1,2,3,4\}$\\
\hline
{$\xymatrix@-15pt{
\fbox{\texttt{read}}\\
\\
\\
\fbox{\texttt{ignore}}\ar[uuu]}$}
&
{$\xymatrix@-15pt{
&\fbox{\texttt{read}}&\\
\fbox{\texttt{start}}\ar[ru]&&\fbox{\texttt{finish}}\ar[lu]\\
&\fbox{\texttt{ignore}}\ar[ur]\ar[ul]&}$}
&
{$\xymatrix@-15pt{
&\fbox{\texttt{read}}&\\
\fbox{\texttt{start}}\ar[ru]&\fbox{\texttt{misread}}\ar[u]&\fbox{\texttt{stop}}\ar[lu]\\
&\fbox{\texttt{ignore}}\ar[u]\ar[ur]\ar[ul]&}$}\\
\hline
\end{tabular}
\end{center}
\end{remark}

\begin{remark}[Notations]
Note that the specification of the data $n_1$ and $n_0$ is redundant with the data of the function $t$ and $c$. Later on, a segment will often be denoted as a pair $(t,c)$ and, every so often, as an arrow $(t,c):[n_1] \multimap [n_0]$.
\end{remark}

\begin{convention}[Domains, topologies \& types]
For every segment $(t,c):[n_1] \multimap [n_0]$, the data $[n_1]$ will be called the \emph{domain} of $(t,c)$, the data $t$ will be called the \emph{topology} of $(t,c)$ and the data $(n_1,n_0)$ will be called the \emph{type} of $(t,c)$. The type of a segment will always be specified as an arrow of the form $[n_1] \multimap [n_0]$.
\end{convention}

\begin{definition}[Homologous segments]
Two segments $(t,c)$ and $(t',c')$ over $\Omega$ will be said to be \emph{homologous} if their topologies $t$ and $t'$ are equal.
\end{definition}

\begin{definition}[Quasi-homologous segments]
Two segments $(t,c):[n_1] \multimap [n_0]$ and $(t',c'):[n_1'] \multimap [n_0']$ over $\Omega$ will be said to be \emph{quasi-homologous} if their domains $[n_1]$ and $[n_1']$ are equal.
\end{definition}

\subsection{Morphisms of segments}\label{ssec:Morphisms_of_chromosomal_patches}
Let $(\Omega,\preceq)$ be a pre-ordered set and $(t,c):[n_1] \multimap [n_0]$ and $(t',c'):[n_1'] \multimap [n_0']$ be two segments over $\Omega$. A morphism of segments from $(t,c)$ to $(t',c')$ consists of
\begin{itemize}
\item[1)] an order-preserving injection $f_1:[n_1] \to [n_1']$;
\item[2)] an order-preserving function $f_0:[n_0] \to [n_0']$;
\end{itemize}
such that the inequality $c' \circ f_0(i) \preceq c(i)$ holds for every $i \in [n_0]$ and the following diagram commutes.
\[
\xymatrix{
[n_1]\ar@{->>}[r]^{t}\ar@{)->}[d]_{f_1}&[n_0]\ar[d]^{f_0}\\
[n_1']\ar@{->>}[r]^{t'}&[n_0']
}
\]
It is easy to check that the class of morphisms of segments over $\Omega$ is stable under component-wise compositions and admits identities on every segment. We will denote by $\mathbf{Seg}(\Omega)$ the resulting category whose objects are segments over $\Omega$ and whose arrows are morphisms between these.

From now on, we will regard the notations $f_1$ and $f_0$ given above as a conventional notation for morphisms in $\mathbf{Seg}(\Omega)$. Below, we give several examples of typical morphisms in $\mathbf{Seg}(\Omega)$ where $\Omega$ is taken to be the Boolean pre-ordered set of Example \ref{exa:pre-ordered_set_0_and_1}.

\begin{example}[Locality]
If both components $f_1$ and $f_0$ are identities, then the inequality $c' \circ f_0\preceq c$ `decreases' the colors of the segment as illustrated below, on the left.
\[
\begin{array}{c}
\xymatrix@C-30pt{
(\bullet&\bullet&\bullet)&(\bullet&\bullet)&(\bullet&\bullet&\bullet&\bullet)&(\bullet&\bullet&\bullet&\bullet&\bullet)&(\circ&\circ&\circ)&(\bullet)
}\\
\rotatebox[origin=c]{-90}{$\longrightarrow$} \\
\xymatrix@C-30pt{
(\circ&\circ&\circ)&(\circ&\circ)&(\bullet&\bullet&\bullet&\bullet)&(\bullet&\bullet&\bullet&\bullet&\bullet)&(\circ&\circ&\circ)&(\circ)
}
\end{array}
\quad\quad\quad\quad\quad
\begin{array}{c}
\xymatrix@C-30pt@R-11pt{
(\dots)&(\circ&\circ\ar[d]|{\xcancel{\quad\,\,}}&\circ)&(\dots)\\
(\dots)&(\bullet&\bullet&\bullet)&(\dots)
}
\end{array}
\]
\textit{Interpretation:} This type of morphisms tells us that one is able to select/cut local patches from a segment. This is, for instance, the type of morphisms that one may want to use to model CRISPR, namely separating a patch from a segment. Note that, because reading a segment (in black) has a higher semantic priority than ignoring it (in white), turning white regions into black ones, as shown above, on the right, is forbidden. The order relation on the colors can therefore be a way of encoding forgetful operations (\textit{e.g.} irreversible or energy-releasing events).
\end{example}

\begin{example}[Relativity]\label{exa:Relativity_morphism}
If only the component $f_1$ is an identity morphism, then the component $f_0$ can merge the regions defining the topology.
\[
\begin{array}{c}
\xymatrix@C-30pt{
(\bullet&\bullet&\bullet)&(\circ&\circ)&(\bullet&\bullet&\bullet&\bullet)&(\bullet&\bullet&\bullet&\bullet&\bullet)&(\circ&\circ&\circ)&(\circ)
}\\
\rotatebox[origin=c]{-90}{$\longrightarrow$} \\
\xymatrix@C-30pt{
(\circ&\circ&\circ&\circ&\circ)&(\bullet&\bullet&\bullet&\bullet&\bullet&\bullet&\bullet&\bullet&\bullet)&(\circ&\circ&\circ&\circ)
}
\end{array}
\]
\textit{Interpretation:} This type of morphisms implies that the way one parses the patches of a segment influences the way one parses the whole segment (\textit{e.g.} from codons to genes). However, because there is no arrow that increases the number of brackets from its domain to its codomain, the way one parses a segment might not necessarily reflect the way the patches are parsed (\textit{e.g.} from gene to codons).
\end{example}

\begin{example}[Flexibility]\label{exa:Flexibility_morphism}
If the component $f_1$ is not an identity morphism, then the range of the segment increases. Below, we suppose that the identity $c' \circ f_0 = c$ holds.
\[
\begin{array}{r}
\xymatrix@C-30pt{
(\bullet&\bullet&\bullet)&(\circ&\circ)&(\bullet&\bullet&\bullet&\bullet)&(\bullet&\bullet&\bullet&\bullet&\bullet)&(\circ&\circ&\circ)&(\circ)
}\\
\multicolumn{1}{c}{\rotatebox[origin=c]{-90}{$\longrightarrow$}}\\
\xymatrix@C-30pt{
(\bullet&\bullet&\bullet&\bullet)&(\bullet)&(\circ&\circ&\circ&\circ)&(\bullet&\bullet&\bullet&\bullet)&(\bullet&\bullet&\bullet&\bullet&\bullet)&(\circ&\circ&\circ)&(\circ)
}
\end{array}
\]
\textit{Interpretation:} This type of morphisms allows one to insert particular nucleotides or spaces in the parsing of a segment. 
For instance, spaces become necessary if one wants to align segments that are not necessarily (quasi-)homologous (this was shown in section \ref{ssec:Main_example}). A morphism inserting a space would then correspond to a choice of `sequence alignment' in bioinformatics (see Example \ref{exa:morphisms_as_inclusion_of_words}, Example \ref{exa:Global-local_seq_alignments} and section \ref{ssec:slices_and_mechanisms}).
\end{example}

\begin{remark}[Initial object]
For every pre-ordered set $(\Omega,\preceq)$, the segment (over $\Omega$) of type $[0] \multimap [0]$ that is given by the obvious order-preserving surjection $!:\emptyset \to \emptyset$ and the canonical function $!:\emptyset \to \Omega$ is an initial object in $\mathbf{Seg}(\Omega)$. Note that such an object is formal and does not really possess any biological interpretation other than giving a way to express the idea of `absence'.
\end{remark}

\subsection{Relating categories of segments}\label{ssec:relating_categories_of_segments}
So far, our examples have only considered categories of segments over the Boolean pre-ordered set $\{0 \leq 1\}$. In practice, Boolean segments are convenient and easy to think about. Thus, it can be useful to have ways to go from a category of segments whose pre-ordered set is not $\{0 \leq 1\}$ to the category of segments whose pre-ordered is $\{0 \leq 1\}$. The goal of Proposition \ref{prop:functor_segment_categories} is to show that this type of transfer is possible.

\begin{proposition}[Functor]\label{prop:functor_segment_categories}
Let $f:(\Omega_1,\preceq_1) \to (\Omega_2,\preceq_2)$ be a morphism in $\mathbf{pOrd}$. The mapping rule $(t,c) \mapsto (t,f \circ c)$ extends to a faithful functor $\mathbf{Seg}(f):\mathbf{Seg}(\Omega_1) \to \mathbf{Seg}(\Omega_2)$ sending a morphism $(f_1,f_0)$ to the same pair $(f_1,f_0)$.
\end{proposition}
\begin{proof}
Let $(f_1,f_0):(t,c) \to (t',c')$ be an arrow in $\mathbf{Seg}(\Omega_1)$ and let $[n_1]$ denote the domain of $(t,c)$. Because $f$ is an order-preserving function, the relation $c'\circ f_0(i) \preceq_1 c(i)$, satisfied for every element $i \in [n_1]$, gives rise to a relation $f \circ c'\circ f_0(i) \preceq_2 f \circ c(i)$, for every element $i\in[n_1]$. Since the domain of the segment $(t,f \circ c)$ is also $[n_1]$, the previous relation shows that the pair $(f_1,f_0)$ defines a representative for an arrow of the form $(t,f \circ c) \to (t',f \circ c')$ in $\mathbf{Seg}(\Omega_2)$. The faithfulness property as well as the composition and identity axioms follow easily.
\end{proof}

In fact, Proposition \ref{prop:functor_segment_categories} hides a functor structure on the category of pre-ordered sets, but this structure will not be needed in this paper.

\begin{example}[Preparation example]\label{exa:preparation_example}
The present example is the first of a series that address the goals presented in section \ref{ssec:Main_example}.
Let $(\Omega,\preceq)$ denote the Boolean pre-ordered set $\{0 \leq 1\}$. We follow the notation of Definition \ref{exa:product_pre-order_set_0_1} and denote the pre-ordered set $\{0,1\}^{\times 4}$ as $\Omega^{\times 4}$. Since $\Omega^{\times 4}$ is the 4-fold Cartesian product of $\Omega$, it is equipped with four order-preserving functions $\pi_i:\Omega^{\times 4} \to \Omega$ (see Example \ref{exa:product_morphisms_set_0_1}), which we purposely index with $i \in \{\mathtt{a},\mathtt{b},\mathtt{c},\mathtt{d}\}$. Later, each of these indices will be used to represent one of the four individuals of section \ref{ssec:Main_example}. Now, by Proposition \ref{prop:functor_segment_categories}, a morphism $\pi_i:\Omega^{\times 4} \to \Omega$ induces a functor as follows.
\[
\mathbf{Seg}(\pi_i):\mathbf{Seg}(\Omega^{\times 4}) \to \mathbf{Seg}(\Omega)
\]
If we represent an element $(\mathtt{x_{a}},\mathtt{x_{b}},\mathtt{x_{c}},\mathtt{x_{d}})$ in $\Omega^{\times 4}$ as a word $[\mathtt{x_{a}}\mathtt{x_{b}}\mathtt{x_{c}}\mathtt{x_{d}}]$, then the functor $\mathbf{Seg}(\pi_i)$ satisfies the following mapping rules for the various values of $i$ shown in the rightmost column.
\[
\begin{array}{ccc|c}
\cellcolor[gray]{0.8}\mathbf{Seg}(\Omega^{\times 4})&\cellcolor[gray]{0.8} \longrightarrow & \cellcolor[gray]{0.8}\mathbf{Seg}(\Omega)&\cellcolor[gray]{0.8}i\\
\hline
\xymatrix@C-30pt{([\mathtt{1\separ0\separ1\separ0}]&[\mathtt{1\separ0\separ1\separ0}])&([\mathtt{0\separ1\separ1\separ0}]&[\mathtt{0\separ1\separ1\separ0}]&[\mathtt{0\separ1\separ1\separ0}])&([\mathtt{1\separ1\separ1\separ1}]&[\mathtt{1\separ1\separ1\separ1}])}& \mapsto &
\xymatrix@C-30pt{(\bullet&\bullet)&(\circ&\circ&\circ)&(\bullet&\bullet)} & i = \mathtt{a}\\
\xymatrix@C-30pt{([\mathtt{1\separ0\separ1\separ0}]&[\mathtt{1\separ0\separ1\separ0}])&([\mathtt{0\separ1\separ1\separ0}]&[\mathtt{0\separ1\separ1\separ0}]&[\mathtt{0\separ1\separ1\separ0}])&([\mathtt{1\separ1\separ1\separ1}]&[\mathtt{1\separ1\separ1\separ1}])}& \mapsto &
\xymatrix@C-30pt{(\circ&\circ)&(\bullet&\bullet&\bullet)&(\bullet&\bullet)} & i = \mathtt{b}\\
\xymatrix@C-30pt{([\mathtt{1\separ0\separ1\separ0}]&[\mathtt{1\separ0\separ1\separ0}])&([\mathtt{0\separ1\separ1\separ0}]&[\mathtt{0\separ1\separ1\separ0}]&[\mathtt{0\separ1\separ1\separ0}])&([\mathtt{1\separ1\separ1\separ1}]&[\mathtt{1\separ1\separ1\separ1}])}& \mapsto &
\xymatrix@C-30pt{(\bullet&\bullet)&(\bullet&\bullet&\bullet)&(\bullet&\bullet)} & i = \mathtt{c}\\
\xymatrix@C-30pt{([\mathtt{1\separ0\separ1\separ0}]&[\mathtt{1\separ0\separ1\separ0}])&([\mathtt{0\separ1\separ1\separ0}]&[\mathtt{0\separ1\separ1\separ0}]&[\mathtt{0\separ1\separ1\separ0}])&([\mathtt{1\separ1\separ1\separ1}]&[\mathtt{1\separ1\separ1\separ1}])}& \mapsto &
\xymatrix@C-30pt{(\circ&\circ)&(\circ&\circ&\circ)&(\bullet&\bullet)} & i = \mathtt{d}\\
\end{array}
\]
In the sequel, the category of segments $\mathbf{Seg}(\Omega^{\times 4})$ will be used as a logic to reason about our main example presented in section \ref{ssec:Main_example}.
\end{example}

\subsection{Pre-orders on homologous segments}
Let $(\Omega,\preceq)$ be a pre-ordered set and let $t:[n_1] \to [n_0]$ be an order-preserving surjection. The subcategory of $\mathbf{Seg}(\Omega)$ whose objects are the homologous segments of topology $t$ and whose arrows are the morphisms of segments for which the components $f_0$ and $f_1$ are identities will be denoted by $\mathbf{Seg}(\Omega:t)$
and referred to as the \emph{category of homologous segments (over $\Omega$) of topology $t$}.

\begin{proposition}[Pre-order category]
For every order-preserving surjection $t:[n_1] \to [n_0]$, the category $\mathbf{Seg}(\Omega:t)$ is a pre-order category.
\end{proposition}
\begin{proof}
According to section \ref{ssec:Morphisms_of_chromosomal_patches} and the definition of $\mathbf{Seg}(\Omega:t)$, giving an arrow $(t,c) \to (t,c')$ in $\mathbf{Seg}(\Omega:t)$ amounts to giving a pre-order relation $c'(i) \preceq c(i)$ in $(\Omega,\preceq)$ for every $i \in [n_0]$. It is straightforward to see that this defines a reflexive and transitive binary relation.
\end{proof}

\subsection{Pre-orders on quasi-homologous segments}
Let $(\Omega,\preceq)$ be a pre-ordered set and let $n_1$ be a non-negative integer. The subcategory of $\mathbf{Seg}(\Omega)$ whose objects are the quasi-homologous segments of domain $[n_1]$ and whose arrows are the morphisms of segments for which the component $f_1$ is an identity will be denoted by $\mathbf{Seg}(\Omega\,|\,n_1)$ and called the \emph{category of quasi-homologous segments (over $\Omega$) of domain $n_1$}.

\begin{proposition}[Pre-order category]\label{prop:Seg_Omega_n_porder}
For every non-negative integer $n_1$, the category of quasi-homologous segments $\mathbf{Seg}(\Omega\,|\,n_1)$ is a pre-order category.
\end{proposition}
\begin{proof}
Let $(\mathrm{id},f_0):(t,c) \to (t',c')$ and $(\mathrm{id},g_0):(t,c) \to (t',c')$ be two morphisms in $\mathbf{Seg}(\Omega\,|\,n_1)$. We want to show that these morphisms are equal. According section \ref{ssec:Morphisms_of_chromosomal_patches} and the definition of $\mathbf{Seg}(\Omega\,|\,n_1)$, the two identities $f_0 \circ t = t'$ and $g_0 \circ t = t'$ hold, which implies that the identity $g_0 \circ t = f_0 \circ t$ holds. Because $t$ is an epimorphism, the identity $g_0=f_0$ must hold.
\end{proof}

\begin{remark}[Zero domain]
The category $\mathbf{Seg}(\Omega\,|\,0)$ of quasi-homologous segments with empty domain $[0]$ is a terminal category whose only object is the initial object of $\mathbf{Seg}(\Omega)$.
\end{remark}

\subsection{Cones}\label{ssec:cones}
Recall that a \emph{cone} in a category $\mathcal{C}$ consists of an object $X$ in $\mathcal{C}$, a small category $A$, a functor $F:A \to \mathcal{C}$ and a natural transformation $\Delta_{A}(X) \Rightarrow F$ where $\Delta_{A}(X)$ denotes the constant functor $A \to \mathbf{1} \to \mathcal{C}$ mapping every object in $A$ to the object $X$ in $\mathcal{C}$.

\begin{definition}[Wide spans]\label{def:wide_spans}
In the sequel, we shall speak of a \emph{wide span} to refer to a cone $\Delta_{A}(X) \Rightarrow F$ defined over a finite discrete small category $A$ whose objects are ordered with respect to a total order (this will allow us to have canonical choices of limit constructions).
\end{definition}

\begin{example}[Wide spans]\label{exa:wide_spans}
Giving a wide span in a category $\mathcal{C}$ amounts to giving a finite collection of arrows $\mathbf{S} :=\{f_i:X \to F_i\}_{i \in [n]}$ in $\mathcal{C}$. When the category $\mathcal{C}$ has products, the implicit order of the set $[n]=\{1,\dots,n\}$ can be used to give a specific representative to the product of the collection $\{F_i\}_{i \in [n]}$ in $\mathcal{C}$.
\end{example}

\subsection{Chromologies}\label{ssec:chromologies}
A \emph{chromology} is a pre-ordered set $(\Omega,\preceq)$ that is equipped, for every non-negative integer $n$, with a set $D[n]$ of cones in the category $\mathbf{Seg}(\Omega\,|\,n)$. Such a chromology will later be denoted as a pair $(\Omega,D)$.

\begin{remark}[Future examples]
In section \ref{ssec:Distributive_and_exactly_distributive_chromologies}, we will see several examples of chromologies, which will be used throughout this article.
\end{remark}

\subsection{Logical systems}\label{ssec:Logical_systems}
By a \emph{logical system}, we mean a category $\mathcal{C}$ that is equipped with a subclass of its cones $\mathcal{W}$ (see section \ref{ssec:cones}).

\begin{remark}[Size matters]
The only difference between a logical system and a limit sketch is the sizes of their collections of objects: that of the latter is a set while that of the former is a class. This does make a difference in the type of properties that the two definitions satisfy. Because of their sizes, logical systems will only be used as codomains of functors. On the other hand, a chromology, which is a limit sketch, will often turn out to be the domain of a functor.
\end{remark}

\subsection{Pedigrads} \label{ssec:Pedigrads}
Pedigrads are algebraic structures that model the logical rules of chromologies.  Their name refers to the concept of `pedigree' used in genetics to draw the genealogy of a set of taxa. Let $(\Omega,D)$ be a chromology and $(\mathcal{C},\mathcal{W})$ be a logical system. A \emph{pedigrad in $(\mathcal{C},\mathcal{W})$ for $(\Omega,D)$} is a functor $\mathbf{Seg}(\Omega) \to \mathcal{C}$ sending, for every non-negative integer $n$, the cones in $D[n]$ to cones in $\mathcal{W}$. 

\begin{convention}[$\mathcal{W}$-pedigrads]
Because we will often consider the same category $\mathcal{C}$ for different classes of cones $\mathcal{W}$, we will often refer to a pedigrad in $(\mathcal{C},\mathcal{W})$ as a $\mathcal{W}$-pedigrad.
\end{convention}

\section{Examples of pedigrads in sets}\label{sec:Examples_of_pedigrads_in_sets}

The goal of the present section is to formalize the concept of sequence alignment in terms of a functor on a subcategory of segments and to show that the right Kan extension of this functor on the whole category of segments computes what can be seen as multiple sequence alignments. The consistency of the integrated data is then studied by using the concepts of pedigrad and chromology. 

\subsection{Truncation functors}
In this section, we define a truncation operation (Definition \ref{def:truncation}) and show that this operation is a functor on a category of quasi-homologous segments (Proposition \ref{prop:Tr_functor_Set_op}). Extending this functoriality property to the whole category of segments is not straightforward and requires a few more steps (see Proposition \ref{prop:Tr_functor_pointed_Set_op}). In section \ref{sec:Example_pedigrads_in_sets}, we use the resulting functor to construct pedigrads.

\begin{definition}[Truncation]\label{def:truncation}
For every segment $(t,c):[n_1] \multimap [n_0]$ over $\Omega$ and element $b \in \Omega$, we will denote by $\mathsf{Tr}_b(t,c)$ the subset $\{i \in [n_1]~|~ b \preceq c \circ t(i)\}$ of $[n_1]$. This is the set of all elements in $[n_1]$ whose images via $c \circ t$ are greater than or equal to $b$ in $\Omega$.
\end{definition}

\begin{example}[Truncation]\label{exa:truncation}
Let $(\Omega,\preceq)$ be the Boolean pre-ordered set $\{0 \leq 1\}$. If we consider the segment $(t,c)$ of $\mathbf{Seg}(\Omega)$ given below on the left, then the operation $\mathsf{Tr}_b$ for which $b$ is taken to be equal to $1$ will only select the integers in the domain of $(t,c)$ that are associated with black nodes. On the other hand, the operation $\mathsf{Tr}_b$ for which $b$ is taken to be equal to $0$ will select all the integers in the domain of $(t,c)$.
\[
(t,c)=\xymatrix@C-30pt{
(\bullet&\bullet&\bullet)&(\circ&\circ)&(\bullet&\bullet&\bullet&\bullet)&(\circ&\circ&\circ&\circ&\circ)&(\bullet&\bullet&\bullet)&(\circ)
}
\quad\quad
\begin{array}{l}
\mathsf{Tr}_1(t,c)=\{1,2,3,6,7,8,9,15,16,17\}\\
\mathsf{Tr}_0(t,c)=[18]
\end{array}
\]

Similarly, if we let $(\Omega,\preceq)$ denote the pre-ordered set $\{0 \leq 1 \leq 2\}$ (used in Remark \ref{rem:resolving_inconsistencies}), then we obtain the following truncations for the segment $(t,c)$ given below, on the left.
\[
(t,c)=\xymatrix@C-30pt{
(\mathtt{1}&\mathtt{1}&\mathtt{1})&(\mathtt{0}&\mathtt{0})&(\mathtt{2}&\mathtt{2}&\mathtt{2}&\mathtt{2})&(\mathtt{0}&\mathtt{0}&\mathtt{0}&\mathtt{0}&\mathtt{0})&(\mathtt{1}&\mathtt{1}&\mathtt{1})&(\mathtt{0})
}
\quad\quad
\begin{array}{l}
\mathsf{Tr}_2(t,c)=\{6,7,8,9\}\\
\mathsf{Tr}_1(t,c)=\{1,2,3,6,7,8,9,15,16,17\}\\
\mathsf{Tr}_0(t,c)=[18]
\end{array}
\]
Here, the reader may have noticed that any relation of the form $b \preceq b'$ will lead to an inclusion of the form $\mathsf{Tr}_{b'}(t,c) \subseteq \mathsf{Tr}_{b}(t,c)$. This property, even though interesting, is not used in this paper. 
\end{example}

\begin{definition}[Sub-objects]
For every non-negative integer $n$, we will speak of a \emph{sub-object} of $[n]$ to refer to a subset of $[n]$. A \emph{morphism of sub-objects of $[n]$} is an inclusion of sets between the two sub-objects.
\end{definition}

\begin{example}[Truncation operations and sub-objects]
Let $(\Omega,\preceq)$ be the Boolean pre-ordered set $\{0 \leq 1\}$. If we consider the morphism of segments that is given in Example \ref{exa:Relativity_morphism}, which we recall below, on the left-hand side, we can see that the truncation operation $\mathsf{Tr}_1$ gives, on the right, two sub-objects of the domain $[18]$ that can be related via a morphism of sub-objects.
\[
\begin{array}{crc}
(t,c)=\xymatrix@C-30pt{
(\bullet&\bullet&\bullet)&(\circ&\circ)&(\bullet&\bullet&\bullet&\bullet)&(\bullet&\bullet&\bullet&\bullet&\bullet)&(\circ&\circ&\circ)&(\circ)
}&&\mathsf{Tr}_1(t,c)=\{1,2,3,6,7,8,9,10,11,12,13,14\}\\
\rotatebox[origin=c]{-90}{$\longrightarrow$} &&\rotatebox[origin=c]{-90}{$\supseteq$}\\
(t',c')=\xymatrix@C-30pt{
(\circ&\circ&\circ&\circ&\circ)&(\bullet&\bullet&\bullet&\bullet&\bullet&\bullet&\bullet&\bullet&\bullet)&(\circ&\circ&\circ&\circ)
}&&\mathsf{Tr}_1(t',c')=\{6,7,8,9,10,11,12,13,14\}
\end{array}
\]
The fact that a morphism of segments of the form $(t,c) \to (t',c')$ gives rise to an inclusion $\mathsf{Tr}_1(t',c') \subseteq \mathsf{Tr}_1(t,c)$ is explained by Lemma \ref{lem:preserve_Tr_b_morphism}.
\end{example}

\begin{lemma}\label{lem:preserve_Tr_b_morphism}
Let $(f_1,f_0):(t,c) \to (t',c')$ be a morphism in $\mathbf{Seg}(\Omega)$. If the relation $f_1(i) \in \mathsf{Tr}_b(t',c')$ holds, then so does the relation $i \in \mathsf{Tr}_b(t,c)$. 
\end{lemma}
\begin{proof}
Recall that, by definition of a morphism in $\mathbf{Seg}(\Omega)$, the inequality $c' \circ f_0 \preceq c$ holds. Now, if the relation $f_1(i) \in \mathsf{Tr}_b(t',c')$ holds, then so do the following pre-order relations.
\[
b \preceq c' \circ t'  \circ f_1(i) = c' \circ f_0 \circ  t (i) \preceq  c \circ  t (i)
\]
By transitivity, we obtain the inequality $b \preceq c \circ  t (i)$, which implies that $i$ must be in $\mathsf{Tr}_b(t,c)$.
\end{proof}

The following proposition only shows one side of the functorial properties satisfied by the truncation operation. Proposition \ref{prop:Tr_functor_pointed_Set_op} will give a different functor construction, which is related to that given in Proposition \ref{prop:Tr_functor_Set_op} via the statement of Proposition \ref{prop:Tr_ast_on_Chr_Omega_t_equals_F_Tr}. While Proposition \ref{prop:Tr_functor_pointed_Set_op} will later be used to construct pedigrads, Proposition \ref{prop:Tr_functor_Set_op} will be used to deduce properties related to them.

\begin{proposition}\label{prop:Tr_functor_Set_op}
For every element $b \in \Omega$ and non-negative integer $n_1$, the mapping $(t,c) \mapsto \mathsf{Tr}_b(t,c)$ extends to a functor
$\mathsf{Tr}_b:\mathbf{Seg}(\Omega\,|\,n_1) \to \mathbf{Set}^{\mathrm{op}}$, which factorizes through the opposite category of sub-objects of $[n_1]$.
\end{proposition}
\begin{proof}
By definition, for every segment $(t,c)$ in $\mathbf{Seg}(\Omega\,|\,n_1)$, the set $\mathsf{Tr}_b(t,c)$ is a subset of $[n_1]$. For every morphism $(\mathrm{id},f_0):(t,c) \to (t,c')$ in $\mathbf{Seg}(\Omega\,|\,n_1)$, Lemma \ref{lem:preserve_Tr_b_morphism} shows that there is an inclusion
$\mathsf{Tr}_b(t,c') \subseteq \mathsf{Tr}_b(t,c)$. Since the opposite category of sub-objects of $[n_1]$ is a pre-order category, the functor structure is obvious and the statement follows. 
\end{proof}

The extension of the functorial property given in Proposition \ref{prop:Tr_functor_Set_op} to the whole category of segments requires to change the codomain category $\mathbf{Set}$ to the category $\mathbf{Set}_{\ast}$ of pointed sets and point-preserving maps (see Example \ref{exa:explain_pointed_maps}). In this respect, recall that there is an adjunction
\[
\xymatrix{
\mathbf{Set} \ar@<+1.2ex>[r]^{\mathbb{F}} \ar@{}[r]|{\bot}\ar@<-1.2ex>@{<-}[r]_{\mathbb{U}} & \mathbf{Set}_{\ast}
}
\]
whose right adjoint $\mathbb{U}:\mathbf{Set}_{\ast} \to \mathbf{Set}$ forgets the pointed structure (\emph{i.e.} $\mathbb{U}:(X,p) \mapsto X$) and whose left adjoint 
$\mathbb{F}:\mathbf{Set} \to \mathbf{Set}_{\ast}$ maps a set $X$ to the obvious pointed set $(X+\{\ast\},\ast)$ and maps a function $f:X \to Y$ to the coproduct map $f+\{\ast\}:X+\{\ast\} \to Y+\{\ast\}$.

\begin{proposition}\label{prop:Tr_functor_pointed_Set_op}
For every element $b \in \Omega$, the mapping $(t,c) \mapsto \mathbb{F}\mathsf{Tr}_b(t,c)$ extends to a functor $\mathsf{Tr}_b^{\ast}:\mathbf{Seg}(\Omega) \to \mathbf{Set}_{\ast}^{\mathrm{op}}$ mapping every function $(f_1,f_0):(t,c) \to (t',c')$ in $\mathbf{Seg}(\Omega)$ to the following map of pointed sets.
\[
\begin{array}{llllll}
\mathsf{Tr}_b^{\ast}(f_1,f_0)&:&\mathbb{F}\mathsf{Tr}_b(t',c')&\to&\mathbb{F}\mathsf{Tr}_b(t,c)&\\
&&j&\mapsto & i &\textrm{if } \exists i \in \mathsf{Tr}_b(t,c) \,: \, j = f_1(i);\\
&&j&\mapsto & \ast &\textrm{otherwise.}\\
\end{array}
\]
\end{proposition}
\begin{proof}

The well-definedness of the point-preserving map $\mathsf{Tr}_b^{\ast}(f_1,f_0)$ follows from Lemma \ref{lem:preserve_Tr_b_morphism}. The mapping $\mathsf{Tr}_b^{\ast}$ obviously satisfies the identity axiom associated with the concept of a functor. The composition axiom is shown as follows. Take two morphisms $(f_1,f_0):(t,c) \to (t',c')$ and $(f_1',f_0'):(t',c') \to (t^{\prime\prime},c^{\prime\prime})$ in $\mathbf{Seg}(\Omega)$. The image of the composition $(f_1',f_0') \circ (f_1,f_0)$ via the operation $\mathsf{Tr}_b^{\ast}$ takes the following form.
\[
\begin{array}{llll}
j&\mapsto & (f_1' \circ f_1)^{-1}(j)& \textrm{if } (f_1' \circ f_1)^{-1}(j) \neq \emptyset\\
j&\mapsto & \ast &\textrm{otherwise.}\\
\end{array}
\]
Since the identity $(f_1' \circ f_1)^{-1}(j) = (f_1)^{-1}((f_1')^{-1}(j))$ holds whenever the inequality $(f_1')^{-1}(j) \neq \emptyset$ is satisfied, we deduce that the following equation holds.
\[
\mathsf{Tr}_b^{\ast}((f_1',f_0') \circ (f_1,f_0)) = \mathsf{Tr}_b^{\ast}(f_1,f_0) \circ \mathsf{Tr}_b^{\ast}(f_1',f_0')
\]
This last equation shows that $\mathsf{Tr}_b^{\ast}$ is a functor going to the opposite category of $\mathbf{Set}_{\ast}$.
\end{proof}

\begin{example}[Truncation operations and pointed sets]\label{exa:explain_pointed_maps}
Let $(\Omega,\preceq)$ be the Boolean pre-ordered set $\{0 \leq 1\}$. 
Consider the morphism of segments of Example \ref{exa:Flexibility_morphism}, given below, on the left. Its mapping has further been detailed by using adequate labeling to show how the first segment is mapped to the second one.
On the right, we can see its image via the truncation operation $\mathsf{Tr}_1$ where we see that the indices 4 and 5, in the truncated codomain, do not have corresponding indices in the truncated domain.
\[
\begin{array}{cccc}
\xymatrix@C-30pt{
(\mathop{\bullet}\limits^1&\mathop{\bullet}\limits^2&\mathop{\bullet}\limits^3)&(\mathop{\circ}\limits^4&\mathop{\circ}\limits^5)&(\mathop{\bullet}\limits^6&\mathop{\bullet}\limits^7&\mathop{\bullet}\limits^8&\mathop{\bullet}\limits^9)&(\bullet&\bullet&\bullet&\bullet&\bullet)&(\circ&\circ&\circ)&(\circ)
}&\{1,2,3,6,7,8,9,10,11,12,13,14\}&&?\\
\rotatebox[origin=c]{-90}{$\longrightarrow$} &\rotatebox[origin=c]{90}{$\longrightarrow$}&&\rotatebox[origin=c]{90}{$\longrightarrow$}\\
\xymatrix@C-30pt{
(\mathop{\bullet}\limits^1&\mathop{\bullet}\limits^2&\mathop{\bullet}\limits^3&\mathop{\bullet}\limits^{\ast})&(\mathop{\bullet}\limits^{\ast})&(\mathop{\circ}\limits^4&\mathop{\circ}\limits^5&\mathop{\circ}\limits^{\ast}&\mathop{\circ}\limits^{\ast})&(\mathop{\bullet}\limits^6&\mathop{\bullet}\limits^7&\mathop{\bullet}\limits^8&\mathop{\bullet}\limits^9)&(\bullet&\bullet&\bullet&\bullet&\bullet)&(\circ&\circ&\circ)&(\circ)
}&\{1,2,3,10,11,12,\dots,16,17,18\}&\cup&\{4,5\}
\end{array}
\]
A way to associate the indices $4$ and $5$ with an element in the truncated domain is to formally add one, thus explaining the pointed structure used in Proposition \ref{prop:Tr_functor_pointed_Set_op}.
\end{example}

\begin{proposition}\label{prop:Tr_ast_on_Chr_Omega_t_equals_F_Tr}
For every element $b \in \Omega$ and non-negative integer $n_1$, the following diagram commutes.
\[
\xymatrix@C-8pt{
*+!R(.4){\mathbf{Seg}(\Omega\,|\,n_1)}\ar[d]_{\mathsf{Tr}_b}\ar[r]^{\subseteq}&*+!L(.4){\mathbf{Seg}(\Omega)}\ar[d]^{\mathsf{Tr}_b^{\ast}}\\
\mathbf{Set}^{\mathrm{op}}\ar[r]_{\mathbb{F}^{\mathrm{op}}}&*+!L(.4){\mathbf{Set}_{\ast}^{\mathrm{op}}}
}
\]
\end{proposition}
\begin{proof}
By definition, if we restrict the functor $\mathsf{Tr}_b^{\ast}:\mathbf{Seg}(\Omega) \to \mathbf{Set}_{\ast}^{\mathrm{op}}$ to the subcategory $\mathbf{Seg}(\Omega\,|\,n_1) \hookrightarrow \mathbf{Seg}(\Omega)$, then every morphism $(t,c) \to (t,c')$ in the pre-order category $\mathbf{Seg}(\Omega\,|\,n_1)$ is sent to the following map in $\mathbf{Set}_{\ast}$ (see Proposition \ref{prop:Tr_functor_pointed_Set_op}).
\[
\begin{array}{llllll}
\mathsf{Tr}_b^{\ast}(f_1,f_0)&:&\mathbb{F}\mathsf{Tr}_b(t,c')&\to&\mathbb{F}\mathsf{Tr}_b(t,c)&\\
&&j&\mapsto & j & j \in \mathsf{Tr}_b(t,c')\\
&&\ast&\mapsto & \ast &\textrm{otherwise.}\\
\end{array}
\]
This means that the restriction of $\mathsf{Tr}_b^{\ast}$ on $\mathbf{Seg}(\Omega\,|\,n_1)$ can be retrieved from the application of the functor $\mathbb{F}$ on the images of $\mathsf{Tr}_b$.
\end{proof}

\subsection{Examples of pedigrads in sets}\label{sec:Example_pedigrads_in_sets}
In this section, we construct a collection of functors $\mathbf{Seg}(\Omega) \to \mathbf{Set}$ for any pointed set $(E,\varepsilon)$ and parameter $b$ in $\Omega$ (see Definition \ref{def:set_E_b_varepsilon}). Later on, we will define various classes of cones $\mathcal{W}$ in $\mathbf{Set}$ for which these functors are $\mathcal{W}$-pedigrad (see Theorem \ref{theo:E_b_varepsilon_W_iso_pedigrad_exactly_distributive} and Theorem \ref{theo:E_b_varepsilon_W_surj_pedigrad_injective}).

\begin{convention}[Notation]
In the sequel, the hom-set of a category $\mathcal{C}$ from an object $X$ to an object $Y$ will be denoted as $\mathcal{C}(X,Y)$. For instance, the set of functions from a set $X$ to a set $Y$ will be denoted by $\mathbf{Set}(X,Y)$. Also, recall that, for any category $\mathcal{C}$, the hom-sets give rise to a functor $\mathcal{C}(\_,\_):\mathcal{C}^{\mathrm{op}}\times \mathcal{C} \to \mathbf{Set}$ called the \emph{hom-functor} \cite[page 27]{MacLane}.
\end{convention}

\begin{definition}[Environment functors]\label{def:set_E_b_varepsilon}
For every element $b \in \Omega$, we will denote by $E_{b}^{\varepsilon}$ the functor $\mathbf{Seg}(\Omega) \to \mathbf{Set}$ defined as the composition of the following pair of functors.
\[
\xymatrix@C+20pt{
\mathbf{Seg}(\Omega)\ar[r]^{\mathsf{Tr}_b^{\ast}}&\mathbf{Set}_{\ast}^{\mathrm{op}}\ar[rr]^{\mathbf{Set}_{\ast}(\_,(E,\varepsilon))}&&\mathbf{Set}
}
\]
\end{definition}

\begin{remark}\label{rem:E_b_varepsilon_as words_functions}
For every object $(t,c)$ in $\mathbf{Seg}(\Omega)$, an element in $E_b^{\varepsilon}(t,c)$ can be seen as a function of the form $\mathsf{Tr}_b(t,c) \to E$ according to the following series of bijections.
\begin{align*}
E_b^{\varepsilon}(t,c)& = \mathbf{Set}_{\ast}(\mathsf{Tr}_b^{\ast}(t,c),(E,\varepsilon))&\\
&= \mathbf{Set}_{\ast}(\mathbb{F}\mathsf{Tr}_b(t,c),(E,\varepsilon))&(\textrm{Definition of }\mathsf{Tr}_b^{\ast})\\
&\cong \mathbf{Set}(\mathsf{Tr}_b(t,c),\mathbb{U}(E,\varepsilon))&(\mathbb{F} \dashv \mathbb{U})\\
&= \mathbf{Set}(\mathsf{Tr}_b(t,c),E)&(\textrm{Definition of }\mathbb{U})
\end{align*}
Because the set $\mathsf{Tr}_b(t,c)$ is equipped with the usual order of natural numbers, we will represent an element in $E_b^{\varepsilon}(t,c)$ as a word of elements in $E$ (see Example \ref{exa:elements_as_words}).
\end{remark}

\begin{example}[Objects]\label{exa:elements_as_words}
Suppose that $(\Omega,\preceq)$ denotes the Boolean pre-ordered set $\{0\leq 1\}$ and let $(E,\varepsilon)$ be the pointed set $\{\mathtt{A},\mathtt{C},\mathtt{G},\mathtt{T},\varepsilon\}$. If we consider the segment
\[
(t,c) = \xymatrix@C-30pt{
(\bullet&\bullet&\bullet)&(\circ&\circ)&(\bullet&\bullet&\bullet&\bullet)&(\bullet&\bullet&\bullet&\bullet&\bullet)&(\circ&\circ&\circ)&(\circ)
}
\]
then the set $E_1^{\varepsilon}(t,c)$ (where $b = 1$) will contain the following words (which have been parenthesized for clarity), among many others.
\[
\begin{array}{c}
\xymatrix@C-30pt{
(\mathtt{A}&\mathtt{G}&\varepsilon)&(\mathtt{T}&\mathtt{C}&\mathtt{A}&\mathtt{A})&(\mathtt{T}&\mathtt{A}&\mathtt{G}&\mathtt{G}&\varepsilon);
}\\
\xymatrix@C-30pt{
(\mathtt{G}&\mathtt{T}&\varepsilon)&(\varepsilon&\varepsilon&\varepsilon&\mathtt{C})&(\mathtt{A}&\mathtt{G}&\mathtt{T}&\mathtt{A}&\mathtt{C});
}\\
\xymatrix@C-30pt{
(\mathtt{T}&\mathtt{A}&\mathtt{A})&(\mathtt{G}&\mathtt{A}&\mathtt{T}&\mathtt{C})&(\mathtt{A}&\mathtt{G}&\mathtt{T}&\mathtt{T}&\mathtt{T});
}\\
\textrm{etc.}\\
\end{array}
\]
\end{example}

\begin{example}[Morphisms]\label{exa:morphisms_as_inclusion_of_words}
Suppose that $\Omega$ denotes the Boolean pre-ordered set $\{0\leq 1\}$ and let $(E,\varepsilon)$ be the pointed set $\{\mathtt{A},\mathtt{C},\mathtt{G},\mathtt{T},\varepsilon\}$. If we consider the morphism of segments given below, in which we use adequate labeling to show how the first segment is included in the second one,
\[
\xymatrix@C-30pt{
(\mathop{\bullet}\limits^1&\mathop{\bullet}\limits^2&\mathop{\bullet}\limits^3)&(\mathop{\circ}\limits^4&\mathop{\circ}\limits^5)&(\mathop{\bullet}\limits^6&\mathop{\bullet}\limits^7&\mathop{\bullet}\limits^8&\mathop{\bullet}\limits^9)&(\mathop{\bullet}\limits^{10\,}&\mathop{\bullet}\limits^{11})
} \to \xymatrix@C-30pt{
(\mathop{\bullet}\limits^1&\mathop{\bullet}\limits^2&\mathop{\bullet}\limits^3&\mathop{\bullet}\limits^{\ast}&\mathop{\bullet}\limits^{\ast})&(\mathop{\circ}\limits^4&\mathop{\circ}\limits^5&\mathop{\circ}\limits^{\ast})&(\mathop{\bullet}\limits^6&\mathop{\bullet}\limits^7&\mathop{\bullet}\limits^8&\mathop{\bullet}\limits^9)&(\mathop{\bullet}\limits^{\ast})&(\mathop{\circ}\limits^{10\,}&\mathop{\circ}\limits^{11})
}
\]
then the image of the previous arrow via $E_1^{\varepsilon}$ is a function whose mappings rules look as follows.
\[
\begin{array}{ccc}
\xymatrix@C-30pt{
(\mathtt{A}&\mathtt{G}&\varepsilon)&(\mathtt{T}&\mathtt{C}&\mathtt{A}&\mathtt{A})&(\mathtt{G}&\mathtt{C})
} &\mapsto &
\xymatrix@C-30pt{
(\mathtt{A}&\mathtt{G}&\varepsilon&\varepsilon&\varepsilon)&(\mathtt{T}&\mathtt{C}&\mathtt{A}&\mathtt{A})&(\varepsilon);
}\\
\xymatrix@C-30pt{
(\mathtt{G}&\mathtt{T}&\varepsilon)&(\varepsilon&\varepsilon&\varepsilon&\mathtt{C})&(\mathtt{T}&\mathtt{A})
} &\mapsto &
\xymatrix@C-30pt{
(\mathtt{G}&\mathtt{T}&\varepsilon&\varepsilon&\varepsilon)&(\varepsilon&\varepsilon&\varepsilon&\mathtt{C})&(\varepsilon);
}\\
\xymatrix@C-30pt{
(\mathtt{T}&\mathtt{A}&\mathtt{A})&(\mathtt{G}&\mathtt{A}&\mathtt{T}&\mathtt{C})&(\mathtt{A}&\mathtt{A})
} &\mapsto& 
\xymatrix@C-30pt{
(\mathtt{T}&\mathtt{A}&\mathtt{A}&\varepsilon&\varepsilon)&(\mathtt{G}&\mathtt{A}&\mathtt{T}&\mathtt{C})&(\varepsilon);
}\\
&\textrm{etc.}&\\
\end{array}
\]
Note that if one restricts oneself to morphisms in $\mathbf{Seg}(\Omega)$ that only insert new nodes and do not turn any black node into white ones, then the mappings associated with the images of such morphisms can be seen as gap insertion operations. As seen in section \ref{ssec:Main_example}, these operations are used in sequence alignment algorithms to compare sequences of different lengths together.
\[
\xymatrix@C-30pt{
(\mathop{\bullet}\limits^1&\mathop{\bullet}\limits^2&\mathop{\bullet}\limits^3)&(\mathop{\circ}\limits^4&\mathop{\circ}\limits^5)&(\mathop{\bullet}\limits^6&\mathop{\bullet}\limits^7&\mathop{\bullet}\limits^8&\mathop{\bullet}\limits^9)&(\mathop{\bullet}\limits^{10\,}&\mathop{\bullet}\limits^{11})
} \to \xymatrix@C-30pt{
(\mathop{\bullet}\limits^1&\mathop{\bullet}\limits^2&\mathop{\bullet}\limits^3)&(\mathop{\circ}\limits^4&\mathop{\circ}\limits^5&\mathop{\circ}\limits^{\ast})&(\mathop{\bullet}\limits^6&\mathop{\bullet}\limits^7&\mathop{\bullet}\limits^{\ast}&\mathop{\bullet}\limits^{\ast}&\mathop{\bullet}\limits^8&\mathop{\bullet}\limits^9)&(\mathop{\bullet}\limits^{\ast}&\mathop{\bullet}\limits^{10\,}&\mathop{\bullet}\limits^{11})
}
\]
\[
\begin{array}{ccc}
\xymatrix@C-30pt{
(\mathtt{G}&\mathtt{A}&\mathtt{C})&(\mathtt{A}&\mathtt{T}&\mathtt{T}&\mathtt{C})&(\mathtt{C}&\mathtt{T})
} &\mapsto &
\xymatrix@C-30pt{
(\mathtt{G}&\mathtt{A}&\mathtt{C})&(\mathtt{A}&\mathtt{T}&\varepsilon&\varepsilon&\mathtt{T}&\mathtt{C})&(\varepsilon&\mathtt{C}&\mathtt{T});
}\\
&\textrm{etc.}&\\
\end{array}
\]
\end{example}

\begin{proposition}\label{prop:pedigrad_representable}
For every domain $[n_1]$, the restriction of the functor $E_{b}^{\varepsilon}:\mathbf{Seg}(\Omega) \to \mathbf{Set}$ on $\mathbf{Seg}(\Omega\,|\,n_1)$ is isomorphic to the functor $\mathbf{Set}(\mathsf{Tr}_b(\_),E):\mathbf{Seg}(\Omega\,|\,n_1) \to \mathbf{Set}$. In other words, the following diagram commutes up to an isomorphism of functors.
\[
\xymatrix{
\mathbf{Seg}(\Omega\,|\,n_1)\ar[r]^{\subseteq}\ar[d]_{\mathsf{Tr}_b}&\mathbf{Seg}(\Omega)\ar[d]^{E_{b}^{\varepsilon}}\\
\mathbf{Set}^{\mathrm{op}}\ar[r]_{\mathbf{Set}(\_,E)}&\mathbf{Set}
}
\]
\end{proposition}
\begin{proof}
Note that the following series of isomorphisms hold on $\mathbf{Seg}(\Omega\,|\,n_1)$.
\begin{align*}
E_{b}^{\varepsilon}(\_) & = \mathbf{Set}_{\ast}\Big(\mathsf{Tr}_b^{\ast}(\_),(E,\varepsilon)\Big)&\\
& = \mathbf{Set}_{\ast}\Big(\mathbb{F}\mathsf{Tr}_b(\_),(E,\varepsilon)\Big)&(\textrm{Proposition }\ref{prop:Tr_ast_on_Chr_Omega_t_equals_F_Tr})\\
& \cong \mathbf{Set}\Big(\mathsf{Tr}_b(\_),\mathbb{U}(E,\varepsilon)\Big)&(\mathbb{F} \dashv \mathbb{U})\\
& = \mathbf{Set}\Big(\mathsf{Tr}_b(\_),E\Big)&(\textrm{Definition of }\mathbb{U})
\end{align*}
Since these isomorphisms are natural on $\mathbf{Seg}(\Omega\,|\,n_1)$, the statement follows.
\end{proof}

\subsection{Sequence alignments}
In this section, we use the functors defined in section \ref{sec:Example_pedigrads_in_sets} to formalize the concept of sequence alignment. The examples given in this section mainly focus on addressing our main example presented in section \ref{ssec:Main_example}.

\begin{definition}[Alignment specification]\label{def:alignment_specification}
We shall speak of an \emph{alignment specification} to refer to a wide span (Definition \ref{def:wide_spans}) in the category $\mathbf{pOrd}$ of pre-ordered sets.
\end{definition}

\begin{example}[Alignment specification]\label{exa:Alignment_specification}
Let $(\Omega,\preceq)$ be the Boolean pre-ordered set $\{0 \leq 1\}$. If we use the notations of Example \ref{exa:preparation_example}, then the following collection of morphisms in $\mathbf{pOrd}$ defines an alignment specification.
\begin{equation}\label{eq:alignment_specification}
\{\pi_i:\Omega^{\times 4} \to \Omega \}_{i \in 
\{\mathtt{a},\mathtt{b},\mathtt{c},\mathtt{d}\}}
\end{equation}
Here, the discrete category $\{\mathtt{a},\mathtt{b},\mathtt{c},\mathtt{d}\}$ is implicitly ordered with respect to the alphabetic order and hence makes (\ref{eq:alignment_specification}) a wide span as defined in Definition \ref{def:wide_spans}.

In general, alignment specifications do not necessarily need to be universal cones and the codomains of the arrows do not need be equal either. For instance, the following pair of wide spans define two valid alignment specifications.
\[
\left\{
\begin{array}{llll}
\pi_{\mathtt{a}}:&\Omega^{\times 4} &\to& \Omega,\\
\pi_{\mathtt{c}}:&\Omega^{\times 4}& \to& \Omega,\\
\pi_{\mathtt{d}}:&\Omega^{\times 4}& \to& \Omega\\
\end{array}
\right\}
\quad\quad\quad\quad
\left\{
\begin{array}{llll}
\pi_{\mathtt{a}} \times \pi_{\mathtt{b}}:&\Omega^{\times 4} &\to& \Omega^{\times 2},\\
\pi_{\mathtt{c}}:&\Omega^{\times 4}& \to& \Omega,\\
\pi_{\mathtt{d}}:&\Omega^{\times 4}& \to& \Omega\\
\end{array}
\right\}
\]
In the next article \cite{Recomb}, we will use alignment specifications made of identity morphisms.
\end{example}

\begin{definition}[Aligned pedigrads]\label{def:Aligned_pedigrads}
Let 
$
\mathbf{A} = \{f_i:(\Omega,\preceq) \to (\Omega_i,\preceq_i)\}_{i \in A}
$ be an alignment specification
and $b$ be an element in $\Omega$. We denote by $\mathbf{A}E_b^{\varepsilon}$ the functor $\mathbf{Seg}(\Omega) \to \mathbf{Set}$ resulting from the composition of the three functors given in (\ref{eq:aligned_pedigrad}), where 
\begin{itemize}
\item[-] the rightmost functor is the obvious Cartesian functor of $\mathbf{Set}$;
\item[-] the middle functor is the Cartesian product of the functors $E_{f_i(b)}^{\varepsilon}:\mathbf{Seg}(\Omega_i) \to \mathbf{Set}$;
\item[-] and the leftmost functor is the product adjoint of the cone induced by the image of $\mathbf{A}$ via $\mathbf{Seg}$ (section \ref{ssec:relating_categories_of_segments}).
\end{itemize}
\begin{equation}\label{eq:aligned_pedigrad}
\xymatrix@C+35pt{
\mathbf{Seg}(\Omega) \ar[r]^-{ (\mathbf{Seg}(f_i))_{i \in A}} & \prod_{i \in A} \mathbf{Seg}(\Omega_i) \ar[r]^-{\prod_{i \in A} E_{f_i(b)}^{\varepsilon}}& \prod_{i \in A} \mathbf{Set} \ar[r]^-{\times} & \mathbf{Set}
}
\end{equation}
Such a functor will be called the \emph{alignment} of $E_b^{\varepsilon}$ on $\mathbf{A}$.
\end{definition}

\begin{example}[Aligned pedigrads]\label{exa:Aligned_pedigrads}
Let $(\Omega,\preceq)$ denote the Boolean pre-ordered $\{0 \leq 1\}$ and let $\mathbf{A}$ denote the alignment specification given in (\ref{eq:alignment_specification}). As usual, we shall let $(E,\varepsilon)$ denote the pointed set $\{\mathtt{A},\mathtt{C},\mathtt{G},\mathtt{T},\varepsilon\}$. For any given segment $(t,c)$ in $\mathbf{Seg}(\Omega^{\times 4})$ and element $b \in \Omega^{\times 4}$, the set $\mathbf{A}E_b^{\varepsilon}(t,c)$ is equal to the following Cartesian product of sets.
\[
E_{\pi_{\mathtt{a}}(b)}^{\varepsilon}(t,\pi_{\mathtt{a}} \circ c) \times E_{\pi_{\mathtt{b}}(b)}^{\varepsilon}(t,\pi_{\mathtt{b}} \circ c) \times E_{\pi_{\mathtt{c}}(b)}^{\varepsilon}(t,\pi_{\mathtt{c}} \circ c) \times E_{\pi_{\mathtt{d}}(b)}^{\varepsilon}(t,\pi_{\mathtt{d}} \circ c)
\]
The following table illustrates what the elements of the previous set look like for different segments $(t,c)$ in $\mathbf{Seg}(\Omega^{\times 4})$ and a fixed value $b$ in $\Omega^{\times 4}$. As usual, parentheses are added for clarity and the underscore symbols\footnote{These underscore symbols should not be confused with the dash symbols that is sometimes used in bioinformatics to denote a substitution. Note that, in our case, the symbol $\varepsilon$ already plays the role of the dash symbols.} are only used to represent spaces ``to be filled out''.

\[
\begin{array}{|c|c|c|}
\hline
\cellcolor[gray]{0.8}b&\cellcolor[gray]{0.8}(t,c)&\cellcolor[gray]{0.8}\mathbf{A}E_b^{\varepsilon}(t,c)\\
\hline
\multirow{9}{*}{(1,1,1,1)} &
\!\xymatrix@C-30pt{([\mathtt{1\separ1\separ1\separ1}]&[\mathtt{1\separ1\separ1\separ1}])&([\mathtt{1\separ1\separ1\separ1}]&[\mathtt{1\separ1\separ1\separ1}]&[\mathtt{1\separ1\separ1\separ1}])}\! & 
\!\!\!\!\begin{array}{c}
\xymatrix@-30pt{
(\mathtt{A}&\mathtt{G})&(\mathtt{C}&\mathtt{G}&\mathtt{T})\\
(\mathtt{A}&\mathtt{T})&(\mathtt{T}&\mathtt{C}&\mathtt{G})\\
(\mathtt{C}&\varepsilon)&(\mathtt{A}&\mathtt{T}&\mathtt{G})\\
(\mathtt{A}&\mathtt{T})&(\mathtt{G}&\mathtt{G}&\mathtt{G})\\
}
\end{array}\!\!\!\!;\!\!\!\!\!\!
\quad
\begin{array}{c}
\xymatrix@-30pt{
(\mathtt{G}&\mathtt{T})&(\mathtt{A}&\mathtt{A}&\mathtt{T})\\
(\mathtt{C}&\mathtt{G})&(\mathtt{G}&\mathtt{T}&\varepsilon)\\
(\mathtt{A}&\mathtt{C})&(\mathtt{T}&\mathtt{T}&\mathtt{G})\\
(\varepsilon&\varepsilon)&(\mathtt{T}&\mathtt{G}&\mathtt{C})\\
}
\end{array}\!\!\!\!;\!\!\!\!\!\!
\quad
\begin{array}{c}
\xymatrix@-30pt{
(\mathtt{C}&\mathtt{A})&(\mathtt{A}&\mathtt{A}&\mathtt{C})\\
(\mathtt{C}&\mathtt{C})&(\varepsilon&\mathtt{A}&\mathtt{C})\\
(\mathtt{C}&\mathtt{T})&(\varepsilon&\mathtt{C}&\mathtt{A})\\
(\mathtt{A}&\mathtt{C})&(\mathtt{T}&\mathtt{T}&\mathtt{G})\\
}
\end{array}\!\!\!\!; \!\!\!
\quad
\textrm{etc.}\\
\cline{2-3}
&
\!\xymatrix@C-30pt{([\mathtt{1\separ0\separ1\separ0}]&[\mathtt{1\separ0\separ1\separ0}])&([\mathtt{0\separ1\separ1\separ0}]&[\mathtt{0\separ1\separ1\separ0}]&[\mathtt{0\separ1\separ1\separ0}])}\! &

\!\!\!\!\begin{array}{c}
\xymatrix@-30pt{
(\mathtt{A}&\mathtt{G})&\_&\_&\_\\
\_&\_&(\mathtt{T}&\mathtt{C}&\mathtt{G})\\
(\mathtt{C}&\varepsilon)&(\mathtt{A}&\mathtt{T}&\mathtt{G})\\
}
\end{array}\!\!\!\!;\!\!\!\!\!\!
\quad
\begin{array}{c}
\xymatrix@-30pt{
(\mathtt{G}&\mathtt{T})&\_&\_&\_\\
\_&\_&(\mathtt{G}&\mathtt{T}&\varepsilon)\\
(\mathtt{A}&\mathtt{C})&(\mathtt{T}&\mathtt{T}&\mathtt{G})\\
}
\end{array}\!\!\!\!;\!\!\!\!\!\!
\quad
\begin{array}{c}
\xymatrix@-30pt{
(\mathtt{C}&\mathtt{A})&\_&\_&\_\\
\_&\_&(\varepsilon&\mathtt{A}&\mathtt{C})\\
(\mathtt{C}&\mathtt{T})&(\varepsilon&\mathtt{C}&\mathtt{A})\\
}
\end{array}\!\!\!\!; \!\!\!
\quad
\textrm{etc.}\\
\cline{2-3}
&
\!\xymatrix@C-30pt{([\mathtt{0\separ0\separ1\separ1}]&[\mathtt{0\separ0\separ1\separ1}])&([\mathtt{0\separ0\separ1\separ1}]&[\mathtt{0\separ0\separ1\separ1}]&[\mathtt{0\separ0\separ1\separ1}])}\!
 & 
\!\!\!\!\begin{array}{c}
\xymatrix@-30pt{
(\mathtt{C}&\varepsilon)&(\mathtt{A}&\mathtt{T}&\mathtt{G})\\
(\mathtt{A}&\mathtt{T})&(\mathtt{G}&\mathtt{G}&\mathtt{G})\\
}
\end{array}\!\!\!\!;\!\!\!\!\!\!
\quad
\begin{array}{c}
\xymatrix@-30pt{
(\mathtt{A}&\mathtt{C})&(\mathtt{T}&\mathtt{T}&\mathtt{G})\\
(\varepsilon&\varepsilon)&(\mathtt{T}&\mathtt{G}&\mathtt{C})\\
}
\end{array}\!\!\!\!;\!\!\!\!\!\!
\quad
\begin{array}{c}
\xymatrix@-30pt{
(\mathtt{C}&\mathtt{T})&(\varepsilon&\mathtt{C}&\mathtt{A})\\
(\mathtt{A}&\mathtt{C})&(\mathtt{T}&\mathtt{T}&\mathtt{G})\\
}
\end{array}\!\!\!\!;\!\!\!
\quad
\textrm{etc.}\\
\cline{2-3}
&
\!\xymatrix@C-30pt{([\mathtt{1\separ1\separ0\separ0}]&[\mathtt{1\separ1\separ0\separ0}])&([\mathtt{1\separ1\separ0\separ0}]&[\mathtt{1\separ1\separ0\separ0}]&[\mathtt{1\separ1\separ0\separ0}])}\!
 & 
\!\!\!\!\begin{array}{c}
\xymatrix@-30pt{
(\mathtt{A}&\mathtt{G})&(\mathtt{C}&\mathtt{G}&\mathtt{T})\\
(\mathtt{A}&\mathtt{T})&(\mathtt{T}&\mathtt{C}&\mathtt{G})\\
}
\end{array}\!\!\!\!;\!\!\!\!\!\!
\quad
\begin{array}{c}
\xymatrix@-30pt{
(\mathtt{G}&\mathtt{T})&(\mathtt{A}&\mathtt{A}&\mathtt{T})\\
(\mathtt{C}&\mathtt{G})&(\mathtt{G}&\mathtt{T}&\varepsilon)\\
}
\end{array}\!\!\!\!;\!\!\!\!\!\!
\quad
\begin{array}{c}
\xymatrix@-30pt{
(\mathtt{C}&\mathtt{A})&(\mathtt{A}&\mathtt{A}&\mathtt{C})\\
(\mathtt{C}&\mathtt{C})&(\varepsilon&\mathtt{A}&\mathtt{C})\\
}
\end{array}\!\!\!\!;\!\!\!
\quad
\textrm{etc.}\\
\hline
\end{array}
\]
We can see that $\mathbf{A}E_b^{\varepsilon}(t,c)$ contains what we would like to understand as sequence alignments. It would therefore be natural to try to model our example, given in section \ref{ssec:Main_example}, relatively to the functor $\mathbf{A}E_{b}^{\varepsilon}$. In what follows, we define the concept of sequence alignment relative to alignments of functors $E_{b}^{\varepsilon}$ for some pointed set $(E,\varepsilon)$, pre-ordered set $\Omega$ and element $b \in \Omega$.
\end{example}

The following definition formalizes the concept of sequence alignment in terms of a subcategory, a functor and a natural transformation. In Example \ref{exa:sequence_alignments}, we will see how one can use this concept to reason about the problem given in section \ref{ssec:Main_example}.

\begin{definition}[Sequence alignment functors]\label{def:sequence_alignment}
Let $\mathbf{A}$ be an alignment specification as given in Definition \ref{def:Aligned_pedigrads} and $b$ be an element in $\Omega$. We define a \emph{sequence alignment functor over $\mathbf{A}E_b^{\varepsilon}$} as a triple $(\iota,T,\sigma)$ where $\iota$ is an inclusion functor $\iota:B \to \mathbf{Seg}(\Omega)$, $T$ is a functor $B \to \mathbf{Set}$ and $\sigma$ is a natural monomorphism $T \Rightarrow \mathbf{A}E_b^{\varepsilon} \circ \iota$.
\end{definition}

\begin{convention}[Notations]\label{conv:segment_notations}
In the sequel, it will be convenient to have short notations for the segments of the subcategory $B \hookrightarrow \mathbf{Seg}(\Omega)$ associated with a sequence alignment functor. In the context of the present article, most of our segments will have trivial topologies -- given by the canonical surjection $!_{[n]}:[n] \to [1]$ -- and will hence be of the following form for some color $b \in \Omega$.
\[
\xymatrix@C-30pt{
(b&b&b&\dots&b&b)
}
\]
In this respect, we shall denote any segment whose topology $t$ is of the form $!_{[n]}:[n] \to [1]$ and whose function $c:[1] \to \Omega$ picks out an element $b \in \Omega$ as a pair $(!_{[n]},b)$.
\end{convention}

\begin{example}[Sequence alignment functors]\label{exa:sequence_alignments}
The present example is a continuation of Example \ref{exa:Aligned_pedigrads} and  aims to illustrate the use of Definition \ref{def:sequence_alignment} in the context of our problem introduced in section \ref{ssec:Main_example}. We shall therefore use the same notations as those used in Example \ref{exa:Aligned_pedigrads}. 

As mentioned at the end of Example \ref{exa:Aligned_pedigrads}, we want to model the example of section \ref{ssec:Main_example} relative to the functor $\mathbf{A}E_{b}^{\varepsilon}$ defined in the example. The idea is to pick out, via the concept of sequence alignment functors introduced in Definition \ref{def:sequence_alignment}, the pairwise sequence alignments living in the images of $\mathbf{A}E_{b}^{\varepsilon}$ outputted by the dynamic programming algorithm presented in section \ref{ssec:Main_example}. In this respect, let us compute the sequence alignments of every pair of individuals given in section \ref{ssec:Main_example} by following the method described therein. Doing so, we obtain the following table of pairwise sequence alignments.

\[
\begin{array}{|l|c|}
\hline
\cellcolor[gray]{0.8}\textbf{Pairs}&\cellcolor[gray]{0.8}\textbf{Sequence alignments}\\
\hline
\cellcolor[gray]{0.9}
\!\!\begin{array}{l}
\texttt{Anne}\\
\texttt{Bob}
\end{array}
&
\begin{array}{l}
\mathtt{ACCGACTG}\\
\mathtt{A\varepsilon CATCTG}
\end{array}\,
\begin{array}{l}
\mathtt{ACCGACTG}\\
\mathtt{ACA\varepsilon TCTG}
\end{array}\,
\begin{array}{l}
\mathtt{ACCGA\varepsilon CTG}\\
\mathtt{A\varepsilon C\varepsilon ATCTG}
\end{array}\,
\begin{array}{l}
\mathtt{ACCGACTG}\\
\mathtt{ACAT\varepsilon CTG}
\end{array}
\\
\hline
\cellcolor[gray]{0.9}
\!\!\begin{array}{l}
\texttt{Anne}\\
\texttt{Craig}
\end{array}
&
\begin{array}{l}
\mathtt{ACCGACTG}\\
\mathtt{ACCGTC\varepsilon A}
\end{array}\,
\begin{array}{l}
\mathtt{ACCGACTG}\\
\mathtt{ACCGTCA\varepsilon}
\end{array}
\\
\hline
\cellcolor[gray]{0.9}
\!\!\begin{array}{l}
\texttt{Anne}\\
\texttt{Doug}
\end{array}
&
\begin{array}{l}
\mathtt{ACCGACTG}\\
\mathtt{A\varepsilon CTACTG}
\end{array}\,
\begin{array}{l}
\mathtt{ACCGACTG}\\
\mathtt{ACT\varepsilon ACTG}
\end{array}
\\
\hline
\cellcolor[gray]{0.9}
\!\!\begin{array}{l}
\texttt{Bob}\\
\texttt{Craig}
\end{array}
&
\begin{array}{l}
\mathtt{A\varepsilon CATCTG}\\
\mathtt{ACCGTC\varepsilon A}
\end{array}\,
\begin{array}{l}
\mathtt{ACA\varepsilon TCTG}\\
\mathtt{ACCGTC\varepsilon A}
\end{array}\,
\begin{array}{l}
\mathtt{A\varepsilon CATCTG}\\
\mathtt{ACCGTCA\varepsilon}
\end{array}\,
\begin{array}{l}
\mathtt{ACA\varepsilon TCTG}\\
\mathtt{ACCGTCA\varepsilon}
\end{array}
\\
\hline
\cellcolor[gray]{0.9}
\!\!\begin{array}{l}
\texttt{Bob}\\
\texttt{Doug}
\end{array}
&
\begin{array}{l}
\mathtt{ACATCTG}\\
\mathtt{ACTACTG}
\end{array}\,
\begin{array}{l}
\mathtt{ACAT\varepsilon CTG}\\
\mathtt{AC\varepsilon TACTG}
\end{array}\,
\begin{array}{l}
\mathtt{AC\varepsilon ATCTG}\\
\mathtt{ACTA\varepsilon CTG}
\end{array}
\\
\hline
\cellcolor[gray]{0.9}
&
\begin{array}{l}
\mathtt{ACCGTCA}\\
\mathtt{ACTACTG}
\end{array}\,
\begin{array}{l}
\mathtt{ACCGTC\varepsilon A}\\
\mathtt{A\varepsilon CTACTG}
\end{array}\,
\begin{array}{l}
\mathtt{ACCGTC\varepsilon A}\\
\mathtt{ACT\varepsilon ACTG}
\end{array}\,
\begin{array}{l}
\mathtt{ACCGT\varepsilon C\varepsilon A}\\
\mathtt{A\varepsilon C\varepsilon TACTG}
\end{array}\,
\begin{array}{l}
\mathtt{ACCGTC\varepsilon A}\\
\mathtt{ACTA\varepsilon CTG}
\end{array}\,
\begin{array}{l}
\mathtt{AC\varepsilon CGTCA}\\
\mathtt{ACTACT\varepsilon G}
\end{array}\\
\cline{2-2}
\cellcolor[gray]{0.9}
\!\!\begin{array}{l}
\texttt{Craig}\\
\texttt{Doug}
\end{array}&
\begin{array}{l}
\mathtt{ACC\varepsilon GTCA}\\
\mathtt{ACTACT\varepsilon G}
\end{array}\,
\begin{array}{l}
\mathtt{ACCG\varepsilon TCA}\\
\mathtt{ACTACT\varepsilon G}
\end{array}\,
\begin{array}{l}
\mathtt{AC\varepsilon \varepsilon CGTCA}\\
\mathtt{ACTAC\varepsilon T\varepsilon G}
\end{array}\,
\begin{array}{l}
\mathtt{ACCGTCA\varepsilon }\\
\mathtt{A\varepsilon CTACTG}
\end{array}\,
\begin{array}{l}
\mathtt{ACCGTCA\varepsilon }\\
\mathtt{ACT\varepsilon ACTG}
\end{array}\,
\begin{array}{l}
\mathtt{ACCGT\varepsilon CA\varepsilon }\\
\mathtt{A\varepsilon C\varepsilon TACTG}
\end{array}
\\
\cline{2-2}
\cellcolor[gray]{0.9}&
\begin{array}{l}
\mathtt{ACCGTCA\varepsilon }\\
\mathtt{ACTA\varepsilon CTG}
\end{array}\,
\begin{array}{l}
\mathtt{AC\varepsilon CGTCA}\\
\mathtt{ACTACTG\varepsilon }
\end{array}\,
\begin{array}{l}
\mathtt{ACC\varepsilon GTCA}\\
\mathtt{ACTACTG\varepsilon }
\end{array}\,
\begin{array}{l}
\mathtt{ACCG\varepsilon TCA}\\
\mathtt{ACTACTG\varepsilon }
\end{array}\,
\begin{array}{l}
\mathtt{AC\varepsilon \varepsilon CGTCA}\\
\mathtt{ACTAC\varepsilon TG\varepsilon }
\end{array}\\
\hline
\end{array}
\]

If we now take $b$ to be equal to the element $(1,1,1,1)$ in $\Omega^{\times 4}$, the previous table can reasonably be seen as a `part' of the functor $\mathbf{A}E_b^{\varepsilon}$ by interpreting each pairwise sequence alignment given above as an element in one of the images of $\mathbf{A}E_b^{\varepsilon}$ (see the table given in Example \ref{exa:Aligned_pedigrads}). In this example, we shall implement this `part' by considering the sequence alignment functor $(B,\iota,T,\sigma)$ whose subcategory $B \hookrightarrow \mathbf{Seg}(\Omega)$ is the union of the full subcategories of $\mathbf{Seg}(\Omega^{\times 4}\,|\,7)$, $\mathbf{Seg}(\Omega^{\times 4}\,|\,8)$ and $\mathbf{Seg}(\Omega^{\times 4}\,|\,9)$ that contain the following segments:
\[
\begin{array}{l@{\quad\quad}l@{\quad\quad}l@{\quad\quad}l@{\quad\quad}l}
(!_{[7]},[\mathtt{0\separ1\separ0\separ1}])&(!_{[7]},[\mathtt{0\separ0\separ1\separ1}])&(!_{[7]},[\mathtt{0\separ0\separ0\separ1}])&(!_{[8]},[\mathtt{1\separ1\separ0\separ0}])&(!_{[8]},[\mathtt{1\separ0\separ1\separ0}])\\
(!_{[8]},[\mathtt{1\separ0\separ0\separ1}])&(!_{[8]},[\mathtt{0\separ1\separ1\separ0}])&(!_{[8]},[\mathtt{0\separ1\separ0\separ1}])&(!_{[8]},[\mathtt{0\separ0\separ1\separ1}])&(!_{[8]},[\mathtt{0\separ0\separ0\separ1}])\\
(!_{[8]},[\mathtt{1\separ0\separ0\separ0}])&(!_{[8]},[\mathtt{0\separ1\separ0\separ0}])&(!_{[8]},[\mathtt{0\separ0\separ1\separ0}])&(!_{[9]},[\mathtt{0\separ0\separ1\separ1}])&(!_{[9]},[\mathtt{1\separ1\separ0\separ0}])
\end{array}
\]
and whose monomorphism $T \Rightarrow \mathbf{A}E_b^{\varepsilon} \circ \iota$ picks out the pairwise sequence alignments shown in the previous table. Before making this last statement more precise, let us explain how the colors and the domains of the segments of $B$ (shown above) will be used to organize the pairwise sequence alignments contained in the images of $T$. First, as already suggested in Example \ref{exa:preparation_example}, we want to use a tuple $(\mathtt{x_{a}},\mathtt{x_{b}},\mathtt{x_{c}},\mathtt{x_{d}})$ in $\Omega^{\times 4}$ to specify whether an individual is included in a sequence alignment or not by setting the variable indexed by the initial of the individual to 1 or 0, respectively. For instance, setting $\mathtt{x_{a}}$ to 1 would mean that \texttt{Anne} is part of the alignment computation. Second, we want to make the cardinality of the domains of the segments match the length of the pairwise sequence alignments with which they are associated. Thus, we could decide to encode the row of the previous table comparing $\texttt{Anne}$ and $\texttt{Bob}$ by taking the following images for $T$.
\[
\begin{array}{c}
T(!_{[8]},[\mathtt{1\separ1\separ0\separ0}]) = 
\left\{
\begin{array}{l}\!\!\!\!\!\!
\begin{array}{l}
\mathtt{ACCGACTG}\\
\mathtt{A\varepsilon CATCTG}
\end{array},\,
\begin{array}{l}
\mathtt{ACCGACTG}\\
\mathtt{ACA\varepsilon TCTG}
\end{array},\,
\begin{array}{l}
\mathtt{ACCGACTG}\\
\mathtt{ACAT\varepsilon CTG}
\end{array}\!\!\!\!\!\!
\end{array}
\right\}
\end{array}
\]

\[
\begin{array}{c}
T(!_{[9]},[\mathtt{1\separ1\separ0\separ0}]) = 
\left\{
\begin{array}{l}\!\!\!\!\!\!
\begin{array}{l}
\mathtt{ACCGA\varepsilon CTG}\\
\mathtt{A\varepsilon C\varepsilon ATCTG}
\end{array}\!\!\!\!\!\!
\end{array}
\right\}
\end{array}
\]
Similarly, the other images of $T$ on segments containing exactly two symbols $\mathtt{1}$ could be taken as follows (the specification of the second last image, of cardinality 12, is left to the reader).
\[
\begin{array}{|l|}
\hline
\multicolumn{1}{|c|}{\cellcolor[gray]{0.8}T}\\
\hline
T(!_{[8]},[\mathtt{1\separ0\separ1\separ0}]) = \left\{
\begin{array}{l}\!\!\!\!\!\!
\begin{array}{l}
\mathtt{ACCGACTG}\\
\mathtt{ACCGTC\varepsilon A}
\end{array}\!;\,
\begin{array}{l}
\mathtt{ACCGACTG}\\
\mathtt{ACCGTCA\varepsilon}
\end{array}\!\!\!\!\!\!
\end{array}
\right\}
\\
T(!_{[8]},[\mathtt{1\separ0\separ0\separ1}]) = \left\{
\begin{array}{l}\!\!\!\!\!\!
\begin{array}{l}
\mathtt{ACCGACTG}\\
\mathtt{A\varepsilon CTACTG}
\end{array}\!;\,
\begin{array}{l}
\mathtt{ACCGACTG}\\
\mathtt{ACT\varepsilon ACTG}
\end{array}\!\!\!\!\!\!
\end{array}
\right\}
\\
T(!_{[8]},[\mathtt{0\separ1\separ1\separ0}]) = \left\{
\begin{array}{l}\!\!\!\!\!\!
\begin{array}{l}
\mathtt{A\varepsilon CATCTG}\\
\mathtt{ACCGTC\varepsilon A}
\end{array}\!;\,
\begin{array}{l}
\mathtt{ACA\varepsilon TCTG}\\
\mathtt{ACCGTC\varepsilon A}
\end{array}\!;\,
\begin{array}{l}
\mathtt{A\varepsilon CATCTG}\\
\mathtt{ACCGTCA\varepsilon}
\end{array}\!;\,
\begin{array}{l}
\mathtt{ACA\varepsilon TCTG}\\
\mathtt{ACCGTCA\varepsilon}
\end{array}\!\!\!\!\!\!
\end{array}
\right\}
\\
T(!_{[7]},[\mathtt{0\separ1\separ0\separ1}]) = \left\{
\begin{array}{l}\!\!\!\!\!\!
\begin{array}{l}
\mathtt{ACATCTG}\\
\mathtt{ACTACTG}
\end{array}\!\!\!\!\!\!
\end{array}
\right\}
\\
T(!_{[8]},[\mathtt{0\separ1\separ0\separ1}]) = \left\{
\begin{array}{l}\!\!\!\!\!\!
\begin{array}{l}
\mathtt{ACAT\varepsilon CTG}\\
\mathtt{AC\varepsilon TACTG}
\end{array}\!;\,
\begin{array}{l}
\mathtt{AC\varepsilon ATCTG}\\
\mathtt{ACTA\varepsilon CTG}
\end{array}\!\!\!\!\!\!
\end{array}
\right\}
\\
T(!_{[7]},[\mathtt{0\separ0\separ1\separ1}]) = \left\{
\begin{array}{l}\!\!\!\!\!\!
\begin{array}{l}
\mathtt{ACCGTCA}\\
\mathtt{ACTACTG}
\end{array}\!\!\!\!\!\!
\end{array}
\right\} \\
T(!_{[8]},[\mathtt{0\separ0\separ1\separ1}]) = \left\{
\begin{array}{l}\!\!\!\!\!\!
\begin{array}{l}
\mathtt{ACCGTC\varepsilon A}\\
\mathtt{A\varepsilon CTACTG}
\end{array}\!;\,
\begin{array}{l}
\dots
\end{array}\!\!\!\!\!\!
\end{array}
\right\} \\
T(!_{[9]},[\mathtt{0\separ0\separ1\separ1}]) = \left\{
\begin{array}{l}\!\!\!\!\!\!
\begin{array}{l}
\mathtt{ACCGT\varepsilon C\varepsilon A}\\
\mathtt{A\varepsilon C\varepsilon TACTG}
\end{array}\!;\,
\begin{array}{l}
\mathtt{AC\varepsilon \varepsilon CGTCA}\\
\mathtt{ACTAC\varepsilon T\varepsilon G}
\end{array}\!;\,
\begin{array}{l}
\mathtt{ACCGT\varepsilon CA\varepsilon }\\
\mathtt{A\varepsilon C\varepsilon TACTG}
\end{array}\!;\,
\begin{array}{l}
\mathtt{AC\varepsilon \varepsilon CGTCA}\\
\mathtt{ACTAC\varepsilon TG\varepsilon }
\end{array}\!\!\!\!\!\!
\end{array}
\right\} \\
\hline
\end{array}
\]
In addition to these sets, we also want to include sets that enable us to compare the previous sequence alignments. For instance, the segments indexing the images $T(!_{[8]},[\mathtt{1\separ1\separ0\separ0}])$ and $T(!_{[8]},[\mathtt{0\separ1\separ1\separ0}])$ both have their coordinates $\texttt{x}_{\mathtt{b}}$ set to $1$, so they should both be able to go to the image $T(!_{[8]},[\mathtt{0\separ1\separ0\separ0}])$. An easy choice for such an image is to pick 
\[
T(!_{[8]},[\mathtt{0\separ1\separ0\separ0}]) := \mathbf{A}E_b^{\varepsilon}(!_{[8]},[\mathtt{0\separ1\separ0\separ0}])
\]
so that we have a diagram of functions as shown in (\ref{exa:construction_T_table_pullback}) by sending the DNA sequences of \texttt{Bob} -- which constitute the bottom rows of the sequence alignments of $T(!_{[8]},[\mathtt{1\separ1\separ0\separ0}])$ and the top rows of the sequence alignments of $T(!_{[8]},[\mathtt{0\separ1\separ1\separ0}])$ -- to the corresponding DNA sequences in $\mathbf{A}E_b^{\varepsilon}(!_{[8]},[\mathtt{0\separ0\separ1\separ0}])$.
\begin{equation}\label{exa:construction_T_table_pullback}
\xymatrix@C-35pt@R-10pt{
T(!_{[8]},[\mathtt{1\separ1\separ0\separ0}])\ar[rd]&&T(!_{[8]},[\mathtt{0\separ1\separ1\separ0}])\ar[ld]\\
&T(!_{[8]},[\mathtt{0\separ1\separ0\separ0}])&
}
\end{equation}
We can proceed similarly for the other images of $T$. However, we want to be cautious in doing so as every relation of the form (\ref{exa:construction_T_table_pullback}) will correspond to a test of compatibility between the sequence alignments contained in the images of $T$. Indeed, trying to link too many images of $T$ together may later lead to a set of empty `associations'. For instance, linking the set $T(!_{[8]},[\mathtt{0\separ0\separ1\separ1}])$ and $T(!_{[9]},[\mathtt{0\separ0\separ1\separ1}])$ through the set $\mathbf{A}E_b^{\varepsilon}(!_{[9]},[\mathtt{0\separ0\separ0\separ1}])$ (as shown below, in (\ref{eq:non_compatibility_different_length})) will label certain alignments of $T(!_{[8]},[\mathtt{0\separ0\separ1\separ1}])$ as `inconsistent' because they cannot be related to those of $T(!_{[9]},[\mathtt{0\separ0\separ1\separ1}])$. Ultimately, this is the type of conclusion that we would like to reach, but, in the present situation, the considered sets of alignments belong to the same pair of individuals and hence should not be able to contradict each other (see Remark \ref{rem:resolving_inconsistencies} for further discussion).
\begin{equation}\label{eq:non_compatibility_different_length}
\begin{array}{ccccccc}
T(!_{[8]},[\mathtt{0\separ0\separ1\separ1}])&\rightarrow&\mathbf{A}E_b^{\varepsilon}(!_{[8]},[\mathtt{0\separ0\separ0\separ1}])&\rightarrow&\mathbf{A}E_b^{\varepsilon}(!_{[9]},[\mathtt{0\separ0\separ0\separ1}])&\leftarrow &T(!_{[9]},[\mathtt{0\separ0\separ1\separ1}])\\
\begin{array}{l}
\mathtt{ACCGTC\varepsilon A}\\
\mathtt{ACT\varepsilon ACTG}
\end{array}
&\mapsto &
\begin{array}{l}
\mathtt{ACT\varepsilon ACTG}
\end{array}
&\mapsto &
\begin{array}{l}
\mathtt{A\varepsilon CT\varepsilon ACTG}
\end{array}
&\rotatebox[origin=c]{180}{$\mapsto$}& ???
\end{array}
\end{equation}
In other words, for this example, we want to take $T(!_{[n]},c)$ to be the set $\mathbf{A}E_b^{\varepsilon}(!_{[n]},c)$ for any integer $n$ and element $c$ in $\Omega^{\times 4}$ satisfying the following relation.
\[
(n,c) \in \{(8,[\mathtt{1\separ0\separ0\separ0}]), (8,[\mathtt{0\separ1\separ0\separ0}]), (8,[\mathtt{0\separ0\separ1\separ0}]), (8,[\mathtt{0\separ0\separ0\separ1}]),(7,[\mathtt{0\separ0\separ0\separ1}])\}
\]
We then link the images of $T$ at segments whose colors contain exactly two symbols \texttt{1} to the images of $T$ at segments whose colors contain a single \texttt{1} by forgetting either the top row or the bottom row of each pairwise sequence alignment contained in the domain, in the same fashion as in diagram (\ref{exa:construction_T_table_pullback}). The obtained maps are obviously indexed by the arrows of $B$ that makes sense with the forgetful operation they define. Doing so defines an inclusion functor $\iota:B \hookrightarrow \mathbf{Seg}(\Omega^{\times 4})$, a functor $T:B \to \mathbf{Set}$ and a natural transformation $T \Rightarrow \mathbf{A}E_b^{\varepsilon} \circ \iota$ that model the table given at the beginning of this example. 
\end{example}

\begin{remark}[Resolving inconsistencies]\label{rem:resolving_inconsistencies}
The present remark discusses the choices made in Example \ref{exa:sequence_alignments} for the construction of the functor $T:B \to \mathbf{Set}$. We shall keep the same notations as those introduced therein. First, recall that the category $B$ (the domain of $T$) was defined so that the images $T(!_{[8]},[\mathtt{0\separ0\separ1\separ1}])$ and $T(!_{[9]},[\mathtt{0\separ0\separ1\separ1}])$ could not connect through another image of $T$.
\begin{equation}\label{eq:resolving_inconsistencies}
\xymatrix@C-35pt@R-10pt{
T(!_{[8]},[\mathtt{0\separ0\separ1\separ1}])\ar[rd]&&T(!_{[9]},[\mathtt{0\separ0\separ1\separ1}])\ar[ld]\\
&\mathbf{A}E_{b}^{\varepsilon}(!_{[9]},[\mathtt{0\separ0\separ0\separ1}])&
}
\end{equation}
The reason for this was that not every element in $T(!_{[8]},[\mathtt{0\separ0\separ1\separ1}])$ could find a corresponding element in $T(!_{[9]},[\mathtt{0\separ0\separ1\separ1}])$ through the set $\mathbf{A}E_{b}^{\varepsilon}(!_{[9]},[\mathtt{0\separ0\separ0\separ1}])$ so that diagram (\ref{eq:resolving_inconsistencies}) would eventually lead to label those elements as inconsistent. In our case, this type of scenario should be avoided because we are mainly interested in unravelling inconsistencies between different tables produced by the dynamic programming algorithm of section \ref{ssec:Main_example} (while diagram (\ref{eq:resolving_inconsistencies}) compares alignments coming from the same table).
However, it could certainly be interesting to be able to know whether the sequence alignments of \texttt{Craig} and \texttt{Doug} of length 8 are consistent with those of length 9. A way to do this without creating a conflict with our main goal would be to add a new color to $\Omega$, say by taking the pre-ordered set $\Omega = \{0 \leq 1 \leq 2\}$ and keeping $b = (1,1,1,1)$. Then, we could use this new color to study the compatibility of the alignments coming from the same table. 
For instance, we could give $T$ an image at the segment $(!_{[8]},[\mathtt{0\separ0\separ2\separ2}])$ that would contain the same alignments as those contained in $T(!_{[9]},[\mathtt{0\separ0\separ1\separ1}])$, but the resulting image $T(!_{[8]},[\mathtt{0\separ0\separ2\separ2}])$ would be linked to the image $T(!_{[8]},[\mathtt{0\separ0\separ1\separ1}])$ through the image 
$$T(!_{[9]},[\mathtt{0\separ0\separ0\separ1}]) := \mathbf{A}E_{b}^{\varepsilon}(!_{[9]},[\mathtt{0\separ0\separ0\separ1}])$$
as shown in the diagram given below (the left-hand side arrow makes sense with the functor structure because the inequality $[\mathtt{0\separ0\separ0\separ1}] \preceq [\mathtt{0\separ0\separ2\separ2}]$ holds in $\Omega^{\times 4}$).
\[
\xymatrix@C-35pt@R-10pt{
T(!_{[8]},[\mathtt{0\separ0\separ2\separ2}])\ar[rd]&&T(!_{[9]},[\mathtt{0\separ0\separ1\separ1}])\ar[ld]\\
&T(!_{[9]},[\mathtt{0\separ0\separ0\separ1}])&
}
\]
Meanwhile, the image $T(!_{[8]},[\mathtt{0\separ0\separ1\separ1}])$ would be reserved to studying the compatibility with the other pairs of individuals and would be isolated from $T(!_{[9]},[\mathtt{0\separ0\separ1\separ1}])$ because the category $\mathbf{Seg}(\Omega)$ does not allow morphisms of the type $(!_{[8]},[\mathtt{0\separ0\separ1\separ1}]) \to (!_{[8]},[\mathtt{0\separ0\separ2\separ2}])$ to exist.
\end{remark}

\subsection{From right Kan extensions to multiple sequence alignments}\label{ssec:From_RKE_to_MSA}
In this section, we show that the right Kan extension (see Definition \ref{def:right_kan_extensions}) of a sequence alignment functor (Definition \ref{def:sequence_alignment}) contains what one would like to understand as the outputs of the table gluing algorithm described at the end of section \ref{ssec:Main_example}.

We start the section with an example showing how sequence alignment functors can be used, along with limits, to reason about the relatedness of a group of individuals. In particular, we illustrate it in the context of our main example (section \ref{ssec:Main_example}).

\begin{example}[Reasoning with sequence alignment functors]\label{exa:reasoning_with_sequence_alignment}
The present example continues the discussion of Example \ref{exa:sequence_alignments} and shows how the sequence alignment functor $(\iota,T,\sigma)$ defined therein can be used to reason about our main example given in section \ref{ssec:Main_example}. More specifically, we will show how the sequence alignment functor $(\iota,T,\sigma)$ can be used to deduce phylogenetic relationships between the individuals of section \ref{ssec:Main_example} by looking at whether certain arrows induced by the structure of $(\iota,T,\sigma)$ are proper surjections, bijections and simply functions.

We start this example by looking at a surjection. First, an easy calculation shows that the pullback of diagram (\ref{exa:construction_T_table_pullback}) can be mapped surjectively onto the image $T(!_{[8]},[\mathtt{1\separ0\separ1\separ0}])$ by forgetting the DNA sequences associated with \texttt{Bob}.
\begin{equation}\label{eq:surjection_AnneCraigBob}
T(!_{[8]},[\mathtt{1\separ1\separ0\separ0}]) \times_{\mathtt{x_{b}}} T(!_{[8]},[\mathtt{0\separ1\separ1\separ0}])  \to T(!_{[8]},[\mathtt{1\separ0\separ1\separ0}])
\end{equation}
More specifically, the associated surjection, of the form shown in (\ref{eq:surjection_AnneCraigBob}), maps the pairs of pairwise sequence alignments shown in (\ref{eq:domain_surjection_ABC}) (the domain of (\ref{eq:surjection_AnneCraigBob})) to the pairwise sequence alignments shown in (\ref{eq:domain_surjection_AC}) (the codomain of (\ref{eq:surjection_AnneCraigBob})) by forgetting the bottom and top sequences of the first and second components of the pairs contained in the domain. 
\begin{equation}\label{eq:domain_surjection_ABC}
\left\{
\begin{array}{l}
\begin{array}{l}
\left(
\begin{array}{l}
\mathtt{ACCGACTG}\\
\mathtt{A\varepsilon CATCTG}
\end{array}
\right)\\
\left(
\begin{array}{l}
\mathtt{A\varepsilon CATCTG}\\
\mathtt{ACCGTC\varepsilon A}
\end{array}
\right)
\end{array}

\begin{array}{l}
\left(
\begin{array}{l}
\mathtt{ACCGACTG}\\
\mathtt{A CA\varepsilon TCTG}
\end{array}
\right)\\
\left(
\begin{array}{l}
\mathtt{A CA\varepsilon TCTG}\\
\mathtt{ACCGTC\varepsilon A}
\end{array}
\right)
\end{array}

\begin{array}{l}
\left(
\begin{array}{l}
\mathtt{ACCGACTG}\\
\mathtt{A\varepsilon CATCTG}
\end{array}
\right)\\
\left(
\begin{array}{l}
\mathtt{A\varepsilon CATCTG}\\
\mathtt{ACCGTCA\varepsilon}
\end{array}
\right)
\end{array}

\begin{array}{l}
\left(
\begin{array}{l}
\mathtt{ACCGACTG}\\
\mathtt{A CA\varepsilon TCTG}
\end{array}
\right)\\
\left(
\begin{array}{l}
\mathtt{A CA\varepsilon TCTG}\\
\mathtt{ACCGTC A\varepsilon}
\end{array}
\right)
\end{array}
\end{array}
\right\}
\end{equation}

\begin{equation}\label{eq:domain_surjection_AC}
\left\{
\begin{array}{l}
\left(
\begin{array}{l}
\mathtt{ACCGACTG}\\
\mathtt{ACCGTC\varepsilon A}
\end{array}
\right)
\left(
\begin{array}{l}
\mathtt{ACCGACTG}\\
\mathtt{ACCGTCA\varepsilon}
\end{array}
\right)
\end{array}
\right\}
\end{equation}
As seen in (\ref{eq:domain_surjection_ABC}), the elements of the domain can be interpreted as multiple sequence alignments of \texttt{Anne}, \texttt{Bob} and \texttt{Craig}, where the sequence of $\texttt{Bob}$ is repeated twice. The fact that these multiple sequence alignments can be sent to the pairwise sequences of \texttt{Anne} and \texttt{Craig} tells us that \texttt{Bob} does not inform us of new ways of relating \texttt{Anne} and \texttt{Craig} together. Furthermore, the fact that function (\ref{eq:domain_surjection_ABC}) is also a surjection tells us that the sequence of \texttt{Bob} is also unable to provide any potential correction to the sequence alignments of \texttt{Anne} and \texttt{Craig} (by showing us some sequence alignment of \texttt{Anne} and \texttt{Caig} that would not be supported by the sequence alignments of (\ref{eq:domain_surjection_ABC})). In the present case, it is as if \texttt{Bob} is unnecessary for understanding possible nuances in the evolution of \texttt{Anne} and \texttt{Craig}. A very probable reason could be that \texttt{Anne} and \texttt{Craig} are much closer to each other genetically than they are to \texttt{Bob} -- this would be represented by the following evolutionary tree.
\[
\mathrm{
\xy
(0,-11)*{}="a";
"a"-<+4ex,-4ex>*{}="a1"**\dir{-},
"a"-<-4ex,-4ex>*{}="a2"**\dir{-},
"a1";
"a1"-<+4ex,-4ex>*{\texttt{Anne}}="a11"**\dir{-},
"a1"-<-4ex,-4ex>*{\texttt{Craig}}="a12"**\dir{-},
"a2";
"a2"-<-4ex,-4ex>*{\texttt{Bob}}="a111"**\dir{-},
\endxy
}
\]
From the point of view of the problem exposed in section \ref{ssec:Main_example}, surjection (\ref{eq:surjection_AnneCraigBob}) tells us that gluing the comparison table of \texttt{Anne} and \texttt{Bob} with the comparison table of \texttt{Bob} and \texttt{Craig} along the edge of \texttt{Bob} does not contradict the table of \texttt{Anne} and \texttt{Craig}. In other words, there is no obstruction for passing from the two dimensional comparisons of \texttt{Anne} with \texttt{Bob} and \texttt{Bob} with \texttt{Craig} to the three dimensional table comparing \texttt{Anne}, \texttt{Bob} and \texttt{Craig} together. 

Let us now look at an example for which the function is not surjective. First, recall that the sequence alignment $(\iota,T,\sigma)$ was defined such that cospan (\ref{eq:cospan_DougCraig_along_Anne}) exists (see Example \ref{exa:sequence_alignments}).
\begin{equation}\label{eq:cospan_DougCraig_along_Anne}
\xymatrix@C-35pt@R-10pt{
T(!_{[8]},[\mathtt{1\separ0\separ1\separ0}])\ar[rd]&&T(!_{[8]},[\mathtt{1\separ0\separ0\separ1}])\ar[ld]\\
&T(!_{[8]},[\mathtt{1\separ0\separ0\separ0}])&
}
\end{equation}
An easy calculation then shows that the pullback of cospan (\ref{eq:cospan_DougCraig_along_Anne}) contains multiple sequence alignments of  \texttt{Anne}, \texttt{Craig} and \texttt{Doug} that can be mapped to the sequence alignments of \texttt{Craig} and \texttt{Doug} contained in $T(!_{[8]},[\mathtt{0\separ0\separ1\separ1}])$ by forgetting the DNA sequences associated with \texttt{Anne}.
\begin{equation}\label{eq:surjection_AnneCraigDoug}
T(!_{[8]},[\mathtt{1\separ0\separ1\separ0}]) \times_{\mathtt{x_{a}}} T(!_{[8]},[\mathtt{1\separ0\separ0\separ1}])  \to T(!_{[8]},[\mathtt{0\separ0\separ1\separ1}])
\end{equation}
However, comparing the cardinalities of the domain and codomain of the resulting function, shown in (\ref{eq:surjection_AnneCraigDoug}), informs us that this mapping cannot be surjective. This suggests that, while \texttt{Anne} does not give new ways of relating \texttt{Craig} and \texttt{Doug} together, the relatedness of \texttt{Craig} and \texttt{Doug} may still be nuanced by the consideration of \texttt{Anne}. In other words, the evolution of \texttt{Craig} and \texttt{Doug} cannot completely be explained without the sequence of \texttt{Anne}. Such a relationship could be represented by one of the following evolutionary trees, in which the removal of \texttt{Anne} may prevent us from understanding how different \texttt{Craig} and \texttt{Doug} are.
\[
\mathrm{
\xy
(0,-11)*{}="a";
"a"-<+4ex,-4ex>*{}="a1"**\dir{-},
"a"-<-4ex,-4ex>*{}="a2"**\dir{-},
"a1";
"a1"-<+4ex,-4ex>*{\texttt{Craig}}**\dir{-},
"a1"-<-4ex,-4ex>*{\texttt{Anne}}**\dir{-},
"a2";
"a2"-<-4ex,-4ex>*{\texttt{Doug}}**\dir{-},
\endxy
}
\quad\quad\quad\quad\quad\quad
\mathrm{
\xy
(0,-11)*{}="a";
"a"-<+4ex,-4ex>*{}="a1"**\dir{-},
"a"-<-4ex,-4ex>*{}="a2"**\dir{-},
"a1";
"a1"-<+4ex,-4ex>*{\texttt{Craig}}**\dir{-},
"a2";
"a2"-<-4ex,-4ex>*{\texttt{Doug}}**\dir{-},
"a2"-<+4ex,-4ex>*{\texttt{Anne}}**\dir{-},
\endxy
}
\]
From the point of view of the problem exposed in section \ref{ssec:Main_example}, function (\ref{eq:surjection_AnneCraigDoug}) tells us that gluing the comparison table of \texttt{Anne} and \texttt{Craig} with the comparison table of \texttt{Anne} and \texttt{Doug} along the edge of \texttt{Anne} does not confirm all the pairwise sequence alignments computed for \texttt{Craig} and \texttt{Doug}. This suggests that the relatedness of \texttt{Craig} and \texttt{Doug} may be too old to be completely described by the table computed for \texttt{Craig} and \texttt{Doug} and would be better analyzed through the gluing of other tables.
\end{example}

Our goal is now to formalize the discussion of Example \ref{exa:reasoning_with_sequence_alignment} through the notion of right Kan extension (given in Definition \ref{def:right_kan_extensions}). This will allow us to motivate the introduction of chromologies in Example \ref{exa:Global-local_seq_alignments}. We start by defining the category on which the right Kan extension is \emph{computed}.

\begin{definition}[Extending category]\label{def:extending_category}
Let $(\Omega,\preceq)$ be a pre-ordered set. For every object $\tau$ in $\mathbf{Seg}(\Omega)$ and functor $\iota:B \to \mathbf{Seg}(\Omega)$, we will denote by $({\tau \downarrow \iota})$ the category whose objects are pairs $(\upsilon,f)$ where $\upsilon$ is an object in $B$ and $f$ is a morphism $\tau \to \iota(\upsilon)$ in $\mathbf{Seg}(\Omega)$ and whose arrows $(\upsilon,f) \to (\upsilon',f')$ are given by morphisms $g:\upsilon \to \upsilon'$ in $B$ that make the following square commute in $\mathbf{Seg}(\Omega)$.
\[
\xymatrix{
\tau\ar[d]_{f}\ar@{=}[r]&\tau\ar[d]^{f'}\\
\iota(\upsilon)\ar[r]_{\iota(g)}&\iota(\upsilon')
}
\]
\end{definition}

\begin{remark}[Extending category as a cone]\label{rem:Extending_category_as_a_cone}
Let $(\Omega,\preceq)$ be a pre-ordered set. For every object $\tau$ in $\mathbf{Seg}(\Omega)$ and functor $\iota:B \to \mathbf{Seg}(\Omega)$, the category $({\tau \downarrow \iota})$ defined in Definition \ref{def:extending_category} can be pictured as a cone of the form $\Delta_{({\tau \downarrow \iota})}(\tau)\Rightarrow \iota_{\tau} \circ \iota$ (section \ref{ssec:cones}). The arrows of the transformation associated with the cone are given by the objects of the category $({\tau \downarrow \iota})$ (see non-dashed arrows shown below) while its diagram is formed by the arrows of $({\tau \downarrow \iota})$ (see dashed arrows shown below).
\[
\xymatrix@C-20pt@R-5pt{
&&&\tau\ar@/_.8pc/[llld]\ar@/_1pc/[lldd]\ar[ld]\ar[rd]\ar@/^1pc/[rrdd]\ar@/^.8pc/[rrrd]&&&\\
\iota(\upsilon_1)&&\iota(\upsilon_2)&\dots&\iota(\upsilon_{n-1})&&\iota(\upsilon_n)\\
&\iota(\upsilon_{1,2})\ar@{-->}[ul]^{\iota(g_{1})}\ar@{-->}[ur]_{\iota(g_{2})}&&\dots&&\iota(\upsilon_{n-1,n})\ar@{-->}[ul]^{\iota(g_{n-1})}\ar@{-->}[ur]_{\iota(g_{n})}&\\
}
\]
\end{remark}

\begin{convention}[Extending diagram]\label{conv:extending_diagram}
Let $(\Omega,\preceq)$ be a pre-ordered set. For every object $\tau$ in $\mathbf{Seg}(\Omega)$ and functor $\iota:B \to \mathbf{Seg}(\Omega)$, we will denote by $\iota_{\tau}$ the obvious functor $({\tau \downarrow \iota}) \to B$ that maps an object $(\upsilon,f)$ in $({\tau \downarrow \iota})$ to the object $\upsilon$ in $B$ and maps an arrow $g:(\upsilon,f) \to (\upsilon',f')$ in $({\tau \downarrow \iota})$ to the arrow $g:\upsilon \to \upsilon'$ in $B$.
\end{convention}

The following definition introduces the concept of right Kan extension by using its well-known expression in terms of limits (see \cite[Chap. X, Th. 1]{MacLane}). 

\begin{definition}[Right Kan extensions]\label{def:right_kan_extensions}
Let $(\Omega,\preceq)$ be a pre-ordered set, and $\iota:B \to \mathbf{Seg}(\Omega)$ and $T:B \to \mathbf{Set}$ be two functors. We define the \emph{right Kan extension of $T$ along $\iota$} as the canonical functor $\mathsf{Ran}_{\iota}T:\mathbf{Seg}(\Omega) \to \mathbf{Set}$ defined by the following limit construction at every object $\tau$ in $\mathbf{Seg}(\Omega)$.
\begin{equation}\label{eq:formula_right_kan_extension}
\mathsf{Ran}_{\iota}T(\tau) = \mathsf{lim}_{({\tau \downarrow \iota})} T \circ \iota_{\tau}
\end{equation}
The images of this functor on the arrows of $\mathbf{Seg}(\Omega)$ will be described, later, in Remark \ref{rem:functoriality_RKE}.
\end{definition}

\begin{example}[Right Kan extensions]\label{exa:right_kan_extension_images}
Let $(\iota,T,\sigma)$ denote the sequence alignment functor over $\mathbf{A}E_b^{\varepsilon}$ defined Example \ref{exa:reasoning_with_sequence_alignment}. The goal of the present example is to show what the images of the right Kan extension of $T:B \to \mathbf{Set}$ along the inclusion $\iota:B \to \mathbf{Seg}(\Omega^{\times 4})$ look like. We will compute two images, namely $\mathsf{Ran}_{\iota}T(!_{[8]},[\mathtt{1\separ1\separ0\separ0}])$ and $\mathsf{Ran}_{\iota}T(!_{[8]},[\mathtt{1\separ1\separ1\separ0}])$. 

Let us start with the image of the segment $(!_{[8]},[\mathtt{1\separ1\separ0\separ0}])$. According to Definition \ref{def:right_kan_extensions}, we need to compute the set of objects and the set of arrows defining the category $((!_{[8]},[\mathtt{1\separ1\separ0\separ0}]) \downarrow \iota)$. We follow Definition \ref{def:extending_category}, for which we use the category $B$ described in Example \ref{exa:sequence_alignments}, and deduce that the objects of $((!_{[8]},[\mathtt{1\separ1\separ0\separ0}]) \downarrow \iota)$ are of the following form.
\begin{equation}\label{eq:RKE_computation_first_image:the_objects}
\begin{array}{lll}
(!_{[8]},[\mathtt{1\separ1\separ0\separ0}]) \to (!_{[8]},[\mathtt{0\separ1\separ0\separ0}]) &\quad\quad\quad &
(!_{[8]},[\mathtt{1\separ1\separ0\separ0}]) \to (!_{[8]},[\mathtt{1\separ0\separ0\separ0}])\\
(!_{[8]},[\mathtt{1\separ1\separ0\separ0}]) \to (!_{[8]},[\mathtt{1\separ1\separ0\separ0}])&\quad\quad\quad&
(!_{[8]},[\mathtt{1\separ1\separ0\separ0}]) \to (!_{[9]},[\mathtt{1\separ1\separ0\separ0}])
\end{array}
\end{equation}
Then, Definition \ref{def:extending_category} tells us that the arrows of the category $((!_{[8]},[\mathtt{1\separ1\separ0\separ0}]) \downarrow \iota)$ are given by the arrows of $B$ that can form commutative triangles with the arrows of (\ref{eq:RKE_computation_first_image:the_objects}). 

Let us describe the arrows of $((!_{[8]},[\mathtt{1\separ1\separ0\separ0}]) \downarrow \iota)$ more explicitly. First, by definition of $B$, any object of $((!_{[8]},[\mathtt{1\separ1\separ0\separ0}]) \downarrow \iota)$ of the type shown in the bottom-right corner of (\ref{eq:RKE_computation_first_image:the_objects}) is isolated from any object of  a different type in (\ref{eq:RKE_computation_first_image:the_objects}). Since there are exactly $9$ representatives of the type 
\[
(!_{[8]},[\mathtt{1\separ1\separ0\separ0}]) \to (!_{[9]},[\mathtt{1\separ1\separ0\separ0}])
\]
in $\mathbf{Seg}(\Omega^{\times 4}\,|\,9)$, the category $((!_{[8]},[\mathtt{1\separ1\separ0\separ0}]) \downarrow \iota)$ has exactly 9 isolated objects of this type. However, because the image of the segment $(!_{[9]},[\mathtt{1\separ1\separ0\separ0}])$ via $T$ is terminal (see Example \ref{exa:sequence_alignments}), these isolated objects will not matter in the computation of the limit $\mathsf{Ran}_{\iota}T(!_{[8]},[\mathtt{1\separ1\separ0\separ0}])$ (see formula (\ref{eq:formula_right_kan_extension})). 
What matters is the diagram encoded by the other objects, which we now describe. First, because $\mathbf{Seg}(\Omega^{\times 4}\,|\,8)$ is a pre-order category (Proposition \ref{prop:Seg_Omega_n_porder}), the arrows of $((!_{[8]},[\mathtt{1\separ1\separ0\separ0}]) \downarrow \iota)$ that relate the remaining objects are unique. Also, observe that the object $(!_{[8]},[\mathtt{1\separ1\separ0\separ0}])$ can be related to the objects $(!_{[8]},[\mathtt{0\separ1\separ0\separ0}])$ and $(!_{[8]},[\mathtt{1\separ0\separ0\separ0}])$ through a cospan in $B$. Using formula (\ref{eq:formula_right_kan_extension}), we deduce that isomorphism (\ref{eq:image_RKE_8_1100}) holds.
\begin{equation}\label{eq:image_RKE_8_1100}
\mathsf{Ran}_{\iota}T(!_{[8]},[\mathtt{1\separ1\separ0\separ0}]) \cong T(!_{[8]},[\mathtt{1\separ1\separ0\separ0}])
\end{equation}
In other words, the image of the right Kan extension $\mathsf{Ran}_{\iota}T$ at the segment $(!_{[8]},[\mathtt{1\separ1\separ0\separ0}])$ contains all the pairwise sequence alignments of \texttt{Anne} and \texttt{Bob} coming from the dynamic programming algorithm.

Let us now compute the image of the segment $(!_{[8]},[\mathtt{1\separ1\separ1\separ0}])$. To do so, we need to describe the set of objects and the set of arrows of the category $({(!_{[8]},[\mathtt{1\separ1\separ1\separ0}]) \downarrow \iota})$. In the present case, the set of objects is made of arrows of the following type.
\[
(!_{[8]},[\mathtt{1\separ1\separ1\separ0}]) \to (!_{[8]},[\mathtt{1\separ0\separ0\separ0}]) \quad\quad (!_{[8]},[\mathtt{1\separ1\separ1\separ0}]) \to (!_{[8]},[\mathtt{0\separ1\separ0\separ0}]) \quad\quad (!_{[8]},[\mathtt{1\separ1\separ1\separ0}]) \to (!_{[8]},[\mathtt{0\separ0\separ1\separ0}])
\]
\[
(!_{[8]},[\mathtt{1\separ1\separ1\separ0}]) \to (!_{[8]},[\mathtt{1\separ0\separ1\separ0}]) \quad\quad\quad\quad 
(!_{[8]},[\mathtt{1\separ1\separ1\separ0}]) \to (!_{[8]},[\mathtt{0\separ1\separ1\separ0}])
\]
\[
(!_{[8]},[\mathtt{1\separ1\separ1\separ0}]) \to (!_{[8]},[\mathtt{1\separ1\separ0\separ0}]) \quad\quad\quad\quad  (!_{[8]},[\mathtt{1\separ1\separ1\separ0}]) \to (!_{[9]},[\mathtt{1\separ1\separ0\separ0}]) 
\]
After eliminating the objects whose images via $T$ are terminal and only considering the morphisms that relate the remaining objects in $({(!_{[8]},[\mathtt{1\separ1\separ1\separ0}]) \downarrow \iota})$, we can deduce from formula (\ref{eq:formula_right_kan_extension}) that the image of $\mathsf{Ran}_{\iota}T$ at the segment $(!_{[8]},[\mathtt{1\separ1\separ1\separ0}])$ is isomorphic to the limit of the following diagram.
\begin{equation}\label{eq:diagram_exa_right_kan_extension}
\xymatrix@C-15pt@R+5pt{
T(!_{[8]},[\mathtt{1\separ0\separ1\separ0}])\ar[rd]\ar[rrd]|>>>>>>>>>\hole&&T(!_{[8]},[\mathtt{0\separ1\separ1\separ0}])\ar[ld]\ar[rd]&&\ar[ld]\ar[lld]|>>>>>>>>>\hole T(!_{[8]},[\mathtt{1\separ1\separ0\separ0}])\\
&T(!_{[8]},[\mathtt{0\separ0\separ1\separ0}])&T(!_{[8]},[\mathtt{1\separ0\separ0\separ0}])&T(!_{[8]},[\mathtt{0\separ1\separ0\separ0}])&
}
\end{equation}
As shown in Example \ref{exa:reasoning_with_sequence_alignment}, the elements of $\mathsf{Ran}_{\iota}T(!_{[8]},[\mathtt{1\separ1\separ1\separ0}])$ can be seen as triples of pairwise sequence alignments contained in $T(!_{[8]},[\mathtt{1\separ0\separ1\separ0}])$, $T(!_{[8]},[\mathtt{0\separ1\separ1\separ0}])$ and $T(!_{[8]},[\mathtt{1\separ1\separ0\separ0}])$ that encode multiple sequence alignments between the sequences of \texttt{Anne}, \texttt{Bob} and \texttt{Craig}.
\end{example}

\begin{remark}[Right Kan extension and gluing of tables]\label{rem:RKE_gluing_tables}
It should be clear that the images of the right Kan extension computed in Example \ref{exa:right_kan_extension_images} capture the higher dimensional gluings mentioned in section \ref{ssec:Main_example}. The idea is that one formally specifies how the pairwise comparison tables are glued together via the category defined in Definition \ref{def:extending_category} and one uses the limit formula given in Definition \ref{def:right_kan_extensions} to collect the compatible sequence alignments for these gluings. As a result, the right Kan extension can be seen as a model for a multiple sequence alignment algorithm. From now on, our goal will be to describe how this model can be used to reason about the mechanisms linking the individuals considered in the sequence alignment functor.
\end{remark}

\begin{remark}[Functoriality]\label{rem:functoriality_RKE}
Let $(\Omega,\preceq)$ be a pre-ordered set, and $\iota:B \to \mathbf{Seg}(\Omega)$ and $T:B \to \mathbf{Set}$ be two functors. Let us explain why the map $\tau \mapsto \mathsf{Ran}_{\iota} T(\tau)$ constructed in Definition \ref{def:right_kan_extensions} induces a functor from $\mathbf{Seg}(\Omega)$ to $\mathbf{Set}$. First, notice that the functor $\iota_{\tau}:({\tau \downarrow \iota}) \to B$ defined in Convention \ref{conv:extending_diagram} is natural in $\tau$ on the opposite category $\mathbf{Seg}(\Omega)^{\mathrm{op}}$, which means that every morphism $h:\tau \to \tau'$ in $\mathbf{Seg}(\Omega)$ induces a functor 
\[
h^{*}:({\tau'\downarrow\iota}) \to ({\tau\downarrow\iota})
\]
for which the identity $\iota_{\tau'}  = \iota_{\tau} \circ h^{*}$ holds. This last equation means that the functor $h^{*}$ sends an object $(\upsilon,f)$ in $({\tau'\downarrow\iota})$ to the object $(\upsilon,f \circ h)$ in $({\tau\downarrow\iota})$. The image of the morphism $h:\tau \to \tau'$ via $\mathsf{Ran}_{\iota}T$ is then the canonical morphism induced by pre-composing the diagram of the limit shown below, on the left, with $h^{*}$.
\[
\xymatrix@C+20pt{
\mathsf{lim}_{({\tau \downarrow \iota})} T \circ \iota_{\tau} \ar[r]^-{\mathsf{Ran}_{\iota}T(h)}& \mathsf{lim}_{({\tau \downarrow \iota})} T \circ \iota_{\tau} \circ h^{*}
}
\]
The functor structure of the mapping $\tau \mapsto \mathsf{Ran}_{\iota} T(\tau)$ then follows from the universality of limits.
\end{remark}

\begin{example}[Reasoning with right Kan extensions]\label{exa:Reasoning_with_right_Kan_extensions}
The present example is a continuation of the discussion of Example \ref{exa:reasoning_with_sequence_alignment} from the point of view of Definition \ref{def:right_kan_extensions} and Example \ref{exa:right_kan_extension_images}. We shall keep the same notations as those used in Example \ref{exa:reasoning_with_sequence_alignment} and let $(\iota,T,\sigma)$ denote the sequence alignment functor over $\mathbf{A}E_b^{\varepsilon}$ used thereof. Our goal is to reformulate the statement given in Example \ref{exa:reasoning_with_sequence_alignment} regarding the surjection
\[
T(!_{[8]},[\mathtt{1\separ1\separ0\separ0}]) \times_{\mathtt{x_{b}}} T(!_{[8]},[\mathtt{0\separ1\separ1\separ0}])  \to T(!_{[8]},[\mathtt{1\separ0\separ1\separ0}])
\]
in terms of the right Kan extension of $T$ along $\iota$. First, it follows from the description given therein that this function, represented by the dashed arrow below, induces a cone over diagram (\ref{eq:diagram_exa_right_kan_extension}) --  the cone is shown below, using the dashed and dotted arrows.
\[
\xymatrix@C-35pt@R+5pt{
&&\ar@{..>}[rrd]\ar@{..>}[d]\ar@{-->}[lld]T(!_{[8]},[\mathtt{1\separ1\separ0\separ0}]) \times_{\mathtt{x_{b}}} T(!_{[8]},[\mathtt{0\separ1\separ1\separ0}])&&\\
T(!_{[8]},[\mathtt{1\separ0\separ1\separ0}])\ar[rd]\ar[rrd]|>>>>>>>>>>>>\hole&&T(!_{[8]},[\mathtt{0\separ1\separ1\separ0}])\ar[ld]\ar[rd]&&\ar[ld]\ar[lld]|>>>>>>>>>>>>\hole T(!_{[8]},[\mathtt{1\separ1\separ0\separ0}])\\
&T(!_{[8]},[\mathtt{0\separ0\separ1\separ0}])&T(!_{[8]},[\mathtt{1\separ0\separ0\separ0}])&T(!_{[8]},[\mathtt{0\separ1\separ0\separ0}])&
}
\]
The universal property of limits then gives us a factorization of the dashed surjection through one of the canonical projections associated with the limit of diagram (\ref{eq:diagram_exa_right_kan_extension}). More specifically, this projection is of the form shown in (\ref{eq:canonical_projection_reformulation_RKE}).
\begin{equation}\label{eq:canonical_projection_reformulation_RKE}
\mathsf{Ran}_{\iota}T(!_{[8]},[\mathtt{1\separ1\separ1\separ0}]) \to T(!_{[8]},[\mathtt{1\separ0\separ1\separ0}])
\end{equation}
The usual properties of surjections (or, in fact, those of epimorphisms) imply that projection (\ref{eq:canonical_projection_reformulation_RKE}) is also a surjection. Using a similar reasoning to the one used to deduce isomorphism (\ref{eq:image_RKE_8_1100}), we can show that the codomain of projection (\ref{eq:canonical_projection_reformulation_RKE}) is isomorphic to $\mathsf{Ran}_{\iota}T(!_{[8]},[\mathtt{1\separ0\separ1\separ0}])$. This means that (\ref{eq:canonical_projection_reformulation_RKE}) is of the following form.
\begin{equation}\label{eq:canonical_projection_reformulation_RKE_images}
\mathsf{Ran}_{\iota}T(!_{[8]},[\mathtt{1\separ1\separ1\separ0}]) \longrightarrow \mathsf{Ran}_{\iota}T(!_{[8]},[\mathtt{1\separ0\separ1\separ0}])
\end{equation}
In fact, we can even show that the previous surjection is the image of the obvious morphism of segments
\[
(!_{[8]},[\mathtt{1\separ1\separ1\separ0}]) \to (!_{[8]},[\mathtt{1\separ0\separ1\separ0}])
\]
via the functor described in Remark \ref{rem:functoriality_RKE}.
In the spirit of Remark \ref{rem:RKE_gluing_tables}, in which the images of the right Kan extension were interpreted as models for higher dimensional gluings of comparison tables between DNA sequences, surjection (\ref{eq:canonical_projection_reformulation_RKE_images}) tells us that the higher dimensional gluing of the comparison tables of \texttt{Anne}, \texttt{Bob} and \texttt{Craig} is completely captured by the comparison table of \texttt{Anne} and \texttt{Craig}, with some uncertainty as to what exactly links \texttt{Anne} and \texttt{Craig} through \texttt{Bob}.
\end{example}

\begin{example}[Reasoning about data consistency]\label{exa:right_kan_extension_data_consistency}
As already suggested through the discussions of Example \ref{exa:sequence_alignments}, Remark \ref{rem:resolving_inconsistencies} and Example \ref{exa:reasoning_with_sequence_alignment}, limits can give ways to assess the \emph{consistency} of the data. For instance, the right Kan extension of a sequence alignment functor is given by a limit (Definition \ref{def:right_kan_extensions}) that computes \emph{compatible} sets of sequence alignments (these compatible sets of alignments were interpreted as multiple sequence alignments in Remark \ref{rem:RKE_gluing_tables}). In this example, we compute limits of images of the right Kan extension to study how \emph{consistent} these compatible sets of sequence alignments are. As before, we assess this consistency by studying the properties of certain canonical arrows induced by these limits. 

We shall continue the discussion started in Example \ref{exa:Reasoning_with_right_Kan_extensions} and look at the limit of diagram (\ref{eq:diagram_exa_right_kan_extension}).
First, observe that we can copy the reasoning that was used in Example \ref{exa:right_kan_extension_images} to deduce isomorphism (\ref{eq:image_RKE_8_1100}) to show that diagram (\ref{eq:diagram_exa_right_kan_extension}) is in fact a diagram of images of $\mathsf{Ran}_{\iota}T$ as follows.
\[
\xymatrix@C-45pt@R+5pt{
\mathsf{Ran}_{\iota}T(!_{[8]},[\mathtt{1\separ0\separ1\separ0}])\ar[rd]\ar[rrd]|>>>>>>>>>\hole&&\mathsf{Ran}_{\iota}T(!_{[8]},[\mathtt{0\separ1\separ1\separ0}])\ar[ld]\ar[rd]&&\ar[ld]\ar[lld]|>>>>>>>>>\hole \mathsf{Ran}_{\iota}T(!_{[8]},[\mathtt{1\separ1\separ0\separ0}])\\
&\mathsf{Ran}_{\iota}T(!_{[8]},[\mathtt{0\separ0\separ1\separ0}])\quad\quad&\mathsf{Ran}_{\iota}T(!_{[8]},[\mathtt{1\separ0\separ0\separ0}])&\quad\quad\mathsf{Ran}_{\iota}T(!_{[8]},[\mathtt{0\separ1\separ0\separ0}])&
}
\]
Remark \ref{rem:functoriality_RKE} even tells us that the previous diagram is the image of the underlying diagram of segments via the right Kan extension $\mathsf{Ran}_{\iota}T$. More specifically, if we let $F:A \to \mathbf{Seg}(\Omega)$ denote the diagram of segments indexing the previous diagram, then the computation of the image $\mathsf{Ran}_{\iota}T(!_{[8]},[\mathtt{1\separ1\separ1\separ0}])$ described in Example \ref{exa:right_kan_extension_images} shows that the following canonical arrow is an isomorphism in $\mathbf{Set}$.
\[
\mathsf{Ran}_{\iota}T(!_{[8]},[\mathtt{1\separ1\separ1\separ0}]) \to \mathsf{lim}_{A} \mathsf{Ran}_{\iota}T \circ F
\]
From the point of view of Remark \ref{rem:RKE_gluing_tables}, this means that there is no uncertainty as to how \texttt{Anne}, \texttt{Bob} and \texttt{Craig} relate to each other from the point of view of their pairwise sequence alignments -- the previous isomorphism thus informs us that the table gluing procedure is perfectly consistent. In Example \ref{exa:Global-local_seq_alignments}, we will look at an instance of a canonical arrow that fails to be an isomorphism and is only a surjection. We will see that this type of arrow informs us that the integrated data is associated with some uncertainties.
\end{example}

The idea behind the right Kan extension of a sequence alignment functor is to collect \emph{all} the local and global information that is accessible from the point of view of a particular segment (for more intuition, see Remark \ref{rem:Extending_category_as_a_cone}). Since the domain category of a sequence alignment functor can be designed to control the integration of this information (see Remark \ref{rem:resolving_inconsistencies}), the right Kan extension of a sequence alignment functor gives us a controlled procedure to construct multiple sequence alignments (Remark \ref{rem:RKE_gluing_tables}) through the parsing of local and global pieces of information within the data set.

In bioinformatics, similar heuristics have been developed for the construction of sequence alignments, one of the most popular being the algorithm BLAST \cite{Altschul}. This algorithm constructs a sequence alignment by looking at the local patches of a set of DNA strands and aligning them according to a given scoring system. While the right Kan extension of a sequence alignment functor proceeds in a similar fashion, its scoring system is more categorical than numerical -- this is discussed in Example \ref{exa:Global-local_seq_alignments}.

\begin{example}[Global alignments \emph{versus} local alignments]\label{exa:Global-local_seq_alignments}
In bioinformatics, the dynamic programming algorithm presented in section \ref{ssec:Main_example} is usually used with two main classes of scoring systems. The first class, known as the \emph{Needleman-Wunsch algorithm} \cite{NeedlemanWunsch}, aims to find global sequence alignments by initializing the comparison table with incremented gap penalties. This was the type of scoring system that we used in section \ref{ssec:Main_example}, in which gap penalties were incremented by 1. 
The second class is known as the \emph{Smith-Waterman algorithm} \cite{SmithWaterman} and is used to find local sequence alignments by initializing the comparison table with null gap penalties. 
Hybrid scoring systems that only set gap penalties to 0 for either the rows or the columns can be used to detect semi-global sequence alignments.
\[
\begin{array}{c|c|c}
\hline
\cellcolor[gray]{0.8}\textbf{Global} & \cellcolor[gray]{0.8}\textbf{Semi-global} &\cellcolor[gray]{0.8} \textbf{Local}\\
\hline
\xymatrix@C-30pt@R-15pt{
\ar[d]\bullet&\bullet&\bullet\ar[d]&\ar[rd]\bullet&\bullet&\bullet&\bullet&\bullet&\bullet\ar[rd]&\\
\bullet&\bullet&\bullet&\circ&\bullet&\bullet&\bullet&\bullet&\bullet&\bullet\\
}
&
\xymatrix@C-30pt@R-15pt{
\ar[d]\bullet&\bullet&\bullet\ar[d]&\ar[rd]\bullet&\bullet&\bullet&\bullet&\bullet&\bullet\ar[rd]&&&&\\
\bullet&\bullet&\bullet&\circ&\bullet&\bullet&\bullet&\bullet&\bullet&\bullet&\circ&\circ&\circ\\
}
&
\xymatrix@C-30pt@R-15pt{
&&&\ar[d]\bullet&\bullet&\bullet\ar[d]&\ar[rd]\bullet&\bullet&\bullet&\bullet&\bullet&\bullet\ar[rd]&&&&\\
\circ&\circ&\circ&\bullet&\bullet&\bullet&\circ&\bullet&\bullet&\bullet&\bullet&\bullet&\bullet&\circ&\circ&\circ\\
}
\\
\end{array}
\]
In this example, our goal is to show that the right Kan extension $\mathsf{Ran}_{\iota}T$ associated with the sequence alignment functor constructed in Example \ref{exa:sequence_alignments} captures both local and global aspects of the sequence alignment algorithms mentioned earlier. We will see that the local information is detected by the type of morphism presented in Example \ref{exa:Flexibility_morphism}. The subsequent discussion will show that local pieces of information often come with more uncertainty than global ones. The first part of our discussion will consist in computing the images of the four homologous segments $(!_{[8]},[\mathtt{1\separ0\separ0\separ0}])$, $(!_{[8]},[\mathtt{1\separ0\separ1\separ0}])$, $(!_{[8]},[\mathtt{1\separ0\separ0\separ1}])$ and $(!_{[8]},[\mathtt{1\separ0\separ1\separ1}])$ via $\mathsf{Ran}_{\iota}T$. Because the computation of these images are all very similar and all follow formula (\ref{eq:formula_right_kan_extension}), we will only detail the calculation of the image $\mathsf{Ran}_{\iota}T(!_{[8]},[\mathtt{1\separ0\separ1\separ1}])$ and directly give the images of the other segments.

To compute the image of the segment $(!_{[8]},[\mathtt{1\separ0\separ1\separ1}])$ via $\mathsf{Ran}_{\iota}T$, we need to look at the collection of objects of the category $((!_{[7]},[\mathtt{1\separ0\separ1\separ1}])\downarrow \iota)$ which consists of all the arrows of the following type in $\mathbf{Seg}(\Omega^{\times 4})$.
\[
(!_{[8]},[\mathtt{1\separ0\separ1\separ1}]) \to (!_{[8]},[\mathtt{1\separ0\separ0\separ1}])\quad\quad
(!_{[8]},[\mathtt{1\separ0\separ1\separ1}]) \to (!_{[8]},[\mathtt{1\separ0\separ1\separ0}])\quad\quad
(!_{[8]},[\mathtt{1\separ0\separ1\separ1}]) \to (!_{[8]},[\mathtt{0\separ0\separ1\separ1}])
\]
\[
(!_{[8]},[\mathtt{1\separ0\separ1\separ1}]) \to (!_{[8]},[\mathtt{1\separ0\separ0\separ0}])\quad\quad
(!_{[8]},[\mathtt{1\separ0\separ1\separ1}]) \to (!_{[8]},[\mathtt{0\separ0\separ1\separ0}])\quad\quad
(!_{[8]},[\mathtt{1\separ0\separ1\separ1}]) \to (!_{[8]},[\mathtt{0\separ0\separ0\separ1}])
\]
\[
(!_{[8]},[\mathtt{1\separ0\separ1\separ1}]) \to (!_{[9]},[\mathtt{0\separ0\separ1\separ1}])
\]
While the objects encoded by arrows in $\mathbf{Seg}(\Omega^{\times 4}\,|\,8)$ are unique (see Proposition \ref{prop:Seg_Omega_n_porder}), the objects of the type $(!_{[8]},[\mathtt{1\separ0\separ1\separ1}]) \to (!_{[9]},[\mathtt{0\separ0\separ1\separ1}])$ possess exactly 9 representatives.
After examining the relations existing between these objects in $((!_{[8]},[\mathtt{1\separ0\separ1\separ1}])\downarrow \iota)$, formula (\ref{eq:formula_right_kan_extension}) implies that the image of $\mathsf{Ran}_{\iota}T$ at the segment $(!_{[8]},[\mathtt{1\separ0\separ1\separ1}])$ is isomorphic to the set
\[
L_8([\mathtt{1\separ0\separ1\separ1}]) \times T(!_{[9]},[\mathtt{0\separ0\separ1\separ1}])^{\times 9}
\]
where we denote by $L_8([\mathtt{0111}])$ the limit of the diagram induced by the arrows of the category $((!_{[7]},[\mathtt{1\separ0\separ1\separ1}])\downarrow \iota)$ between the segments of the domain $[8]$ (see the diagram below).
\[
\xymatrix@C-15pt@R+5pt{
T(!_{[8]},[\mathtt{0\separ0\separ1\separ1}])\ar[rd]\ar[rrd]|>>>>>>>>>\hole&&T(!_{[8]},[\mathtt{1\separ0\separ1\separ0}])\ar[ld]\ar[rd]&&\ar[ld]\ar[lld]|>>>>>>>>>\hole T(!_{[8]},[\mathtt{1\separ0\separ0\separ1}])\\
&T(!_{[8]},[\mathtt{0\separ0\separ1\separ0}])&T(!_{[8]},[\mathtt{0\separ0\separ0\separ1}])&T(!_{[8]},[\mathtt{1\separ0\separ0\separ0}])&
}
\]
A similar analysis for the images of $\mathsf{Ran}_{\iota}T$ at the segments $(!_{[8]},[\mathtt{1\separ0\separ0\separ0}])$, $(!_{[8]},[\mathtt{1\separ0\separ1\separ0}])$, and $(!_{[8]},[\mathtt{1\separ0\separ0\separ1}])$ gives the following collection of isomorphisms.
\[
\mathsf{Ran}_{\iota}T(!_{[8]},[\mathtt{1\separ0\separ0\separ0}]) \cong 
T(!_{[8]},[\mathtt{1\separ0\separ0\separ0}])
\]
\[
\mathsf{Ran}_{\iota}T(!_{[8]},[\mathtt{1\separ0\separ1\separ0}]) \cong T(!_{[8]},[\mathtt{1\separ0\separ1\separ0}])
\]
\[
\mathsf{Ran}_{\iota}T(!_{[8]},[\mathtt{1\separ0\separ0\separ1}])\cong 
T(!_{[8]},[\mathtt{1\separ0\separ0\separ1}])
\]

Let us now use the images of the four segments $(!_{[8]},[\mathtt{1\separ0\separ0\separ0}])$, $(!_{[8]},[\mathtt{1\separ0\separ1\separ0}])$, $(!_{[8]},[\mathtt{1\separ0\separ0\separ1}])$ and $(!_{[8]},[\mathtt{1\separ0\separ1\separ1}])$ to see how the images of the right Kan extension integrate the local pieces of information available from the sequence alignment functor $(\iota,T,\sigma)$. First, the functoriality of $\mathsf{Ran}_{\iota}T$ gives us a commutative diagram as follows for the obvious choices of morphisms in $\mathbf{Seg}(\Omega\,|\,8)$.
\[
\xymatrix@C-20pt@R-10pt{
&\mathsf{Ran}_{\iota}T(!_{[8]},[\mathtt{1\separ0\separ1\separ1}])\ar[ld]\ar[rd]&\\
\mathsf{Ran}_{\iota}T(!_{[8]},[\mathtt{1\separ0\separ1\separ0}])\ar[rd]&&\ar[ld]\mathsf{Ran}_{\iota}T(!_{[8]},[\mathtt{1\separ0\separ0\separ1}])\\
&\mathsf{Ran}_{\iota}T(!_{[8]},[\mathtt{1\separ0\separ0\separ0}])&
}
\]
Even though the conclusion of Example \ref{exa:right_kan_extension_data_consistency} could suggest that this diagram is a pullback, our computation shows that the image $\mathsf{Ran}_{\iota}T(!_{[8]},[\mathtt{1\separ0\separ1\separ1}])$ is not isomorphic to the pullback of the lower part of the diagram and is only related to it via a projection of the following form.
\begin{equation}\label{eq:arrow_Ran_obstruction}
L_8([\mathtt{1\separ0\separ1\separ1}]) \times T(!_{[9]},[\mathtt{0\separ0\separ1\separ1}])^{\times 9} \longrightarrow T(!_{[8]},[\mathtt{1\separ0\separ1\separ0}]) \times_{\mathtt{x_{a}}} T(!_{[8]},[\mathtt{1\separ0\separ0\separ1}])
\end{equation}
Here, we can view the set $T(!_{[9]},[\mathtt{0\separ0\separ1\separ1}])^{\times 9}$ as (formally) containing the local sections of length 8 taken from the sequence alignments of length 9 associated with \texttt{Doug} and \texttt{Craig}. Arrow (\ref{eq:arrow_Ran_obstruction}) then tries to relate these local sections to the local sections of length 8 taken from the sequence alignment of \texttt{Anne} and \texttt{Craig} and that of \texttt{Anne} and \texttt{Doug}. However, arrow (\ref{eq:arrow_Ran_obstruction}) fails to map the elements of the product
\[
T(!_{[9]},[\mathtt{0\separ0\separ1\separ1}])^{\times 9}
\]
to elements in its codomain and is forced to forget these elements in the same way as a proper Cartesian projection map would do. This failure is not surprising since we designed the domain of $T$ so that the alignments of $T(!_{[9]},[\mathtt{0\separ0\separ1\separ1}])$ can never be connected to those of $T(!_{[8]},[\mathtt{0\separ0\separ1\separ1}])$ through $T(!_{[9]},[\mathtt{0\separ0\separ0\separ1}])$ and $T(!_{[9]},[\mathtt{0\separ0\separ1\separ0}])$.
The reason for this was to prevent the limit construction of the right Kan extension from forgetting the alignments of $T(!_{[8]},[\mathtt{0\separ0\separ1\separ1}])$ that were inconsistent with the alignments of $T(!_{[9]},[\mathtt{0\separ0\separ1\separ1}])$ (see Example \ref{exa:sequence_alignments}). In fact, the reader can check that if we had done so, then we would also have prevented the resulting version of (\ref{eq:arrow_Ran_obstruction}) from being a bijection.
This suggests that whatever method we try to use, the data contained in $T$ tends to prevent arrow (\ref{eq:arrow_Ran_obstruction}) from being a bijection. 

$\triangleright$ \textit{Interpretation:} This last fact actually hides important information about the relatedness of our four individuals. Indeed, the difference between the conclusion of the present example and that of Example \ref{exa:right_kan_extension_data_consistency}, in which we were able to show that a certain canonical arrow was a bijection, informs us that the genetic data of \texttt{Anne}, \texttt{Bob} and \texttt{Craig} are overall rather similar while the genetic data of \texttt{Anne}, \texttt{Craig} and \texttt{Doug} are much more different. This can already be seen in the sizes of the images of the functor $T$ at the segments $(!_{[8]},[\mathtt{0\separ0\separ1\separ1}])$ and $(!_{[9]},[\mathtt{0\separ0\separ1\separ1}])$, which are much larger than the images of $T$ at the other segments (see the table of Example \ref{exa:sequence_alignments}). In fact, the sizes of these images are related to the uncertainty of finding the right alignment for the sequences of \texttt{Craig} and \texttt{Doug}, and the simple fact that $T$ even has an image at the segment $(!_{[9]},[\mathtt{0\separ0\separ1\separ1}])$, while the genetic data of \texttt{Craig} and \texttt{Doug} is only of length 7, tells us that the dynamic programming algorithm is struggling to find an obvious match between the sequences of \texttt{Craig} and \texttt{Doug}.

$\triangleright$ \textit{Conclusion:} We see that the obstruction -- or rather the uncertainty -- resulting from aligning a set of distant DNA sequences is detected by the ability of certain canonical arrows to be isomorphisms or epimorphisms. It is precisely for these reasons that the concepts of chromology and pedigrad become relevant to the study of our main example.
\end{example}

\subsection{Exactly distributive and injective chromologies}\label{ssec:Distributive_and_exactly_distributive_chromologies}
The goal of the present section is to define two canonical classes of chromologies. As usual, we let $(\Omega,\preceq)$ be a pre-ordered set, $b$ be an element in $\Omega$, $A$ be a small category, $\tau$ be an object in $\mathbf{Seg}(\Omega\,|\,n)$ and $\rho:\Delta_{A}(\tau) \Rightarrow \theta$ be a cone in $\mathbf{Seg}(\Omega\,|\,n)$ for some non-negative integer $n$. First, note that the application of the truncation functor $\mathsf{Tr}_b:\mathbf{Seg}(\Omega\,|\,n) \to \mathbf{Set}^{\mathrm{op}}$ on the cone $\rho$ gives rise to a cocone in $\mathbf{Set}$ as follows.
\[
\mathsf{Tr}_b(\rho):\mathsf{Tr}_b\theta \Rightarrow \Delta_{A} \circ\mathsf{Tr}_b(\tau)
\]
The colimit adjoint of this natural transformation in $\mathbf{Set}$ gives us a function as follows.
\begin{equation}\label{eq:epi_mono_factorization_cocone_for_definition_distributive_cones}
\mathsf{colim}_A\mathsf{Tr}_b(\rho):\mathsf{colim}_A\mathsf{Tr}_b\theta \longrightarrow  \mathsf{Tr}_b(\tau)
\end{equation}

\begin{definition}[Exactly distributive cones]\label{def:exactly_distributive_cones}
A cone of the form $\rho:\Delta_{A}(\tau) \Rightarrow \theta$ in $\mathbf{Seg}(\Omega\,|\,n)$  will be said to be \emph{exactly $b$-distributive} if the arrow of (\ref{eq:epi_mono_factorization_cocone_for_definition_distributive_cones}) is an isomorphism in $\mathbf{Set}$.
\end{definition}

\begin{definition}[Injective cones]\label{def:injective_cones}
A cone of the form $\rho:\Delta_{A}(\tau) \Rightarrow \theta$ in $\mathbf{Seg}(\Omega\,|\,n)$  will be said to be \emph{$b$-injective} if the arrow of (\ref{eq:epi_mono_factorization_cocone_for_definition_distributive_cones}) is a monomorphism in $\mathbf{Set}$.
\end{definition}

\begin{example}[Exactly distributive cones]\label{exa:Exactly_distributive_cones}
Let $\Omega$ denote the pre-ordered set $\{0 \leq 1 \leq 2\}$. In this example, we give various instance of exactly distributive cones and injective cones in $\mathbf{Seg}(\Omega)$. Before showing these instances, let us mention that a cone $\Delta(\tau) \Rightarrow \theta$ should be seen as a structure specifying an integration operation from the diagram $\theta$ to the object $\tau$ -- this may be useful to understand what these cones are meant to specify.

First, we can give the following diagram, living in one of the pre-order categories $\mathbf{Seg}(\Omega:t)$ for the obvious topology $t$ of domain $[12]$, as an example of an exactly 1-distributive cone, but also as an example of a 2-injective cone.
\[
\xymatrix@C-30pt@R-15pt{
&&&&&&&&&&&&&(\mathtt{0}&\mathtt{0}&\mathtt{0})&(\mathtt{1}&\mathtt{1})&(\mathtt{1}&\mathtt{1}&\mathtt{1})&(\mathtt{0}&\mathtt{0}&\mathtt{0}&\mathtt{0})\ar[rrd]&&&&&&&&&&&&&\\
(\mathtt{2}\ar@{}@<-2ex>[rrrrrrrrrrr]_{\textrm{\large$(\tau)$}}&\mathtt{2}&\mathtt{2})&(\mathtt{1}&\mathtt{1})&(\mathtt{2}&\mathtt{2}&\mathtt{2})&(\mathtt{2}&\mathtt{2}&\mathtt{2}&\mathtt{2})\ar[rru]\ar[rr]\ar[rrd]&\quad\quad\quad\quad&(\mathtt{0}&\mathtt{0}&\mathtt{0})&(\mathtt{1}&\mathtt{1})&(\mathtt{0}&\mathtt{0}&\mathtt{0})&(\mathtt{2}&\mathtt{2}&\mathtt{2}&\mathtt{2})\ar[rr]&\quad&(\mathtt{0}&\mathtt{0}&\mathtt{0})&(\mathtt{1}&\mathtt{1})&(\mathtt{0}&\mathtt{0}&\mathtt{0})&(\mathtt{0}&\mathtt{0}&\mathtt{0}&\mathtt{0})\\
&&&&&&&&&&&&&(\mathtt{2}&\mathtt{2}&\mathtt{2})&(\mathtt{1}&\mathtt{1})&(\mathtt{0}&\mathtt{0}&\mathtt{0})&(\mathtt{0}&\mathtt{0}&\mathtt{0}&\mathtt{0})\ar[rru]&&&&&&&&&&&&&
}
\]
Here, the idea is that the positions of the nodes of color 1 or above in the middle and right segments of the diagram are the union of the positions of the nodes of color 1 or above in the leftmost segment ($\tau$).
Intuitively, one could imagine to use this type of cone to integrate different alignment methods on the objects of the diagram $\theta$ into a unique one on the object $\tau$.

We now give the following diagram as an example of a exactly 1-distributive cone and a 2-injective cone in the category of quasi-homologous segments $\mathbf{Seg}(\Omega\,|\,12)$.
\[
\xymatrix@C-30pt@R-15pt{
&&&&&&&&&&&&&(\mathtt{0}&\mathtt{0}&\mathtt{0})&(\mathtt{1}&\mathtt{1})&(\mathtt{1}&\mathtt{1}&\mathtt{1})&(\mathtt{0}&\mathtt{0}&\mathtt{0}&\mathtt{0})\ar[rrd]&&&&&&&&&&&&&\\
(\mathtt{2})\ar@{}@<-2ex>[rrrrrrrrrrr]_{\textrm{\large$(\tau)$}}&(\mathtt{2})&(\mathtt{2})&(\mathtt{1}&\mathtt{1})&(\mathtt{2}&\mathtt{2})&(\mathtt{2})&(\mathtt{2}&\mathtt{2})&(\mathtt{2}&\mathtt{2})\ar[rru]\ar[rr]\ar[rrd]&\quad\quad\quad\quad&(\mathtt{0}&\mathtt{0}&\mathtt{0})&(\mathtt{1}&\mathtt{1})&(\mathtt{0}&\mathtt{0}&\mathtt{0})&(\mathtt{2}&\mathtt{2}&\mathtt{2}&\mathtt{2})\ar[rr]&\quad&(\mathtt{0}&\mathtt{0}&\mathtt{0})&(\mathtt{1}&\mathtt{1})&(\mathtt{0}&\mathtt{0}&\mathtt{0})&(\mathtt{0}&\mathtt{0}&\mathtt{0}&\mathtt{0})\\
&&&&&&&&&&&&&(\mathtt{2}&\mathtt{2}&\mathtt{2})&(\mathtt{1}&\mathtt{1})&(\mathtt{0}&\mathtt{0}&\mathtt{0})&(\mathtt{0}&\mathtt{0}&\mathtt{0}&\mathtt{0})\ar[rru]&&&&&&&&&&&&&
}
\]
The difference between the very first cone and the one given above is that the latter specifies an integration operation whose action also applies to a more refined topology (on $\tau$). 
For instance, aligning a set of DNA strands with respect to the codon topology will necessarily align the DNA strands with respect to the nucleotide topology.

Finally, the following arrow in $\mathbf{Seg}(\Omega\,|\,12)$ is an example of an exactly $0$-distributive cone as well as an example of a $1$-injective cone.
\[
\xymatrix@C-30pt{
(\mathtt{1})\ar@{}@<-2ex>[rrrrrrrrrrr]_{\textrm{\large$(\tau)$}}&(\mathtt{1}&\mathtt{1})&(\mathtt{1}&\mathtt{1})&(\mathtt{1}&\mathtt{1})&(\mathtt{1})&(\mathtt{1}&\mathtt{1}&\mathtt{1})&(\mathtt{1})\ar[rr]&\quad\quad\quad&
(\mathtt{0}&\mathtt{0}&\mathtt{0})&(\mathtt{1}&\mathtt{1})&(\mathtt{0}&\mathtt{0}&\mathtt{0})&(\mathtt{1}&\mathtt{1}&\mathtt{1}&\mathtt{1})
}
\]
\end{example}

\begin{definition}[Exactly distributive chromologies]\label{def:exactly_distributive_chromologies}
Let $b$ be an element in $\Omega$. A chromology $(\Omega,D)$ will be said to be \emph{exactly $b$-distributive} if all the cones in $D$ are exactly $b$-distributive.
\end{definition}

\begin{definition}[Injective chromologies]\label{def:injective_chromologies}
Let $b$ be an element in $\Omega$. A chromology $(\Omega,D)$ will be said to be \emph{$b$-injective} if all the cones in $D$ are $b$-injective.
\end{definition}

\subsection{Logical systems for pedigrads in sets}
In this section, we show that the functors defined in Definition \ref{def:set_E_b_varepsilon} are pedigrads in two different logical systems of $\mathbf{Set}$ for two different types of chromologies.

\begin{definition}[Logical systems of bijections]\label{def:logical_systems_cones_iso_limits}
We will denote by $\mathcal{W}^{\textrm{bij}}$ the class of cones $\Delta_{A}(X) \Rightarrow F$ in $\mathbf{Set}$ whose limit adjoints $X \to \mathsf{lim}_{A}F$ are bijections.
\end{definition}

In section \ref{sec:Solving_main_example}, we will show that one can use the following theorem to study the information contained in a sequence alignment functor (Definition \ref{def:sequence_alignment}).

\begin{theorem}\label{theo:E_b_varepsilon_W_iso_pedigrad_exactly_distributive}
For every element $b$ in $\Omega$ and exactly $b$-distributive chromology $(\Omega,D)$, the functor $E_{b}^{\varepsilon}:\mathbf{Seg}(\Omega) \to \mathbf{Set}$ is a $\mathcal{W}^{\textrm{bij}}$-pedigrad for $(\Omega,D)$.
\end{theorem}
\begin{proof}
Let $\rho:\Delta_{A}(\tau) \Rightarrow \theta$ be a cone in $D[n_1]$ for some given non-negative integer $n_1$. Because $(\Omega,D)$ is an exactly $b$-distributive chromology, it follows from Definition \ref{def:exactly_distributive_cones} and Definition \ref{def:exactly_distributive_chromologies} that the canonical arrow
\[
\mathsf{colim}_A\mathsf{Tr}_b\theta \to \mathsf{Tr}_b(\tau)
\]
is an isomorphism in $\mathbf{Set}$. As a result, the image of this arrow via the functor $\mathbf{Set}(\_,E):\mathbf{Set}^{\mathrm{op}} \to \mathbf{Set}$ is a bijection. By Proposition \ref{prop:pedigrad_representable} and the usual definition of colimits in $\mathbf{Set}$, the resulting bijection is (naturally) isomorphic to the following canonical arrow.
\[
E_{b}^{\varepsilon}(\tau) \to \mathsf{lim}_A E_{b}^{\varepsilon}\circ \theta
\]
This precisely shows that $E_{b}^{\varepsilon}:\mathbf{Seg}(\Omega) \to \mathbf{Set}$ is a $\mathcal{W}^{\textrm{bij}}$-pedigrad for $(\Omega,D)$.
\end{proof}

\begin{definition}[Logical systems of surjections]\label{def:logical_systems_cones_surj_limits}
We will denote by $\mathcal{W}^{\textrm{surj}}$ the class of cones $\Delta_{A}(X) \Rightarrow F$ in $\mathbf{Set}$ whose limit adjoints $X \to \mathsf{lim}_{A}F$ are surjections.
\end{definition}

\begin{theorem}\label{theo:E_b_varepsilon_W_surj_pedigrad_injective}
For every element $b$ in $\Omega$ and $b$-injective chromology $(\Omega,D)$, the functor $E_{b}^{\varepsilon}:\mathbf{Seg}(\Omega) \to \mathbf{Set}$ is a $\mathcal{W}^{\textrm{surj}}$-pedigrad for $(\Omega,D)$.
\end{theorem}
\begin{proof}
Before showing the statement, recall that for every monomorphism $m:A \to B$ in $\mathbf{Set}$, the function $\mathbf{Set}(j,E):\mathbf{Set}(B,E) \to \mathbf{Set}(A,E)$ is a surjection. Indeed, because $E$ has a pointed structure, every function $f:A \to E$ can be extended to a function $f':B \to E$ by mapping every $x \in B \backslash A$ to the point $\varepsilon$ of $E$. We can check that the identity $f = f' \circ j$ holds, which amounts to saying that the image of $j$ via the functor $\mathbf{Set}(\_,E):\mathbf{Set}^{\mathrm{op}} \to \mathbf{Set}$ is a surjection.

We now prove the statement. Let $\rho:\Delta_{A}(\tau) \Rightarrow \theta$ be a cone in $D[n_1]$ for some given non-negative integer $n_1$. Because $(\Omega,D)$ is a $b$-injective chromology, it follows from Definition \ref{def:exactly_distributive_cones} and Definition \ref{def:exactly_distributive_chromologies} that the canonical arrow
\[
\mathsf{colim}_A\mathsf{Tr}_b\theta \to \mathsf{Tr}_b(\tau)
\]
is a monomorphism in $\mathbf{Set}$. As a result, the image of this arrow via the functor $\mathbf{Set}(\_,E):\mathbf{Set}^{\mathrm{op}} \to \mathbf{Set}$ is a surjection. By Proposition \ref{prop:pedigrad_representable} and the usual definition of colimits in $\mathbf{Set}$, the resulting surjection is (naturally) isomorphic to the following canonical arrow.
\[
E_{b}^{\varepsilon}(\tau) \to \mathsf{lim}_A E_{b}^{\varepsilon}\circ \theta
\]
This precisely shows that $E_{b}^{\varepsilon}:\mathbf{Seg}(\Omega) \to \mathbf{Set}$ is a $\mathcal{W}^{\textrm{surj}}$-pedigrad for $(\Omega,D)$.
\end{proof}

\begin{example}[Controlling uncertainties via chromologies]\label{exa:uncertainty_vs_chromologies}
The goal of this example is to show that the presence of uncertainties
 discussed at the end of Example \ref{exa:Global-local_seq_alignments} can be controlled through the design of chromologies. We will keep the same notations as those used therein. First, recall that the point of Example \ref{exa:right_kan_extension_data_consistency} was to show that the following exactly $(1,1,1,1)$-distributive cone in $\mathbf{Seg}(\Omega^{\times 4})$ was suitable to define an exactly $(1,1,1,1)$-distributive chromology that would make the right Kan extension $\mathsf{Ran}_{\iota}T:\mathbf{Seg}(\Omega^{\times 4}) \to \mathbf{Set}$ a $\mathcal{W}^{\textrm{bij}}$-pedigrad.
\[
\xymatrix@C-25pt@R-5pt{
&&(!_{[8]},[\mathtt{1\separ1\separ1\separ0}])\ar[dll]\ar[d]\ar[drr]&&\\
(!_{[8]},[\mathtt{1\separ0\separ1\separ0}])\ar[rd]\ar[rrd]|>>>>>>\hole&&(!_{[8]},[\mathtt{0\separ1\separ1\separ0}])\ar[ld]\ar[rd]&&\ar[ld]\ar[lld]|>>>>>>\hole (!_{[8]},[\mathtt{1\separ1\separ0\separ0}])\\
&(!_{[8]},[\mathtt{0\separ0\separ1\separ0}])\quad\quad&(!_{[8]},[\mathtt{1\separ0\separ0\separ0}])&\quad\quad(!_{[8]},[\mathtt{0\separ1\separ0\separ0}])&
}
\]
On the other hand, the point of Example \ref{exa:Global-local_seq_alignments} was to show that not every exactly $(1,1,1,1)$-distributive cone is suitable to make a functor a $\mathcal{W}^{\textrm{bij}}$-pedigrad. In particular, it was shown that the following cone could not make the right Kan extension $\mathsf{Ran}_{\iota}T$ a $\mathcal{W}^{\textrm{bij}}$-pedigrad, but only a $\mathcal{W}^{\textrm{surj}}$-pedigrad
\begin{equation}\label{eq:cone_for_ran_surj_pedigrad}
\xymatrix@C-20pt@R-10pt{
&(!_{[8]},[\mathtt{1\separ0\separ1\separ1}])\ar[ld]\ar[rd]&\\
(!_{[8]},[\mathtt{1\separ0\separ1\separ0}])\ar[rd]&&\ar[ld](!_{[8]},[\mathtt{1\separ0\separ0\separ1}])\\
&(!_{[8]},[\mathtt{1\separ0\separ0\separ0}])&
}
\end{equation}
The reason for this obstruction was that the genetic data of \texttt{Craig} was much different from that of \texttt{Doug}. As a result, the image of the functor $T:B \to \mathbf{Set}$ at the segment $(!_{[9]},[\mathtt{0\separ0\separ1\separ1}])$ was non-trivial and thus prevented the canonical limit arrow associated with the previous cone from being a bijection. At least, knowing that it is a surjection tells us that there is no inconsistencies between the table of \texttt{Anne} and \texttt{Craig} and that of \texttt{Anne} and \texttt{Doug}

Note that the obstruction associated with (\ref{eq:cone_for_ran_surj_pedigrad}) to make $\mathsf{Ran}_{\iota}T$ a $\mathcal{W}^{\textrm{bij}}$-pedigrad could be reduced if we could prevent the objects of (\ref{eq:cone_for_ran_surj_pedigrad}) from going to segments of $B$ that are associated with this so-called uncertainty. This would prevent the limit construction of 
$\mathsf{Ran}_{\iota}T$ from considering too many images of $T$ in its computation. 

A way to do so could be to change the topology of the segments on which the functor $T:B \to \mathbf{Set}$ is defined. Of course, being able to do so would mean that we either know more about our problem or that we make an assumption about our four individuals. For example, we could consider the situation in which one decides to parenthesize all the adjacent matches that appear in the pairwise sequence alignments of $T$ together. Here is an illustration. First, recall that the image of $T(!_{[9]},[\mathtt{0\separ0\separ1\separ1}])$ was taken to be as follows (see Example \ref{exa:sequence_alignments}).

\[
T(!_{[9]},[\mathtt{0\separ0\separ1\separ1}]) = \left\{
\begin{array}{l}\!\!\!\!\!\!
\begin{array}{l}
\mathtt{ACCGT\varepsilon C\varepsilon A}\\
\mathtt{A\varepsilon C\varepsilon TACTG}
\end{array}\!;\,
\begin{array}{l}
\mathtt{AC\varepsilon \varepsilon CGTCA}\\
\mathtt{ACTAC\varepsilon T\varepsilon G}
\end{array}\!;\,
\begin{array}{l}
\mathtt{ACCGT\varepsilon CA\varepsilon }\\
\mathtt{A\varepsilon C\varepsilon TACTG}
\end{array}\!;\,
\begin{array}{l}
\mathtt{AC\varepsilon \varepsilon CGTCA}\\
\mathtt{ACTAC\varepsilon TG\varepsilon }
\end{array}\!\!\!\!\!\!
\end{array}
\right\}
\]
If we now parenthesize the matching and mismatching patches of maximal lengths in the sequence alignments of $T(!_{[9]},[\mathtt{0\separ0\separ1\separ1}])$, we obtain the parenthesized sequence alignments shown in the following table. 
\[
\begin{array}{|ll|}
\hline
\multicolumn{2}{|c|}{\cellcolor[gray]{0.8}{T(!_{[9]},[\mathtt{0\separ0\separ1\separ1}])}}\\
\hline
T(t_1,[\mathtt{0\separ0\separ1\separ1}]) = &
\left\{
\begin{array}{l}
\mathtt{(A)(C)(C)(G)(T)(\varepsilon) (C)(\varepsilon A)}\\
\mathtt{(A)(\varepsilon) (C)(\varepsilon)(T)(A)(C)(TG)}
\end{array}\!,\,
\begin{array}{l}
\mathtt{(A)(C)(C)(G)(T)(\varepsilon) (C)(A\varepsilon) }\\
\mathtt{(A)(\varepsilon) (C)(\varepsilon) (T)(A)(C)(TG)}
\end{array}
\right\}\\
T(t_2,[\mathtt{0\separ0\separ1\separ1}]) =&
\left\{\begin{array}{l}
\mathtt{(AC)(\varepsilon \varepsilon) (C)(G)(T)(CA)}\\
\mathtt{(AC)(TA)(C)(\varepsilon) (T)(\varepsilon G)}
\end{array}\!,\,
\begin{array}{l}
\mathtt{(AC)(\varepsilon \varepsilon) (C)(G)(T)(CA)}\\
\mathtt{(AC)(TA)(C)(\varepsilon) (T)(G\varepsilon) }
\end{array}
\right\}\\
\hline
\end{array}
\]
Because this bracketing suggests the use of two new topologies, we want to stay consistent with the definition of $T$ and associate the previous alignments with the two segments of non-terminal topologies shown on the left-hand side of the previous table. Since a segment of trivial topology $!_{[8]}:[8] \to [1]$ cannot be mapped to segments of non-terminal topologies, the domain of the canonical arrow
\begin{equation}\label{eq:ran_surjection_to_bijection}
\mathsf{Ran}_{\iota}T(!_{[8]},[\mathtt{1\separ0\separ1\separ1}]) \to \mathsf{Ran}_{\iota}T \circ F,
\end{equation}
where the diagram $F$ is the lower cospan of diagram (\ref{eq:cone_for_ran_surj_pedigrad}), will not contain the sets $T(t_1,[\mathtt{0\separ0\separ1\separ1}])$ and $T(t_2,[\mathtt{0\separ0\separ1\separ1}])$ (and the term $T(!_{[9]},[\mathtt{0\separ0\separ1\separ1}])^{\times 9}$ no longer appears in the domain). This means that arrow (\ref{eq:ran_surjection_to_bijection}) is more likely to be a bijection of sets, hence making diagram (\ref{eq:cone_for_ran_surj_pedigrad}) more likely to be suitable for the definition of a chromology that makes $\mathsf{Ran}_{\iota}T$ a $\mathcal{W}^{\textrm{bij}}$-pedigrad.

We thus conclude that the process of looking for chromologies (or, in fact, their cones), given a set of possible functors $T:B \to \mathbf{Set}$, can be seen as a way of isolating uncertainties and hence producing a refined analysis of the genetic data of \texttt{Anne}, \texttt{Bob}, \texttt{Craig} and \texttt{Doug}.
\end{example}

\section{Solving our problem and identifying mechanisms}\label{sec:Solving_main_example}

In this section, we formalize what should be seen as the categorical answer of the problem exposed in section \ref{ssec:Main_example}, namely a method to assess the validity of the multiple sequence alignments computed by a right Kan extension of a sequence alignment functor. More specifically, we show that chromologies give a way to select multiple sequence alignments with respect to various types of mechanisms. As usual, we will let $(E,\varepsilon)$ denote a pointed set and $(\Omega,\preceq)$ be a pre-ordered set. 

\subsection{Link between right Kan extensions and pedigrads}
In this section, we recall the definition of the unit associated with a right Kan extension and relate this unit to the concepts of pedigrad and chromology. 

\begin{convention}[Notations]
For every small category $C$, we will denote by $[C,\mathbf{Set}]$ the category whose objects are functors from $C$ to $\mathbf{Set}$ and whose arrows are natural transformations. 
\end{convention}

The following proposition states that a right Kan extension in $\mathbf{Set}$ along a certain functor $\iota$ (Definition \ref{def:right_kan_extensions}) is a right adjoint for the pre-composition functor induced by the functor $\iota$ -- see \cite[Chapter X]{MacLane} for more detail.

\begin{theorem}[\cite{MacLane}]\label{theo:RKE_right_adjoint}
Let $(\Omega,\preceq)$ be a pre-ordered set and $\iota:B \to \mathbf{Seg}(\Omega)$ be a functor. The right Kan extension operation $\mathsf{Ran}_{\iota}$ induces a functor $[B,\mathbf{Set}] \to [\mathbf{Seg}(\Omega),\mathbf{Set}]$ that maps any natural transformation $\sigma:A \Rightarrow B$ in $\mathbf{Set}$ over $B$ to the canonical natural transformation (\ref{eq:alignment_with_Ran}) induced by the limit construction of Definition \ref{def:right_kan_extensions}.
\begin{equation}\label{eq:alignment_with_Ran}
\mathsf{Ran}_{\iota}\sigma:\mathsf{lim}_{({\tau \downarrow \iota})} A \circ \iota_{\tau}\Rightarrow \mathsf{lim}_{({\tau \downarrow \iota})} B \circ \iota_{\tau}
\end{equation}
The resulting functor $\mathsf{Ran}_{\iota}:[B,\mathbf{Set}] \to [\mathbf{Seg}(\Omega),\mathbf{Set}]$ is a right adjoint for the pre-composition functor induced by $\iota$.
\end{theorem}

\begin{remark}[Units of right Kan extensions]\label{rem:units_RKE}
The present remark reminds the reader about the form of the unit associated with the adjunction described in Theorem \ref{theo:RKE_right_adjoint}.
Let $(\Omega,\preceq)$ be a pre-ordered set and $\iota:B \to \mathbf{Seg}(\Omega)$ be a functor. As a right adjoint of the pre-composition functor 
\[
\_ \circ \iota: [\mathbf{Seg}(\Omega),\mathbf{Set}] \to [B,\mathbf{Set}],
\]
the functor $\mathsf{Ran}_{\iota}$ is associated with a natural transformation, called the \emph{unit}, of the form shown in (\ref{eq:alignment_with_Ran_unit}) for every functor $P:\mathbf{Seg}(\Omega) \to \mathbf{Set}$.
\begin{equation}\label{eq:alignment_with_Ran_unit}
\eta:P \Rightarrow \mathsf{Ran}_{\iota}(P \circ \iota)
\end{equation}
The components of (\ref{eq:alignment_with_Ran_unit}) correspond to the canonical arrows associated with the limit construction of Definition \ref{def:right_kan_extensions}. This means that the evaluation of the previous natural transformation at an object $\tau$ in $\mathbf{Seg}(\Omega)$, as shown below in (\ref{eq:comparing_pedigrads_unit}), is the limit adjoint arrow for the cone described in Remark \ref{rem:Extending_category_as_a_cone}.
\begin{equation}\label{eq:comparing_pedigrads_unit}
P(\tau) \to \mathsf{lim}_{({\tau \downarrow \iota})}P \circ \iota_{\tau} \circ \iota
\end{equation}
\end{remark}

\begin{remark}[Right Kan extensions and pedigrads]\label{rem:RKE_and_pedigrads}
The present remark extends Remark \ref{rem:units_RKE} and shows how right Kan extensions and pedigrads are related.
First, note that the arrow shown in (\ref{eq:comparing_pedigrads_unit}) looks a lot like the type of arrow used in Definition \ref{def:logical_systems_cones_iso_limits} and Definition \ref{def:logical_systems_cones_surj_limits} to define  $\mathcal{W}^{\textrm{bij}}$- and $\mathcal{W}^{\textrm{surj}}$-pedigrads. To make this more precise, let us denote by $\rho_{\iota}[\tau]$ the cone defined in Remark \ref{rem:Extending_category_as_a_cone}, that is to say the obvious natural transformation
\[
\Delta_{({\tau \downarrow \iota})}(\tau)\Rightarrow  \iota_{\tau} \circ \iota
\]
induced by the objects of the category $({\tau \downarrow \iota})$ over the diagram formed by its arrows. Let us also suppose that the cone $\rho_i[\tau]$ is part of a chromology $(\Omega,D)$. In this case, we can notice two facts:
\begin{itemize}
\item[-] if $P$ is a $\mathcal{W}^{\textrm{bij}}$-pedigrad for $(\Omega,D)$, then arrow (\ref{eq:comparing_pedigrads_unit}) is a bijection. 
\item[-] if $P$ is a $\mathcal{W}^{\textrm{surj}}$-pedigrad for $(\Omega,D)$, then arrow (\ref{eq:comparing_pedigrads_unit}) is a surjection. 
\end{itemize}
Later on, we will use these two facts with some functor $P$ that is the composition of functors of the form described in Definition \ref{def:set_E_b_varepsilon} and Example \ref{exa:preparation_example} (for instance, see Convention \ref{conv:E_Seg_f_i}). In particular, if the pedigrad $P$ can be written as a composite $Q \circ R$ where $Q$ is a more canonical pedigrad, then we will prefer to apply the previous two points to $Q$, namely by looking at whether the image of the cone $\rho_{\iota}[\tau]$ via $R$ is in the chromology of $Q$.
\end{remark}

\subsection{Slices of a sequence alignment}\label{sec:slices_seq_alignment}
In this section, we define the concept of slice for a sequence alignment functor. This concept will later be used to reason about the possible mutation mechanisms contained in a sequence alignment functor. Throughout this section, we shall let $\mathbf{A}$ be an alignment specification of the form $\{f_i:(\Omega,\preceq) \to (\Omega_i,\preceq_i)\}_{i \in A}$.

\begin{remark}[Projection maps]\label{rem:aligned_pedigrad_projections}
Let $(E,\varepsilon)$ be a pointed set. For every element $b$ in $\Omega$, the product structure of $\mathbf{A}E_b^{\epsilon}$ (see Definition \ref{def:Aligned_pedigrads}) gives us the following natural projection for every $i \in A$.
\[
\kappa_i:\mathbf{A}E_{b}^{\varepsilon} \Rightarrow E_{f_i(b)}^{\varepsilon} \circ \mathbf{Seg}(f_i)
\]
This arrow, living in $[\mathbf{Seg}(\Omega),\mathbf{Set}]$, will repeatedly be used throughout this section.
\end{remark}

\begin{convention}[Notation]\label{conv:E_Seg_f_i}
For the sake of convenience, for every element $i \in A$ and element $b \in \Omega$, we will let $f_i^{*}E_{b}^{\varepsilon}$ denote the composite functor $E_{f_i(b)}^{\varepsilon} \circ \mathbf{Seg}(f_i):\mathbf{Seg}(\Omega)\to\mathbf{Set}$.  
\end{convention}

In Example \ref{exa:Global-local_seq_alignments}, we saw that, even for  exactly distributive chromologies (Definition \ref{def:exactly_distributive_chromologies}), the right Kan extension of a sequence alignment functor is not necessarily a $\mathcal{W}^{\textrm{bij}}$-pedigrad. On the other hand, the functor of Definition \ref{def:set_E_b_varepsilon} was shown to be a $\mathcal{W}^{\mathrm{bij}}$-pedigrad for any such cone (Theorem \ref{theo:E_b_varepsilon_W_iso_pedigrad_exactly_distributive}). The idea of Definition \ref{def:Slices}, given below, is to compare a non-pedigrad object to a pedigrad object in order to detect the pieces of pedigradic information that would live in the non-pedigrad object.

\begin{definition}[Slices]\label{def:Slices}
Let $b$ be an element in $\Omega$ and $(\iota, T,\sigma)$ be a sequence alignment functor over $\mathbf{A}E_{b}^{\varepsilon}$. For every element $i \in A$, we will speak of the \emph{$i$-slice} of $(\iota, T,\sigma)$ to refer to the pullback arrow $\eta^{*}_i:[T/\mathbf{A}E_{b}^{\varepsilon}]_i \Rightarrow \mathsf{Ran}_{\iota}T$ of the unit of the right Kan extension at the functor $f_i^{*}E_{b}^{\varepsilon}$ along the natural transformation $\mathsf{Ran}_{\iota}(\kappa_i \circ \sigma)$ (as shown below).
\begin{equation}\label{eq:pullback_sequence_alignment}
\xymatrix{
[T/\mathbf{A}E_{b}^{\varepsilon}]_i\ar@{}[rd]|<<<{\rotatebox[origin=c]{90}{\huge{\textrm{$\llcorner$}}}}\ar@{=>}[d]_{\eta^{\ast}_i}\ar@{=>}[rr]&&f_i^{*}E_{b}^{\varepsilon}\ar@{=>}[d]^{\eta}\\
\mathsf{Ran}_{\iota}T \ar@{=>}[r]_-{\mathsf{Ran}_{\iota}\sigma}&\mathsf{Ran}_{\iota}(\mathbf{A}E_{b}^{\varepsilon} \circ \iota)\ar@{=>}[r]_-{\mathsf{Ran}_{\iota}\kappa_i}&\mathsf{Ran}_{\iota}(f_i^{*}E_{b}^{\varepsilon} \circ \iota)
}
\end{equation}
\end{definition}

\begin{remark}[General \emph{versus} individual slices]
Throughout the present section, the reader may wonder why we only consider the pullback of $\eta$ along a natural transformation of the form $\mathsf{Ran}_{\iota}(\kappa_i \circ \sigma)$ while the pullback of $\eta$ along the natural transformation $\mathsf{Ran}_{\iota}(\sigma)$ is left out. The reason is that the former integrates the data with respect to a unique `individual' $i\in A$ while the latter consider all `individuals' in $A$. As a result, the latter will contain very few elements, if any. The algorithm proposed in Remark \ref{rem:finding_right_seq_alignment} tries to maximize both the size of the integrated data and the number of individuals considered.
\end{remark}


\begin{remark}[Data integration along cones]\label{rem:resolving_uncertainties}
The idea behind the slice of a sequence alignment functor is to select the multiple sequence alignments of $\mathsf{Ran}_{\iota}T$ for which the type of uncertainty described in Example \ref{exa:Global-local_seq_alignments} and Example \ref{exa:uncertainty_vs_chromologies} can be resolved from the point of view of a particular individual. While Example \ref{exa:lifting_gluings_of_tables} will illustrate how this type of uncertainty can be eliminated, the present remark shows how the cones of a chromology inform us of the ways via which we can safely integrate the data through slices.

Let $b$ be an element in $\Omega$, $(\iota, T,\sigma)$ be a sequence alignment functor over $\mathbf{A}E_{b}^{\varepsilon}$ and $\tau$ be an object in $\mathbf{Seg}(\Omega)$. For every element $i \in A$, evaluating the $i$-slice of $(\iota, T,\sigma)$ at $\tau$ gives us a pullback square in $\mathbf{Set}$ as follows.
\begin{equation}\label{eq:pullback_sequence_alignment_detail}
\xymatrix@C+15pt{
[T/\mathbf{A}E_{b}^{\varepsilon}]_i(\tau)\ar@{}[rd]|<<<{\rotatebox[origin=c]{90}{\huge{\textrm{$\llcorner$}}}}\ar[d]\ar[r]^-{\sigma^{*}_{\tau}}&f_i^{*}E_{b}^{\varepsilon}(\tau)\ar[d]^{\eta_{\tau}}\\
\mathsf{lim}_{({\tau \downarrow \iota})}T(\tau) \ar[r]_-{\mathsf{Ran}_{\iota}(\kappa_i \circ \sigma)_{\tau}}&*+!L(.4){\mathsf{lim}_{({\tau \downarrow \iota})}f_i^{*}E_{b}^{\varepsilon} \circ \iota_{\tau} \circ \iota}
}
\end{equation}
Now, recall that, by universal property of pullbacks, the pullback of an isomorphism (resp. a surjection) is an isomorphism (resp. a surjection). By Remark \ref{rem:RKE_and_pedigrads}, this means that if
\begin{itemize}
\item[-] the image of the cone $\rho_{\iota}[\tau]$ via $\mathbf{Seg}(f_i)$ is in the chromology of $E_{f_i(b)}^{\varepsilon}$, and
\item[-] the functor $E_{f_i(b)}^{\varepsilon}$ is a $\mathcal{W}^{\textrm{bij}}$-pedigrad (resp. $\mathcal{W}^{\textrm{surj}}$-pedigrad),
\end{itemize}
then the leftmost vertical arrow of diagram (\ref{eq:pullback_sequence_alignment_detail}) is an isomorphism (resp. a surjection). In this sense, we would like to say that all the gluings computed by $\mathsf{Ran}_{\iota}T$ make sense (resp. make sense up to some uncertainty) from the point of view of $f_i$, for they can be lifted to the pullback $[T/\mathbf{A}E_{b}^{\varepsilon}]_i(\tau)$. 

Thus, the role of the functor $[T/\mathbf{A}E_{b}^{\varepsilon}]_i$ is to take care of selecting all those multiple sequence alignments generated by $T$ that make sense with the component $f_i:(\Omega,\preceq) \to (\Omega,\preceq_i)$, where the idea of ``making sense'' is strongly related to the pedigrad structure of the functor $$E_{f_i(b)}^{\varepsilon}:\mathbf{Seg}(\Omega_i) \to \mathbf{Set}.$$ In Example \ref{exa:lifting_gluings_of_tables}, we show that something more subtle happens when the cone $\mathbf{Seg}(f_i) \rho_{\iota}[\tau]$ is not in the chromology of $E_{f_i(b)}^{\varepsilon}$.
\end{remark}

\begin{example}[Resolving uncertainties]\label{exa:lifting_gluings_of_tables}
Let $(\iota,T,\sigma)$ be the sequence alignment constructed in Example \ref{exa:sequence_alignments}. The present example discusses the meaning of the elements contained in the image $\mathsf{Ran}_{\iota}T(!_{[8]},[\mathtt{1\separ0\separ1\separ1}])$, which we computed in Example \ref{exa:Global-local_seq_alignments}, from the point of view of slices. We shall consider the same notations as those used in Example \ref{exa:Global-local_seq_alignments}, but, for convenience, we will denote the segment $(!_{[8]},[\mathtt{1\separ0\separ1\separ1}])$ of $\mathbf{Seg}(\Omega^{\times 4})$ as $\tau$. In this case, the evaluation of diagram (\ref{eq:pullback_sequence_alignment}) at $\tau$ is of the following form for every $i \in \{\mathtt{a},\mathtt{b},\mathtt{c},\mathtt{d}\}$.
\begin{equation}\label{eq:pullback-slice_exception_resolving _uncertainties}
\xymatrix@C+15pt{
[T/\mathbf{A}E_{b}^{\varepsilon}]_i(\tau)\ar[d]\ar[rr]^-{\sigma^{*}_{\tau}}&&\pi_i^{*}E_{b}^{\varepsilon}(\tau)\ar[d]^{\eta_{\tau}}\\
*+!R(.0){L_8([\mathtt{1\separ0\separ1\separ1}]) \times T(!_{[9]},[\mathtt{0\separ0\separ1\separ1}])^{\times 9}}  \ar[r]_-{\mathsf{Ran}_{\iota}(\sigma)_{\tau}}&\mathsf{Ran}_{\iota}(\mathbf{A}E_{b}^{\varepsilon} \circ \iota)(\tau)\ar[r]_-{\mathsf{Ran}_{\iota}(\kappa_i)_{\tau}}&\mathsf{Ran}_{\iota}(\pi_i^{*}E_{b}^{\varepsilon} \circ \iota)(\tau)
}
\end{equation}
We are going to illustrate a case in which the 9-fold Cartesian product 
\begin{equation}\label{eq:red_flag_9-fld_Cartesian_product}
T(!_{[9]},[\mathtt{0\separ0\separ1\separ1}])^{\times 9}
\end{equation}
appearing in the left-bottom corner of (\ref{eq:pullback-slice_exception_resolving _uncertainties}) (\emph{i.e} the object $\mathsf{Ran}_{\iota}T(\tau)$) prevents pullback (\ref{eq:pullback-slice_exception_resolving _uncertainties}) from lifting any element in $\mathsf{Ran}_{\iota}T(\tau)$ to the $i$-slice $[T/\mathbf{A}E_{b}^{\varepsilon}]_i(\tau)$. The reason for this obstruction is that object (\ref{eq:red_flag_9-fld_Cartesian_product}) contains tuples that fail to match the type of tuples associated with the images of the function $\eta_{\tau}$ (given on the right of diagram (\ref{eq:pullback-slice_exception_resolving _uncertainties})).

Before discussing the aforementioned obstruction, let us describe the form of the mappings associated with the arrows of diagram (\ref{eq:pullback-slice_exception_resolving _uncertainties}) in more detail (see diagram (\ref{eq:description_slice_elements_uncertaintites}), below, for future reference). First, recall that every copy $T(!_{[9]},[\mathtt{0\separ0\separ1\separ1}])$ of the $9$-fold Cartesian product given in (\ref{eq:red_flag_9-fld_Cartesian_product}) is associated with one of the $9$ morphisms of the form $(!_{[8]},[\mathtt{1\separ0\separ1\separ1}]) \to (!_{[9]},[\mathtt{0\separ0\separ1\separ1}])$ in $\mathbf{Seg}(\{0,1\}^{\times 4})$ (see examples below).
\[
\xymatrix@C-30pt{(\mathop{\bullet}\limits^{1}&\mathop{\bullet}\limits^{2}&\mathop{\bullet}\limits^{3}&\mathop{\bullet}\limits^{4}&\mathop{\bullet}\limits^{5}&\mathop{\bullet}\limits^{6}&\mathop{\bullet}\limits^{7})}
\longrightarrow
\xymatrix@C-30pt{(\mathop{\bullet}\limits^{\ast}&\mathop{\bullet}\limits^{1}&\mathop{\bullet}\limits^{2}&\mathop{\bullet}\limits^{3}&\mathop{\bullet}\limits^{4}&\mathop{\bullet}\limits^{5}&\mathop{\bullet}\limits^{6}&\mathop{\bullet}\limits^{7})},
\quad\quad
\xymatrix@C-30pt{(\mathop{\bullet}\limits^{1}&\mathop{\bullet}\limits^{2}&\mathop{\bullet}\limits^{3}&\mathop{\bullet}\limits^{4}&\mathop{\bullet}\limits^{5}&\mathop{\bullet}\limits^{6}&\mathop{\bullet}\limits^{7})}
\longrightarrow
\xymatrix@C-30pt{(\mathop{\bullet}\limits^{1}&\mathop{\bullet}\limits^{\ast}&\mathop{\bullet}\limits^{2}&\mathop{\bullet}\limits^{3}&\mathop{\bullet}\limits^{4}&\mathop{\bullet}\limits^{5}&\mathop{\bullet}\limits^{6}&\mathop{\bullet}\limits^{7})},
\quad\quad
\textrm{etc.}
\]
Let us denote these morphisms as $g_k:(!_{[8]},[\mathtt{1\separ0\separ1\separ1}]) \to (!_{[9]},[\mathtt{0\separ0\separ1\separ1}])$ for every $k \in [9]$. The leftmost horizontal arrow given at the bottom of (\ref{eq:pullback-slice_exception_resolving _uncertainties}) then sends a pair $(x,y)$, where $x$ is an element of $L_8([\mathtt{1\separ0\separ1\separ1}])$ and $y = (y_1,\dots,y_9)$ is an element of the 9-fold Cartesian product of $T(!_{[9]},[\mathtt{0\separ0\separ1\separ1}])$, to the same tuple $(x,y)$ in $\mathsf{Ran}_{\iota}(\mathbf{A}E_{b}^{\varepsilon} \circ \iota)(\tau)$ provided that one sees $x$ and $y_1,\dots,y_9$ as elements taken from the images of $\mathbf{A}E_{b}^{\varepsilon}$ (see Definition \ref{def:sequence_alignment}). Then, the tuple $(x,y)$ in $\mathsf{Ran}_{\iota}(\mathbf{A}E_{b}^{\varepsilon} \circ \iota)(\tau)$ is sent to a tuple of the form 
$
(x_i,(y_{1,i},\dots,y_{9,i}))
$
made of the $i$-th projections of $x$ and $y_1,\dots,y_9$ with respect to the 4-fold product structure underlying the definition of $\mathbf{A}E_{b}^{\varepsilon}$. It follows that the tuple $(x,y)$ living in the left-bottom corner of (\ref{eq:pullback-slice_exception_resolving _uncertainties}) is lifted to the left-top corner of (\ref{eq:pullback-slice_exception_resolving _uncertainties}) if there exists an element $z$ in $\pi_i^{*}E_{b}^{\varepsilon}(\tau)$ whose image through $\eta_{\tau}$ is equal to the image $\mathsf{Ran}_{\iota}(\kappa_i)_{\tau}(x,y)$. 
\begin{equation}\label{eq:description_slice_elements_uncertaintites}
\xymatrix{
((x,y),z)\ar@{|-->}[rr]\ar@{|-->}[d]&&z\ar@{|->}[d]^{\eta_{\tau}}\\
(x,y)\ar@{|->}[r]^{\mathsf{Ran}_{\iota}(\sigma)_{\tau}}&(x,y)\ar@{|->}[r]^-{\mathsf{Ran}_{\iota}(\kappa_i)_{\tau}}&*+!L(.7){(x_i,(y_{1,i},\dots,y_{9,i}))}
}
\end{equation}
Let us now focus on a particular example. Let the index $i$ be equal to the element $\mathtt{c}$ representing \texttt{Craig}'s viewpoint. If we take $z$ to be the element $\mathtt{ACCGTC\varepsilon A}$ in $\pi_{\mathtt{c}}^{*}E_{b}^{\varepsilon}(!_{[8]},[\mathtt{1\separ0\separ1\separ1}])$, then its image through $\eta_{\tau}$ is of the following form.
\begin{equation}\label{eq:image_through_eta_tau_resolving}
\left(
\begin{array}{l}
\mathtt{ACCGTC\varepsilon A}
\end{array}\!,\,\left(
\begin{array}{l}
\mathtt{\varepsilon ACCGTC\varepsilon A}
\end{array}\!,\,
\begin{array}{l}
\mathtt{A\varepsilon CCGTC\varepsilon A}
\end{array}\!,\, \dots,
\begin{array}{l}
\mathtt{ACCGTC\varepsilon A\varepsilon}
\end{array}\right)
\right)
\end{equation}
On the other hand, the following element 
\begin{equation}\label{exa:element_x_y}
\left(
\begin{array}{l}
\begin{array}{l}
\mathtt{ACCGACTG}\\
\mathtt{ACCGTC\varepsilon A}\\
\mathtt{A\varepsilon CTACTG }
\end{array}\!,\,
\left(
\begin{array}{l}
\mathtt{ACCGT\varepsilon C\varepsilon A}\\
\mathtt{A\varepsilon C\varepsilon TACTG }
\end{array}\!,\,
\begin{array}{l}
\mathtt{ACCGT\varepsilon C\varepsilon A}\\
\mathtt{A\varepsilon C\varepsilon TACTG }
\end{array}\!,\, \dots,
\begin{array}{l}
\mathtt{ACCGT\varepsilon C\varepsilon A}\\
\mathtt{A\varepsilon C\varepsilon TACTG }
\end{array}
\right)
\end{array}
\right),
\end{equation}
living in the image $\mathsf{Ran}_{\iota}T(\tau)= L_8([\mathtt{1\separ0\separ1\separ1}]) \times T(!_{[9]},[\mathtt{0\separ0\separ1\separ1}])^{\times 9}$, is sent through $\mathsf{Ran}_{\iota}(\kappa_i \circ \sigma)_{\tau}$ to the following tuple in $\mathsf{Ran}_{\iota}(\pi_i^{*}E_{b}^{\varepsilon} \circ \iota)(\tau)$.
\begin{equation}\label{eq:image_through_kappa_sigma_tau_resolving}
\left(
\begin{array}{l}
\mathtt{ACCGTC\varepsilon A}
\end{array}\!,\,\left(
\begin{array}{l}
\mathtt{ACCGT\varepsilon C\varepsilon A}
\end{array}\!,\,
\begin{array}{l}
\mathtt{ACCGT\varepsilon C\varepsilon A}
\end{array}\!,\, \dots,
\begin{array}{l}
\mathtt{ACCGT\varepsilon C\varepsilon A}
\end{array}\right)
\right)
\end{equation}
As can be seen, the element $z = \mathtt{ACCGTC\varepsilon A}$ cannot be a valid integration of the element given in (\ref{exa:element_x_y}) because the image of (\ref{exa:element_x_y}) through $\mathsf{Ran}_{\iota}(\kappa_i \circ \sigma)_{\tau}$ -- shown in (\ref{eq:image_through_kappa_sigma_tau_resolving}) -- is not equal to the image of $z$ through $\eta_{\tau}$ -- shown in (\ref{eq:image_through_eta_tau_resolving}).
The reader can easily see that the main reason for this is that the rightmost  tuple of (\ref{exa:element_x_y}) (which stands for $y = (y_1,\dots,y_9)$) consists of equal components (\emph{i.e.} $y_1 = \dots = y_9$) while the existence of a lift would imply that the resulting collection of components $y_{1,i},\dots,y_{9,i}$ can equal the distinct components of the image of $z$ through $\eta_{\tau}$. Obviously, to do so, we would need to make the components of the rightmost tuple of (\ref{exa:element_x_y}) vary through various pairwise sequence alignments.  
Unfortunately, an analysis of the elements of $T(!_{[9]},[\mathtt{0\separ0\separ1\separ1}])$ quickly reveals that the set $T(!_{[9]},[\mathtt{0\separ0\separ1\separ1}])$  does not contain enough elements to make the elements of $L_8([\mathtt{1\separ0\separ1\separ1}]) \times T(!_{[9]},[\mathtt{0\separ0\separ1\separ1}])^{\times 9}$ match the images of the morphism $\eta_{\tau}$. This shows that there is not enough evidence that the alignment 
\begin{equation}\label{eq:good_bad_alignment}
\begin{array}{l}
\mathtt{ACCGACTG}\\
\mathtt{ACCGTC\varepsilon A}\\
\mathtt{A\varepsilon CTACTG }
\end{array}
\end{equation}
is a good alignment from the point of view of \texttt{Craig} (\emph{i.e.} $i = \texttt{c}$). Here the main obstruction is that the exponent of $T(!_{[9]},[\mathtt{0\separ0\separ1\separ1}])^{\times 9}$ is too big for the cardinality of $T(!_{[9]},[\mathtt{0\separ0\separ1\separ1}])$ (which is due to the uncertainty related to aligning distant DNA sequences). On the other hand, adding more colors (see Example \ref{rem:resolving_inconsistencies}) and using more complex topologies (see Example \ref{exa:uncertainty_vs_chromologies}) can reduce the exponent of $T(!_{[9]},[\mathtt{0\separ0\separ1\separ1}])^{\times 9}$, which would have the consequence of making alignment (\ref{eq:good_bad_alignment}) more likely to be liftable from the point of view of the added knowledge.
\end{example} 

\begin{remark}[Multiple sequence alignments and mechanisms]\label{ref:slices_multiple_seq_align_and_mechanisms}
Tuple (\ref{exa:element_x_y}), displayed in Example \ref{exa:lifting_gluings_of_tables}, shows us that slices produce two types of data integration. The first type of integration looks at the construction of multiple sequence alignments, such as the one shown on the left of tuple (\ref{exa:element_x_y}), while the second type of integration concerns the other part of tuple (\ref{exa:element_x_y}) that consists of the collection of pairwise sequence alignments. Example \ref{exa:lifting_gluings_of_tables} does not tell us much about this type of integration, except for creating some uncertainty. In section \ref{ssec:slices_and_mechanisms}, we will see that this uncertainty is actually associated with the presence and detection of mutation mechanisms. In the case of Example \ref{exa:lifting_gluings_of_tables}, no mechanism could be recognized from the data, which translates into an absence of lift.
\end{remark}

\begin{remark}[Finding a multiple sequence alignment]\label{rem:finding_right_seq_alignment}
Example \ref{exa:lifting_gluings_of_tables} and Remark \ref{ref:slices_multiple_seq_align_and_mechanisms} implicitly motivate an algebraic method to select  multiple sequence alignments along with mechanisms. Specifically, the method would look at the pullbacks of the slices of a certain sequence alignment functor, say $(\iota,T,\sigma)$ and, for a given segment $\tau$, would try to find the maximal subset $A' \subseteq A$ for which the wide pullback of the $i$-slices, for every $i \in A'$, is maximal at the segment $\tau$ (see below).
\[
\xymatrix{
&[T/\mathbf{A}E_{b}^{\varepsilon}]_{i_1}(\tau)\ar[rd]&\\
\cap_{i \in A'} [T/\mathbf{A}E_{b}^{\varepsilon}]_{i}(\tau)\ar[ru]\ar[r]\ar[rd]&[T/\mathbf{A}E_{b}^{\varepsilon}]_{i_2}(\tau)\ar@{}[d]|{\vdots}\ar[r]&\mathsf{Ran}_{\iota} T(\tau)\\
&[T/\mathbf{A}E_{b}^{\varepsilon}]_{i_n}(\tau)\ar[ru]&
}
\]
Ideally, the segment $\tau$ should only be made of a maximal color, but segments of intermediate colors could also be used for heuristics, if necessary. 
\end{remark}

\subsection{Slices and mechanisms}\label{ssec:slices_and_mechanisms}
In section \ref{sec:slices_seq_alignment}, we showed that the cones associated with the slices of a sequence alignment functor $(\iota,T,\sigma)$ could be used to query multiple sequence alignments that make sense from the point of view of a particular individual (see Example \ref{exa:lifting_gluings_of_tables}). Then, in Remark \ref{rem:finding_right_seq_alignment}, we suggested that these queries could be used to find a set of multiple sequence alignments in the right Kan extension of $T$ along $\iota$ that maximize the number of individuals agreeing with the alignments.

The goal of this section is to show that the querying process inherent to slices (Definition \ref{def:Slices}) can be used to query mechanisms, too. Here, we will show that mechanisms such as duplication events, which are mutations responsible for triggering certain cancers \cite{ReamsRoth}, and inversion events, which are rearrangements of certain sections of a segment in reverse order, can be detected through particular types of cones (see Remark \ref{rem:resolving_uncertainties}).

Note that previous works have already investigated the recognition of duplication and inversion mechanisms in multiple sequence alignments \cite{SammethEDSI,SammethIEEE,Schoniger,Vellozo}. While these papers develop methods for predetermined patterns, slices are more flexible in that they inform us of existing patterns without requiring us to know what these patterns should look like. More specifically, these patterns are encoded in the shape of the cones described in Remark \ref{rem:resolving_uncertainties} and are thus systematically given through the computation of right Kan extension. The recognition of mechanisms could then be done by comparing the shape of the cones for which the slices are non empty with the shape of cones that are known to characterize specific mechanisms.

The present section does not introduce any new concept and only aims to show examples. The reader will be assumed to remember the reasoning of Example \ref{exa:lifting_gluings_of_tables}, which we intend to mimic in this section -- the goal is again to match certain tuples through the arrows of diagram (\ref{eq:pullback_sequence_alignment}). 

\begin{example}[Duplication mechanisms]\label{exa:slices_and_mechanisms_duplications}
For our first example, we shall let $\mathbf{A}$ denote the alignment specification defined by the following collection of projections (already used in Example \ref{exa:Alignment_specification})
\[
\{\pi_i:\{0,1\}^{\times 4} \to \{0,1\}\}_{i \in \{\mathtt{a},\mathtt{b},\mathtt{c},\mathtt{d}\}}
\]
and let $b$ be the element $(1,1,1,1)$ of $\{0,1\}^{\times 4}$.
Because the sequence alignment functor of Example \ref{exa:sequence_alignments} was built from data that are not suited for a good illustration of duplication mechanisms, we will let $(\iota,T,\sigma)$ be an undetermined sequence alignment functor over $\mathbf{A}E_{b}^{\varepsilon}$, where $E$ is the set $\{\mathtt{A},\mathtt{C},\mathtt{G},\mathtt{T},\varepsilon\}$. Our goal is to illustrate the types of situations in which duplication mechanisms can be queried within the right Kan extension.

For every object $\tau$ in $\mathbf{Seg}(\{0,1\}^{\times 4})$, we will denote by $\rho_{\iota}[\tau]$ the cone in $\mathbf{Seg}(\{0,1\}^{\times 4})$ encoded by the arrows and objects of the category $({\tau \downarrow \iota})$ (see Remark \ref{rem:Extending_category_as_a_cone} for a detailed description).
As in Example \ref{exa:lifting_gluings_of_tables}, we want to study the functor $T$ from the point of view of a certain individual; we will consider \texttt{Craig}, who is, as usual, associated with the index $\mathtt{c}$. Now, suppose that the image of the cone $\rho_{\iota}[\tau]$ via the functor $\mathbf{Seg}(\pi_{\mathtt{c}}):\mathbf{Seg}(\{0,1\}^{\times 4}) \to \mathbf{Seg}(\{0,1\})$ is of the form shown in (\ref{eq:cone_for_duplication}).
\begin{equation}\label{eq:cone_for_duplication}
\xymatrix@C-30pt@R-25pt{
&&&&&&&&&(\mathop{\bullet}\limits^{1}&\mathop{\bullet}\limits^{2})&(\mathop{\bullet}\limits^{3})&(\mathop{\bullet}\limits^{*})&(\bullet&\bullet&\bullet&\bullet)\\
(\mathop{\bullet}\limits^{1}&\mathop{\bullet}\limits^{2})&(\mathop{\bullet}\limits^{3})&(\bullet&\bullet&\bullet&\bullet)\ar[rru]\ar[rrd]&\quad\quad\quad\quad&&&&&&&&\\
&&&&&&&&&(\mathop{\bullet}\limits^{1}&\mathop{\bullet}\limits^{2})&(\mathop{\bullet}\limits^{*})&(\mathop{\bullet}\limits^{3})&(\bullet&\bullet&\bullet&\bullet)
}
\end{equation}
For such a cone, the pullback of Definition \ref{def:Slices} lifts any element of $\mathsf{Ran}_{\iota}T(\tau)$ to \texttt{Craig}'s slice if this element can be sent, through the function $\mathsf{Ran}_{\iota}(\kappa_{\mathtt{c}}\circ \sigma)_{\tau}:\mathsf{Ran}_{\iota}T(\tau)  \to \mathsf{Ran}_{\iota}(\pi_{\mathtt{c}}^{*}E_{b}^{\varepsilon} \circ \iota)(\tau)$, to an image of $\eta_{\tau}$ in $\mathsf{Ran}_{\iota}(\pi_{\mathtt{c}}^{*}E_{b}^{\varepsilon} \circ \iota)(\tau)$ -- for instance, a pair of the following form.
\begin{equation}\label{eq:matching_pair_cone_for_duplication}
\eta_{\tau}(\mathtt{x_1x_2 Z x_3x_4x_5x_6}) = (\mathtt{x_1x_2\varepsilon Z x_3x_4x_5x_6},\,\mathtt{x_1x_2 Z \varepsilon x_3x_4x_5x_6})
\end{equation}
If our sequence alignment functor $(\iota,T,\sigma)$ is constructed from the outputs of a dynamic programming algorithm, as in Example \ref{exa:sequence_alignments}, then the elements of $\mathsf{Ran}_{\iota}T(\tau)$ that match pair (\ref{eq:matching_pair_cone_for_duplication}) through the function $\mathsf{Ran}_{\iota}(\kappa_{\mathtt{c}}\circ \sigma)_{\tau}$ will most likely be tuples of sequence alignments in which each of the components of tuple (\ref{eq:matching_pair_cone_for_duplication}) appears. An example of such an element in $\mathsf{Ran}_{\iota}T(\tau)$ is given by the following pair of sequence alignments, which tries to align the sequence $\mathtt{x_1x_2ZZ x_3x_4x_5x_6}$ with the sequence $\mathtt{x_1x_2Z x_3x_4x_5x_6}$.
\[
\left(
\begin{array}{l}
\begin{array}{l}
\mathtt{x_1x_2\varepsilon Zx_3x_4x_5x_6}\\
\mathtt{x_1x_2ZZx_3x_4x_5x_6}
\end{array}\!,\,
\begin{array}{l}
\mathtt{x_1x_2 Z \varepsilon x_3x_4x_5x_6}\\
\mathtt{x_1x_2ZZx_3x_4x_5x_6}
\end{array}
\end{array}
\right)
\]
In this case, the tuple given below lives in the image of \texttt{Craig}'s slice at the segment $\tau$ (\emph{i.e.} $[T/\mathbf{A}E_{b}^{\varepsilon}]_{\mathtt{c}}(\tau)$) and plays the role of the desired lift.
\[
\left(
\begin{array}{l}
\begin{array}{l}
\mathtt{x_1x_2\varepsilon Zx_3x_4x_5x_6}\\
\mathtt{x_1x_2ZZx_3x_4x_5x_6}
\end{array}\!,\,
\begin{array}{l}
\mathtt{x_1x_2 Z \varepsilon x_3x_4x_5x_6}\\
\mathtt{x_1x_2ZZx_3x_4x_5x_6}
\end{array}\!,\,
\begin{array}{l}
\mathtt{x_1x_2Z x_3x_4x_5x_6}
\end{array}
\end{array}
\right)
\]
Intuitively, the previous element tells us that \texttt{Craig} is separated from a certain other individual (whose genetic data is $\mathtt{x_1x_2ZZx_3x_4x_5x_6}$) by a duplication mechanism. On the other hand, if the sequence of the other individual were of a different form, say $\mathtt{x_1x_2ZYx_3x_4x_5x_6}$, the scoring system of the dynamic programming algorithm would only be expected to provide the sequence alignment given below, on the left (at least for an adequate scoring system).
\[
\begin{array}{l}
\mathtt{x_1x_2Z\varepsilon x_3x_4x_5x_6}\\
\mathtt{x_1x_2ZYx_3x_4x_5x_6}
\end{array}
\quad\quad\quad\quad\quad
\xcancel{
\begin{array}{l}
\mathtt{x_1x_2\varepsilon Zx_3x_4x_5x_6}\\
\mathtt{x_1x_2ZYx_3x_4x_5x_6}
\end{array}}
\]
In the end, this would prevent the existence of an element in $\mathsf{Ran}_{\iota}T(\tau)$ matching tuple (\ref{eq:matching_pair_cone_for_duplication}) and would hence prevent the existence of a lift to \texttt{Craig}'s slice. 

To conclude, designing the domain $B$ of the functor $T$ (either using colors or brackets) such that cone (\ref{eq:cone_for_duplication}) is the image of a cone of the form $\rho_{\iota}[\tau]$ via $\mathbf{Seg}(\pi_{\mathtt{c}})$ will force the pullback of Definition \ref{def:Slices} to select alignments that may be explained by duplication mechanisms -- at least from \texttt{Craig}'s viewpoint.
\end{example}

\begin{remark}[Types of cones detecting mechanisms]
Note that, while multiple sequence alignments would be lifted along cones that belong to chromologies (as in Remark \ref{rem:resolving_uncertainties}),  mechanisms (such as duplications) would usually be lifted along cones that are not part of chromologies, since they would usually be made of morphisms of the type described in Example \ref{exa:Flexibility_morphism} to create mutation events (see diagram (\ref{eq:cone_for_duplication}) and the definition of section \ref{ssec:chromologies}). 
\end{remark}

\begin{example}[Inversion mechanisms]\label{exa:slices_and_mechanisms_inversions}
Let us now give an example of a cone that lifts inversion mechanisms. This time, we will need to take $\Omega$ to be the pre-ordered set $\{0 \leq 1 \leq 2\}$. We will implicitly use the color 2 of $\Omega$ to restrict the number of arrows that $\rho_{\iota}[\tau]$ may possess (see diagram (\ref{eq:cone_inversions})). In addition, we will to take $\mathbf{A}$ to be the alignment specification consisting of the following projections.
\[
\{\pi_i:\{0,1,2\}^{\times 4} \to \{0,1,2\}\}_{i \in \{\mathtt{a},\mathtt{b},\mathtt{c},\mathtt{d}\}}
\]
The element $b$ will be taken to be equal to the element $(1,1,1,1)$ of $\{0,1,2\}^{\times 4}$ and $(\iota,T,\sigma)$ will denote an undetermined sequence alignment over $\mathbf{A}E_{b}^{\varepsilon}$, where $E$ is the set $\{\mathtt{A},\mathtt{C},\mathtt{G},\mathtt{T},\varepsilon\}$. For the present example, we will suppose that the image of the cone $\rho_{\iota}[\tau]$ via the functor $\mathbf{Seg}(\pi_{\mathtt{c}}):\mathbf{Seg}(\{0,1,2\}^{\times 4}) \to \mathbf{Seg}(\{0,1,2\})$ is of the form given in (\ref{eq:cone_inversions}).
\begin{equation}\label{eq:cone_inversions}
\xymatrix@C-30pt@R-15pt{
&&&&&&&&&&(\mathop{\mathtt{1}}\limits^{1}&\mathop{\mathtt{1}}\limits^{2})&(\mathop{\mathtt{2}}\limits^{*})&(\mathop{\mathtt{2}}\limits^{*})&(\mathop{\mathtt{1}}\limits^{3})&(\mathop{\mathtt{1}}\limits^{4})&(\mathop{\mathtt{1}}\limits^{5})&(\mathtt{1}&\mathtt{1}&\mathtt{1}&\mathtt{1})\\
(\mathop{\mathtt{1}}\limits^{1}&\mathop{\mathtt{1}}\limits^{2})&(\mathop{\mathtt{1}}\limits^{3})&(\mathop{\mathtt{1}}\limits^{4})&(\mathop{\mathtt{1}}\limits^{5})&(\mathtt{1}&\mathtt{1}&\mathtt{1}&\mathtt{1})\ar[rru]\ar[rrd]\ar[rr]&\quad\quad\quad\quad&(\mathop{\mathtt{1}}\limits^{1}&\mathop{\mathtt{1}}\limits^{2})&(\mathop{\mathtt{1}}\limits^{3})&(\mathop{\mathtt{2}}\limits^{*})&(\mathop{\mathtt{1}}\limits^{4})&(\mathop{\mathtt{2}}\limits^{*})&(\mathop{\mathtt{1}}\limits^{5})&(\mathtt{1}&\mathtt{1}&\mathtt{1}&\mathtt{1})\\
&&&&&&&&&&(\mathop{\mathtt{1}}\limits^{1}&\mathop{\mathtt{1}}\limits^{2})&(\mathop{\mathtt{1}}\limits^{3})&(\mathop{\mathtt{1}}\limits^{4})&(\mathop{\mathtt{1}}\limits^{5})&(\mathop{\mathtt{2}}\limits^{*})&(\mathop{\mathtt{2}}\limits^{*})&(\mathtt{1}&\mathtt{1}&\mathtt{1}&\mathtt{1})
}
\end{equation}

We want to show that, for such a cone, \texttt{Craig}'s slice, evaluated at the segment $\tau$, can detect inversion mechanisms.
First, by the shape of cone (\ref{eq:cone_inversions}), the pullback of Definition \ref{def:Slices} lifts any element of $\mathsf{Ran}_{\iota}T(\tau)$ to \texttt{Craig}'s slice if this element can be sent, through the function $\mathsf{Ran}_{\iota}(\kappa_{\mathtt{c}}\circ \sigma)_{\tau}:\mathsf{Ran}_{\iota}T(\tau)  \to \mathsf{Ran}_{\iota}(\pi_{\mathtt{c}}^{*}E_{b}^{\varepsilon} \circ \iota)(\tau)$, to an image of $\eta_{\tau}$ in $\mathsf{Ran}_{\iota}(\pi_{\mathtt{c}}^{*}E_{b}^{\varepsilon} \circ \iota)(\tau)$, and hence a triple of the following form.
\begin{equation}\label{eq:matching_pair_cone_for_inversion}
(
\mathtt{x_1x_2\varepsilon\varepsilon ABCx_3x_4x_5x_6},\,
\mathtt{x_1x_2 A\varepsilon B \varepsilon C x_3x_4x_5x_6},\,
\mathtt{x_1x_2 ABC \varepsilon\varepsilon x_3x_4x_5x_6}
)
\end{equation}
If our sequence alignment functor $(\iota,T,\sigma)$ is constructed from the outputs of a dynamic programming algorithm, as in Example \ref{exa:sequence_alignments}, then the elements of $\mathsf{Ran}_{\iota}T(\tau)$ that match pair (\ref{eq:matching_pair_cone_for_inversion}) through the function $\mathsf{Ran}_{\iota}(\kappa_{\mathtt{c}}\circ \sigma)_{\tau}$ will most likely be tuples of sequence alignments in which each of the components of (\ref{eq:matching_pair_cone_for_inversion}) appears. An example of such an element in $\mathsf{Ran}_{\iota}T(\tau)$ is given by the following triple of sequence alignments, which tries to align the sequence $\mathtt{x_1x_2ABCx_3x_4x_5x_6}$ with the sequence $\mathtt{x_1x_2CBAx_3x_4x_5x_6}$.
\begin{equation}\label{eq:triple_inversion_query_RKE}
\left(
\begin{array}{l}
\begin{array}{l}
\mathtt{x_1x_2\varepsilon\varepsilon ABCx_3x_4x_5x_6}\\
\mathtt{x_1x_2CBA\varepsilon\varepsilon x_3x_4x_5x_6}
\end{array}\!,\,
\begin{array}{l}
\mathtt{x_1x_2 A\varepsilon B \varepsilon C x_3x_4x_5x_6}\\
\mathtt{x_1x_2\varepsilon CBA \varepsilon x_3x_4x_5x_6}
\end{array}\!,\,
\begin{array}{l}
\mathtt{x_1x_2 ABC \varepsilon\varepsilon x_3x_4x_5x_6}\\
\mathtt{x_1x_2\varepsilon\varepsilon CBA  x_3x_4x_5x_6}
\end{array}
\end{array}
\right)
\end{equation}
In this case, the tuple displayed below lives in the image of \texttt{Craig}'s slice at the segment $\tau$ and plays the role of the desired lift.
\[
\left(
\begin{array}{l}
\begin{array}{l}
\mathtt{x_1x_2\varepsilon\varepsilon ABCx_3x_4x_5x_6}\\
\mathtt{x_1x_2CBA\varepsilon\varepsilon x_3x_4x_5x_6}
\end{array}\!,\,
\begin{array}{l}
\mathtt{x_1x_2 A\varepsilon B \varepsilon C x_3x_4x_5x_6}\\
\mathtt{x_1x_2\varepsilon CBA \varepsilon x_3x_4x_5x_6}
\end{array}\!,\,
\begin{array}{l}
\mathtt{x_1x_2 ABC \varepsilon\varepsilon x_3x_4x_5x_6}\\
\mathtt{x_1x_2\varepsilon\varepsilon CBA  x_3x_4x_5x_6}
\end{array}\!,\,
\mathtt{x_1x_2ABCx_3x_4x_5x_6}
\end{array}
\right)
\]
As can be seen, this type of tuple tries to align the sequence $\mathtt{x_1x_2 ABCx_3x_4x_5x_6}$ with the sequence $\mathtt{x_1x_2 CBAx_3x_4x_5x_6}$, which are clearly related by an inversion of the patch $\mathtt{ABC}$. As in Example \ref{exa:slices_and_mechanisms_duplications}, trying a different sequence, say $\mathtt{x_1x_2 EFGx_3x_4x_5x_6}$, is unlikely to create a triple as in (\ref{eq:triple_inversion_query_RKE}) and hence a lift to \texttt{Craig}'s slice. In other words, the earlier tuple specifically tells us that \texttt{Craig} is separated from a certain other individual (whose genetic data is $\mathtt{x_1x_2CBAx_3x_4x_5x_6}$) by an inversion mechanism.

To conclude, designing the domain $B$ of the functor $T$ such that cone (\ref{eq:cone_inversions}) is the image of a cone of the form $\rho_{\iota}[\tau]$ via $\mathbf{Seg}(\pi_{\mathtt{c}})$ will force the pullback of Definition \ref{def:Slices} to select alignments that may be explained by inversion mechanisms -- at least from \texttt{Craig}'s viewpoint.
\end{example}

\section{Conclusion}

We formalized the concept of sequence alignment in terms of a subcategory $B$ of segments (Definitions \ref{ssec:Segments} \& \ref{ssec:Morphisms_of_chromosomal_patches}) and a functor $T:B \to \mathbf{Set}$ whose images contain usual sequence alignments (Definition \ref{def:sequence_alignment}). We showed that we could design comparison rules between the sequence alignments contained in $T$ through the structure of the category $B$ (Examples \ref{exa:sequence_alignments} \& \ref{rem:resolving_inconsistencies}). We then showed that we could integrate multiple sequence alignments from the data contained in $T$ by using the right Kan extension of $T$ to the whole category of segments (section \ref{ssec:From_RKE_to_MSA}). In addition, we showed how the right Kan extension of $T$ could be used to study the consistency of the integrated data through the use of limits and the existence of certain functions (see Example \ref{exa:reasoning_with_sequence_alignment}). While inconsistent data could be associated with non-surjective functions, consistent data could be associated with either surjections or isomorphisms. Regarding these last two types of arrows, we showed that isomorphisms informed us of a perfect consistency (Example \ref{exa:right_kan_extension_data_consistency}) while surjections indicated some uncertainty (Examples 
\ref{exa:Reasoning_with_right_Kan_extensions}, \& \ref{exa:Global-local_seq_alignments}), which we later related to the presence of mutation mechanisms (Remark \ref{ref:slices_multiple_seq_align_and_mechanisms}).
We then introduced the concept of slice (Definition \ref{def:Slices}) as a way to resolve this uncertainty and, at the same time, discover mutation mechanisms (Remark \ref{rem:finding_right_seq_alignment} and section \ref{ssec:slices_and_mechanisms}).


\bibliographystyle{plain}

\begin{thebibliography}{10}
\bibitem{Altschul} S. F. Altschul, W. Gish, W. Miller, E. W. Myers, D. J. Lipman, (5 Oct 1990) \emph{Basic local alignment search tool}, Journal of Molecular Biology, Volume 215, Issue 3, pp 403--410.

\bibitem{BrownPorter} R. Brown, T. Porter, (2006), \emph{Category Theory:  an abstract setting for analogy and comparison}, In: What is Category Theory? Advanced Studies in Mathematics and Logic, Polimetrica Publisher, Italy, pp 257--274.

\bibitem{slice_bio} A. Carbone, M. Gromov, (2001), \emph{Mathematical slices of molecular biology}, Journal Article, La Gazette des Math\'{e}maticiens, Soci\'{e}t\'{e} Math\'{e}matique de France, special edition, pp 11--80.

\bibitem{Chowdhury} B. Chowdhury, G. Garai, (2017), \emph{A review on multiple sequence alignment from the perspective of genetic algorithm}, Genomics, Volume 109, Issues 5–6, pp 419--431.

\bibitem{Daugelaite} J. Daugelaite, A. O' Driscoll, and R. D. Sleator, (2013), \emph{An Overview of Multiple Sequence Alignments and Cloud Computing in Bioinformatics}, ISRN Biomathematics, Volume 2013, Article ID 615630, 14 pp.

\bibitem{A_Ehresmann} A. Ehresmann, (2013), \emph{Le r\^{o}le des limites projectives dans le d\'{e}veloppement des m\'{e}moires proc\'{e}durale et s\'{e}mantique}, S\'{e}minaire Mamuphi, ENS; see \url{http://ehres.pagesperso-orange.fr/}.

\bibitem{Ehresmann} C. Ehresmann, (1968), \emph{Esquisses et types de structures alg\'{e}briques}. Bull. Instit. Polit. Ia\c{s}i \textbf{XIV}.


\bibitem{Feng} D. F. Feng, R. F. Doolittle, (1987), \emph{Progressive sequence alignment as a prerequisitetto correct phylogenetic trees}. J Mol Evol., Volume 25, Issue 4, pp 351--360.




\bibitem{Lawvere1963} W. Lawvere, (1963), \emph{Functorial Semantics of Algebraic Theories}. Ph.D. Thesis, Columbia University, New York, NY, USA. 

\bibitem{Lazebnik} Y. Lazebnik, (Sep 2002), \emph{Can a biologist fix a radio?--Or, what I learned while studying apoptosis.}, Cancer Cell., Volume 2, Issue 3, pp 179--82.

\bibitem{Loytynoja} Ari Löytynoja, Nick Goldman, (Jul 2005), \emph{An algorithm for progressive multiple alignment of sequences with insertions}, Proceedings of the National Academy of Sciences, Volume 102, Issue 30, pp 10557--10562.


\bibitem{MacLane}  S. Mac Lane, (1998), \emph{Categories for the working mathematician}, Graduate texts in mathematics, Springer, New York.

\bibitem{MorrisonWhy} D. A. Morrison, (Feb 2009), \emph{Why would phylogeneticists ignore computerized sequence alignment?}, Syst Biol., Volume 58, Issue 1, pp 150--158.

\bibitem{MorrisonFramework} D. A. Morrison, (Oct 2009), \emph{A framework for phylogenetic sequence alignment}, Plant Systematics and Evolution, Volume 282, Issue 3--4, pp 127--149.

\bibitem{Mount} D.W. Mount, (2004), \emph{Bioinformatics: Sequence and Genome Analysis}, Cold Spring Harbor, NY: Cold Spring Harbor Laboratory Press, 692 pp.

\bibitem{NeedlemanWunsch} P. B. Needleman, C. D. Wunsch, (1970), \emph{A general method applicable to the search for similarities in the amino acid sequence of two proteins}, Journal of Molecular Biology, Volume 48, Issue 3, pp 443--453


\bibitem{Pennisi} E. Pennisi, (2013), \emph{The CRISPR Craze}, Science, Volume 341, Issue 6148, pp 833--836.


\bibitem{ReamsRoth} A. B. Reams, R. R. John, (Feb 2015),  \emph{Mechanisms of Gene Duplication and Amplification}, Cold Spring Harbor Perspectives in Biology, Volume 7, Issue 2: a016592.

\bibitem{Rosenberg} M. Rosenberg (Ed.), (2009), \emph{Sequence Alignment: Methods, Models, Concepts, and Strategies}. University of California Press. 


\bibitem{SammethIEEE} M. Sammeth, J. Stoye, (2006), \emph{Comparing Tandem Repeats with Duplications and Excisions of Variable Degree.}, IEEE/ACM Transactions on Computational Biology and Bioinformatics 3.

\bibitem{SammethEDSI} M. Sammeth, T. Weniger, D. Harmsen, J. Stoye, (2005), \emph{Alignment of Tandem Repeats with Excision, Duplication, Substitution and Indels (EDSI)}, Algorithms in Bioinformatics, Volume 3692, pp 276--290.

\bibitem{Japanese_work} J. Sawamura , S. Morishita , J. Ishigooka, (7 May 2014), \emph{A symmetry model for genetic coding via a wallpaper group composed of the traditional four bases and an imaginary base E: towards category theory-like systematization of molecular/genetic biology}, Theoretical Biology and Medical Modelling, Volume 11, Issue 18.

\bibitem{Schoniger} M. Sch\"{o}niger, M. S. Waterman, (1992), \emph{A local algorithm for DNA sequence alignment with inversions}, Bulletin of Mathematical Biology, Volume 54, Issue 4, pp 521-536.

\bibitem{Servedia14} M. R. Servedio, Y. Brandvain, et al., (9 Dec 2014), \emph{Not Just a Theory--The Utility of Mathematical Models in Evolutionary Biology}, PLOS Biology, Volume 12, Issue 12: e1002017

\bibitem{Simossis} V. Simossis, J. Kleinjung , J. Heringa , (Nov 2003), \emph{An overview of multiple sequence alignment}, Current Protocols in Bioinformatics, Volume 3, Issue 1, pp 3.7.1-3.7.26.

\bibitem{SmithWaterman} T. F. Smith, M. S. Waterman, (1981), \emph{Identification of Common Molecular Subsequences}, Journal of Molecular Biology, Volume 147, pp 195--197.

\bibitem{SpivakBook} D. I. Spivak, (2014), \emph{Category theory for the sciences}. MIT Press, Cambridge, MA, viii+486 pp.

\bibitem{Recomb} R\'{e}my Tuy\'{e}ras, (2018), \emph{Category theory for genetics II: genotype, phenotype and haplotype}, \url{https://arxiv.org/abs/1805.07004}


\bibitem{Vellozo} A. F. Vellozo , C. E. R. Alves, A. P. do Lago, (2006), \emph{Alignment with Non-overlapping Inversions in $O(n^3)$-Time.}, In: Bücher P., Moret B.M.E. (eds) Algorithms in Bioinformatics, WABI 2006, Lecture Notes in Computer Science, Volume 4175, Springer, Berlin, Heidelberg.




\end{thebibliography}

\end{document}